 \documentclass[11pt]{article}

\usepackage{amsmath,amssymb,amsthm,latexsym}

\usepackage{epsfig}

\usepackage{color}
\usepackage{comment}
\allowdisplaybreaks

\newtheorem{theorem}{Theorem}[section]
\newtheorem{thm}[theorem]{Theorem}
\newtheorem{corollary}[theorem]{Corollary}

\newtheorem{lemma}[theorem]{Lemma}
\newtheorem{proposition}[theorem]{Proposition}

\newtheorem{definition}[theorem]{Definition}
\newtheorem{remark}[theorem]{Remark}
\newtheorem{example}[theorem]{Example}
\numberwithin{equation}{section}

\def\be{\begin{equation}}
\def\ee{\end{equation}}
\def\bes{\begin{equation*}}
\def\ees{\end{equation*}}

\def\vp{{\varphi}}

\def\wt{\widetilde}

\def\eps{\varepsilon}
\def\lam{{\lambda}}
\def\ol{\overline}

\def\qed{{\hfill $\square$ \bigskip}}

\def\supp{{\mathop {{\rm supp\, }}}}
\def\esssup{{\mathop{\rm ess \; sup \, }}}
\def\essinf{{\mathop{\rm ess \; inf \, }}}

\def\sE {{\cal E}}
\def\sF {{\cal F}}
\def\sL {{\cal L}}
\def\sD{{\cal D}}

\def\sN {{\cal N}}

\def\bE {{\mathbb E}} \def\bN {{\mathbb N}}
\def\bP {{\mathbb P}} \def\bR {{\mathbb R}} \def\bZ {{\mathbb Z}}

\def\sms{\smallskip}

\def\dcap{\mathrm{cap}}
\def\EHR{\mathrm{EHR}}
\def\PHR{\mathrm{PHR}}
\def\PHI{\mathrm{PHI}}
\def\UHK{\mathrm{UHK}}
\def\Gcap{\mathrm{Gcap}}
\def\LHK{\mathrm{LHK}}
\def\HK{\mathrm{HK}}
\def\CSA{\mathrm{CSA}}
\def\CSJ{\mathrm{CSJ}}
\def\CS{\mathrm{CS}}

\def\PI{\mathrm{PI}}

\def\FK{\mathrm{FK}}
\def\NL{\mathrm{NL}}
\def\VD{\mathrm{VD}}
\def\RVD{\mathrm{RVD}}

\def\PHI{\mathrm{PHI}}
\def\EP{\mathrm{EP}}
\def\UHKD{\mathrm{UHKD}}
\definecolor{dred}{rgb}{0.8, 0.0, 0.0}

\def\NDL{\mathrm{NDL}}
\def\UJS{\mathrm{UJS}}

\def\E{\mathrm{E}}
\def\J{\mathrm{J}}
\def\<{\langle}
\def\>{\rangle}
\def\FF{{\cal F}}

\def\T{\mathrm{Tail}}

\begin{document}
\title{\bf Heat kernel estimates
and parabolic Harnack inequalities for symmetric
Dirichlet forms}

\author{{\bf Zhen-Qing Chen}, \quad {\bf Takashi Kumagai} \quad and \quad
{\bf Jian Wang}}

\date{}
\maketitle

\begin{abstract} In this paper,
we consider the following symmetric Dirichlet forms on a metric measure space $(M,d,\mu)$:
$$
\sE(f,g)=\sE^{(c)}(f,g)+\int_{M\times M} (f(x)-f(y))(g(x)-g(y))\,J(dx,dy),
$$
where $\sE^{(c)}$ is a strongly local symmetric bilinear form and $J(dx,dy)$ is a symmetric Random measure on $M\times M$.
Under general volume doubling condition on $(M,d,\mu)$ and some mild assumptions on scaling functions, we establish
stability results for upper bounds of heat kernel (resp.\ two-sided heat kernel estimates)
in terms of the
jumping kernels,
the cut-off Sobolev inequalities,
and the Faber-Krahn inequalities (resp.\ the Poincar\'e inequalities). We also obtain characterizations
of parabolic Harnack inequalities.
Our
results apply to symmetric diffusions with jumps even when
the underlying spaces have walk dimensions larger than $2$.
\end{abstract}



{\small
 \begin{tableofcontents}
 \end{tableofcontents} }

\section{Introduction and main results} \label{sec:intro}

\subsection{Setting and some history}
Let $(M, d, \mu)$ be a \emph{metric measure space}; that is, $(M,d)$
is a locally compact separable metric space, and $\mu$ is a positive
Radon measure on $M$ with full support. We assume that all balls are
relatively compact and assume for simplicity that $\mu (M)=\infty$
throughout the paper.
 Note that we do not assume $M$ to be connected nor
 $(M, d)$ to be geodesic.

Consider a regular \emph{Dirichlet form} $(\sE, \sF)$
on $L^2(M; \mu)$. The Beurling-Deny formula asserts that such a form can be decomposed into the strongly local term, the pure-jump term and the killing term (see \cite[Theorem 4.5.2]{FOT}).
Throughout this paper, we consider the Dirichlet form $(\sE, \sF)$  having both the strongly local term and  the pure-jump term,
and having no killing term.
That is,
\begin{equation}\label{e:1.1}
\begin{split}
\sE(f,g)=& \, \sE^{(c)}(f,g)+\int_{M\times M\setminus \textrm{diag}}(f(x)-f(y)(g(x)-g(y))\,J(dx,dy)\\
=&:\sE^{(c)}(f,g)+\sE^{(j)}(f,g),\qquad f,g\in \sF,\end{split}
\end{equation}
where $(\sE^{(c)},\sF)$ is the strongly local part of  $(\sE, \sF)$
(namely $\sE^{(c)}(f, g)=0$ for all $f, g\in \sF$ having $(f-c)g=0$ $\mu$-a.e.\ on $M$ for some constant $c\in \bR$)
 and  $J(\cdot,\cdot)$ is a symmetric Radon measure  $M\times M\setminus \textrm{diag}$;
see \cite[Theorems 4.3.3 and 4.3.11]{CF}.
Here and in what follows, we will always take a quasi-continuous version when we pick a function on $\sF$ (see \cite[Theorem
2.1.3]{FOT} for the definition and existence of a quasi-continuous
version of the element in $\sF$).
In this paper, we assume that neither $\sE^{(c)}(\cdot,\cdot)$ nor $J(\cdot,\cdot)$ are identically zero.

Let $(\sL, \sD (\sL))$ be the $L^2$-\emph{generator} of $(\sE,\sF)$ on $L^2(M; \mu)$; namely, $\sL$  is
the self-adjoint operator on $L^2(M; \mu)$ and
$\sD (\sL)$ is the domain. $f\in \sD (\sL)$ if $f\in \sF$ and
there exists (unique) $u\in L^2(M; \mu)$ such that
$$
  \sE (f,g) =  - \langle u,g \rangle \quad \hbox{ for all }
  g \in \sF ,
$$
where $\langle  \cdot , \cdot \rangle$ is the
inner product in $L^2(M; \mu)$. We write $\sL f :=  u$.
Let $\{P_t\}_{t\geq 0}$  be
the associated \emph{semigroup}. Given a regular
Dirichlet form $(\sE, \sF)$ on $L^2(M; \mu)$, there is an associated
$\mu$-symmetric
\emph{Hunt process} $X=\{X_t, t \ge 0; \, \bP^x, x \in M\setminus
\sN\}$ where $\sN \subset M$
is a properly exceptional set for $(\sE, \sF)$,
and the Hunt process is unique up to a
properly exceptional set --- see \cite[Theorem 4.2.8]{FOT} for details.
We fix $X$ and $\sN$, and write $M_0 = M \setminus \sN$.
For any bounded Borel measurable function $f$ on $M$, we may set
$$ P_t f(x) = \bE^x f(X_t), \quad x \in M_0. $$

The \emph{heat kernel} associated with the semigroup $\{P_t\}_{t\ge0}$ (if it exists) is a
measurable function
$p(t, x,y): (0,\infty)\times M_0 \times M_0 \to (0,\infty)$
that satisfies the following:
\begin{align} \label{e:hkdef}
\bE^x f(X_t) &= P_tf(x) = \int p(t, x,y) f(y) \, \mu(dy)  \quad \hbox{for all } x \in M_0,
f \in L^\infty(M;\mu), \\
p(t, x,y) &= p(t, y,x) \quad \hbox{for all } t>0,\, x ,y \in M_0, \\
\label{e:ck}
p(s+t , x,z) &= \int p(s, x,y) p(t, y,z) \,\mu(dy)
\quad \hbox{for all } s>0, t>0,
\,\, x,z \in M_0.
\end{align}
It is possible to regularize $p(t, x,y)$ so that \eqref{e:hkdef}--\eqref{e:ck} hold for every point in $M_0$, see \cite[Theorem 3.1]{BBCK} and \cite[Section 2.2]{GT} for details.
 Note that in some arguments of
our paper, we can extend (without further mention) $p(t,x,y)$ to all
$x$, $y\in M$ by setting $p(t,x,y)=0$ if either $x$ or $y$ is outside
$M_0$.

\medskip

There is a long history on the heat kernel estimates and related topics for strongly local Dirichlet forms. Let us briefly mention some of
the previous works which are related to our work.
For diffusions on manifolds, Grigor'yan \cite{Gr1} and Saloff-Coste \cite{Sa1} independently proved that the following are equivalent: (i)  Aronson-type Gaussian bounds for heat kernel,
(ii) parabolic Harnack equality, and (iii) VD and Poincar\'e inequality. The results are later extended to strongly local Dirichlet forms on metric measure spaces in \cite{BM, St1, St2}
and to graphs in \cite{De}.
Detailed heat kernel estimates are heavily related to the control of
harmonic and parabolic functions, and the origin of ideas and techniques used in this field goes back to the work by De Giorgi, Nash, Moser and Aronson.
For more details, see, for example, \cite{gribk,Sa2} and the references therein.
For anomalous diffusions on disordered media such as fractals
(where the so-called walk dimension being larger than 2),
the above equivalence still holds  but one needs to replace (i) by (i')
sub-Gaussian bounds for heat kernel, (iii)
by (iii') VD, Poincar\'e inequality and a cut-off Sobolev inequality; see \cite{AB, BB2, BBK1, GHL3}.

For heat kernel estimates of
symmetric jump processes in general metric measure spaces,
when $\alpha \in (0, 2)$ and
 the metric measure space $M$ is a $d$-set,
 characterizations of $\alpha$-stable-like heat kernel estimates were obtained in \cite{CK1}
which are stable under rough isometries.  This result has later
been extended to mixed stable-like processes in \cite{CK2} under
some growth  condition on the rate function $\phi$ such as
\be\label{eq:into2fb} \int_0^r \frac{s}{\phi(s)} \,ds\le \frac{c\,
r^2}{\phi(r)} \quad \hbox{for all }  r>0 \ee with some constant
$c>0$. For $\alpha$-stable-like processes where $\phi (r)=r^\alpha$,
condition \eqref{eq:into2fb} corresponds exactly to $0<\alpha<2$.
Some of the key methods used in \cite{CK1} were inspired by a
previous work \cite{BL} on random walks on integer lattice $\bZ^d$.
 A long standing open problem in this field is
to find a characterization of heat kernel estimates, which is stable
under rough isometries, for $\alpha$-stable-like processes even with
$\alpha\ge2$ when the underlying spaces have walk dimensions larger
than $2$. This question  has been resolved recently in \cite{CKW1}
under some mild volume growth condition. Actually, in \cite{CKW1} we
obtained stability of two-sided heat kernel estimates and upper
bound heat kernel estimates for symmetric jump processes of mixed
types on general metric measure spaces. There are also recent work
on stable-like jump processes with Ahlfors $d$-set condition in the
framework of metric measure spaces \cite{GHH} and in the framework
of infinite connected locally finite graphs \cite{MS}. The readers
can further refer to \cite{CKW2} for the stability results of
parabolic Harnack inequalities for
symmetric pure jump
Dirichlet forms, and for \cite{CKW3} for various characterizations of
elliptic Harnack inequalities.

\medskip

In this paper, we consider symmetric regular Dirichlet forms that
have
 both the strongly local term and the pure-jump term.
As mentioned above, we can also consider the corresponding operators and Hunt processes (diffusions with jumps).
We use the following example, which is a special case of our much more general results,
to illustrate the novelty and strength of our main results.

\begin{example}\label{Lip-dom} Let $U\subset \bR^d$ be
an unbounded global Lipschitz domain equipped with the Euclidean distance.
Let $X=\{X_t\}_{t\ge0}$ be
a symmetric reflected diffusion with jumps on $\overline U$
associated with the regular
 Dirichlet form  $(\sE, W^{1,2}(U))$
  on $L^2(U; dx)$  given by
 \begin{equation}\label{eq:DFw}
\sE (u,v)=\int_{U} \nabla u(x)\cdot A(x) \nabla
v(x)\,dx+ \int_{U}\int_{U} (u(x)-u(y))(v(x)-v(y))
\frac{c(x,y)}
{|x-y|^{d+\alpha}}
\,dx\,dy, \end{equation}
where $A(x)=(a_{ij}(x))_{1\leq i,
j\leq d}$ is a measurable uniformly elliptic and bounded
 $d\times d$ matrix-valued function on $U$,
$0<\alpha<2$, and $c(\cdot,\cdot)$ is a symmetric measurable
function on $U\times U$ that is bounded between two positive constants.
Its $L^2$-infinitesimal generator  is of the form
\[
\sL  u(x) =
\sum_{i, j=1}^d \frac{\partial}{\partial x_i} \left(a_{ij}(x)
 \frac{\partial u(x)}{\partial x_j}\right)
+ 2\lim_{\varepsilon \to 0} \int_{\{y\in U: \, |y-x|>\varepsilon\}}
 (u(y)-u(x)) \frac{c(x,y)}
{|x-y|^{d+\alpha}} \,dy
\]
with ``Neumann" boundary condition.
Then $\HK (\phi_c, \phi_j)$ and $\PHI (\phi)$ hold,
where $\HK (\phi_c, \phi_j)$ $($reps. $\PHI (\phi)$$)$ is the detailed
two-sided heat kernel estimates defined in \eqref{HKjum} $($resp.
the parabolic Harnack inequality defined in \eqref{e:phidef}$)$
with $\phi_c (r)=r^2$, $\phi_j(r)=r^\alpha$ and $\phi (r)=\phi_c (r) \wedge\phi_j(r)$.
\end{example}

To the best of our knowledge, even this result is
new. When $U=\bR^d$, these results are first obtained in   \cite{CK3}.
See \cite{CKKW2} for a different approach to this example without using the stability results of this paper.
The proof of Example
\ref{Lip-dom} will be given in Section \ref{Sect7}.

 \medskip

Although it is very natural and important to
study heat kernels for symmetric Dirichlet forms that have both the diffusive and jumping parts,
there are very limited work in literature
on this topic.
For a L\'evy process $X$ that is the independent sum of a Brownian motion $W$
and a symmetric
$\alpha$-stable process $Y$ on Euclidean spaces, its transition density is
the convolution of the transition densities of $W$ and $Y$.
In \cite{SV}, heat kernel estimates are derived for $X$ by computing the convolution
in four cases. In one of cases (the case of $|x|^2<t<|x|^\alpha\leq 1$),
the upper and lower bounds do not match.
Nevertheless,  parabolic Harnack
inequality for $X$ can be obtained from these estimates.
In \cite{CK3}, sharp and comparable upper and lower bounds are obtained for
a large class of symmetric diffusions with mixture stable-like jumps
on Euclidean spaces. This sharp two-sided heat kernel estimates were new even for L\'evy processes that are the independent
sum of Brownian motion and symmetric stable processes.
The results of \cite{CK3} have been further extended to general metric measure spaces in \cite{CKKW2} and to the cases where the jumping kernels can have exponential decay.
One of the
difficulties in obtaining fine properties for diffusions with jumps
and associated operators is that it exhibits two different scales: the
strongly local terms part has a diffusion scaling $r\mapsto
\phi_c(r)$ while the pure jump part has a different type of scaling
$r\mapsto \phi_j(r)$. On the other hand, as shown in \cite{CKW2}, in
contrast to the cases of local operators/diffusions, for symmetric
pure-jump processes, parabolic Harnack inequalities
are no
longer equivalent to (in fact weaker than) the two-sided heat kernel
estimates. This discrepancy is caused by the heavy tail of the
jumping kernel. This heavy tail phenomenon is also one of main
sources of difficulties in analyzing non-local operators/jump
processes.
Diffusions with jumps are even more
complex than pure jumps case studied in
\cite{CKW2}, in fact Theorem \ref{C:1.25} of this paper asserts
that, with $\phi (r):=\phi_c (r) \wedge \phi_j(r)$,
$$
\HK_- (\phi_c, \phi_j)\Longleftrightarrow\PHI(\phi)+ \J_{\phi_j};
$$
see Definition \ref{D:1.11}(ii) and \eqref{jsigm} for definitions of $\HK_- (\phi_c, \phi_j)$ and $\J_{\phi_j}$, respectively.
In the pure jump case, it holds that
 $$
 \HK (\phi_j)\Longleftrightarrow\PHI(\phi_j)+ \J_{\phi_j,\ge}
$$
 see \cite[Corollary 1.21]{CKW2}.
Intuitively speaking, this discrepancy is due to the fact that
the scale corresponding to the small time behavior of diffusions with jumps is dominated by the
diffusive part and hence one can not recover information about the jumping scale function $\phi_j (r)$
for $r\leq 1$ from $\PHI (\phi)$.

\medskip

Due to the above difficulties and differences, obtaining the complete picture of heat kernel estimates, and the stability of heat kernel estimates and parabolic Harnack inequalities for symmetric Dirichlet forms including both local and non-local terms/diffusions with jumps requires new ideas. Our approach contains the following
four key
ingredients, and all of them are highly non-trivial:

\begin{itemize}
\item[(i)] We adopt
the generalized capacity inequality formulation
from the recent study on
stable-like jumps processes
under the
Ahlfors $d$-set condition in the
framework of metric measure spaces \cite{GHH}, and
make use of the arguments depending on cut-off Sobolev inequality from \cite{AB,CKW1} to
derive some useful analytical
properties of Dirichlet forms consisting of both strongly local terms and pure jumps terms.
The generalized capacity condition $\Gcap (\phi)$ is clean to state
but it is not known whether it is stable under rough isometry or not,
while the cut-off Sobolev inequality condition $\CS (\phi)$ is lengthy to state but
is clearly
 stable under rough isometry.

\item[(ii)] We find a new self-improving argument for upper bounds for diffusions with jumps. The advantage of this technique is that it not only can take care of
different scales both from the strongly local term part and the
pure jump term, but also can treat the case that the volume of balls is not uniformly comparable.

\item[(iii)] As mentioned above, different from the assertions for diffusions or
symmetric pure jump processes,
in the present setting parabolic Harnack inequalities are not equivalent to two-sided heat kernel estimates.
Moreover,  the parabolic Harnack inequalities alone can not imply bounds (even upper bounds) of the jumping kernel. So, to obtain  the characterizations of parabolic Harnack inequalities for diffusions with jumps we shall consider the weaker upper bounds of jumping kernels $\J_{\phi,\le}$ instead of the exact upper bounds of jumping kernels $\J_{\phi_j,\le}$. This indicates that
only rough upper bounds of heat kernels $\UHK_{weak}$
together with $\NDL (\phi)$ and $\UJS$ (see \eqref{heat:remark1}, Definitions \ref{D:1.11}(vi) and \ref{thm:defUJS}
 for their definitions) is involved in the characterization of parabolic Harnack inequalities;
see Theorem \ref{C:1.25}.

\item[(iv)] Our results are obtained under some mild volume growth condition, called volume doubling and reverse volume doubling conditions; see Definition \ref{D:1.1} for details. This is much weaker
than the Ahlfors $d$-set condition and significant technical difficulties arise when working
under
 this setting.

\end{itemize}

\subsection{Functional inequalities and heat kernel estimates}

In this paper, we are concerned with
stable characterizations of
 both upper bounds and two-sided estimates on heat kernel, as well as of parabolic Harnack inequalities,
 for symmetric Dirichlet forms having both local and non-local terms on general metric measure spaces.
To state our results for heat kernel estimates precisely, we need a number of definitions; some of them are taken from \cite{CKW1}.
Denote the ball centered at $x$ with radius $r$ by
$B(x, r)$ and $\mu (B(x, r))$ by $V(x, r)$.

\begin{definition} \label{D:1.1}\rm
\noindent {\rm
(i) We say that $(M,d,\mu)$ satisfies the {\it volume doubling property} ($\VD$), if
there exists a constant ${C_\mu}\ge 1$ such that
for all $x \in M$ and $r > 0$,
\be \label{e:vd}
 V(x,2r) \le {C_\mu} V(x,r).
 \ee
(ii) We say that $(M,d,\mu)$ satisfies the {\it reverse volume doubling property} ($\RVD$),
 if there exist constants ${l_\mu}, {c_\mu}>1$ such that
for all $x \in M$ and $r > 0$,
\be \label{e:rvd2}
 V(x,{l_\mu}  r) \geq  {c_\mu} V(x,r).
 \ee }\end{definition}

VD condition \eqref{e:vd} (resp. RVD condition \eqref{e:rvd2})
is equivalent to the second (resp. the first) inequality in the following  display:
\be \label{e:vd2}
 {\wt c_\mu}  \Big(\frac Rr\Big)^{d_1}\le \frac{V(x,R)}{V(x,r)} \leq   {\wt C_\mu}  \Big(\frac Rr\Big)^{d_2} \quad \hbox{for all } x\in M \hbox{ and } R\geq r>0,
\ee
where $d_1,d_2, {\wt c_\mu}$ and ${\wt C_\mu}$  are positive constants.
Under RVD, $\mu (M)=\infty$ if and only if $M$ has infinite diameter.
If $M$ is connected and unbounded, then $\VD$ implies $\RVD$;
see \cite[Proposition 5.1 and Corollary 5.3]{GH}.

\sms

Let $\bR_+:=[0,\infty)$, and $\phi_c: \bR_+\to \bR_+$ (resp. $\phi_j: \bR_+\to \bR_+$)  be a strictly increasing continuous
function  with $\phi_c (0)=0$ (resp. $\phi_j(0)=0$),
$\phi_c(1)=1$ (resp. $\phi_j(1)=1$)
and satisfying that there exist constants $c_{1,\phi_c},c_{2,\phi_c}>0$ and $\beta_{2,\phi_c}\ge \beta_{1,\phi_c}>1$ (resp. $c_{1,\phi_j},c_{2,\phi_j}>0$ and $\beta_{2,\phi_j}\ge \beta_{1,\phi_j}>0$)
such that
\begin{equation}\label{polycon}
\begin{split}
 &c_{1,\phi_c} \Big(\frac Rr\Big)^{\beta_{1,\phi_c}} \leq
\frac{\phi_c (R)}{\phi_c (r)}  \ \leq \ c_{2,\phi_c} \Big(\frac
Rr\Big)^{\beta_{2,\phi_c}}
\quad \hbox{for all }
0<r \le R. \\
\bigg(\hbox{resp. }  &c_{1,\phi_j} \Big(\frac Rr\Big)^{\beta_{1,\phi_j}} \leq
\frac{\phi_j (R)}{\phi_j (r)}  \ \leq \ c_{2,\phi_j} \Big(\frac
Rr\Big)^{\beta_{2,\phi_j}}
\quad \hbox{for all }
0<r \le R.\bigg)\end{split}\end{equation}
Note that (\ref{polycon}) is
equivalent to the existence of constants
$ c_{3,\phi_c}, {l_{0,\phi_c}}>1$
such that
$$   c_{3,\phi_c}^{-1}\phi_c (r) \leq \phi_c ({l_{0,\phi_c}}r)
\leq c_{3,\phi_c} \, \phi_c (r)\quad \hbox{for all } r>0,
$$ the same as $\phi_j(r)$.
Throughout the paper, we assume that
 \begin{equation}\label{e:1.11}
 \phi_c(r)\le \phi_j(r)   \hbox{ for } r\in (0,1] \quad \hbox{ and } \quad \phi_c(r)\ge \phi_j(r) \hbox{  for } r\in [1,\infty).
 \end{equation}
 Since $\beta_{1,\phi_c}>1$, by  \cite[Definition, p.\ 65; Definition, p.\ 66; Theorem 2.2.4 and its remark, p.\ 73]{BGT}, there exists a strictly increasing
 continuous  function $\bar \phi_c (r): \bR_+\to \bR_+$
  such that there are  constants $c_2\geq c_1>0$ so that
 \begin{equation}\label{e:1.12}
  c_1 \frac{ \phi_c(r)}{r} \le \bar \phi_c (r)\le c_2 \frac{\phi_c(r)}{r}  \quad \hbox{for all } r>0.
  \end{equation}
 By   \eqref{polycon},
$\beta_{1,\phi_c}>1$ and \eqref{e:1.12},
  we have $\lim_{r\to 0} \bar \phi_c (r)=0$,  $\lim_{r\to \infty} \bar \phi_c (r)=\infty$, and
  there are constants $c_{1,\bar \phi_c}, c_{2,\bar \phi_c}>0$ such that
\begin{equation}\label{e:effdiff}c_{1,\bar\phi_c} \Big(\frac Rr\Big)^{\beta_{1,\phi_c}-1} \leq
\frac{\bar\phi_c (R)}{\bar\phi_c (r)}  \ \leq \ c_{2,\bar\phi_c} \Big(\frac
Rr\Big)^{\beta_{2,\phi_c}-1}
\quad \hbox{for all } 0<r \le R.
\end{equation}

Given $\phi_c$ and $\phi_j$ satisfying \eqref{e:1.11}, we set
\begin{equation}\label{e:scaling-}
\phi(r):=\phi_c(r)\wedge \phi_j(r)
=\begin{cases} \phi_c(r),\quad& r\in (0, 1], \\
\phi_j(r),\quad &r\in [1, \infty).
\end{cases}
\end{equation}
It is clear that $\phi(r)$ is a strictly increasing function on $\bR_+$ with $\phi(0)=0$ and $\phi(1)=1$, and satisfies that there
exist constants $c_{1,\phi},c_{2,\phi}>0$ so that
\begin{equation}\label{polycon000}c_{1,\phi} \Big(\frac Rr\Big)^{\beta_{1,\phi}} \leq
\frac{\phi (R)}{\phi (r)}  \ \leq \ c_{2,\phi} \Big(\frac
Rr\Big)^{\beta_{2,\phi}}
\quad \hbox{for all }
0<r \le R,\end{equation}  where $\beta_{1,\phi}=\beta_{1,\phi_c}\wedge \beta_{1,\phi_j}$ and $\beta_{2,\phi}=\beta_{2,\phi_c}\vee \beta_{2,\phi_j}$.
Throughout this paper, without any mention we will fix the notations for these three functions $\phi_c$, $\phi_j$ and $\phi$.

\begin{definition}{\rm Let $\psi:\bR_+\to \bR_+$.  We say condition $\J_{\psi}$ holds if there exists a non-negative symmetric
function $J(x, y)$ on $M\times M$ so that for $\mu\times \mu $-almost  all $x, y \in M$,
\begin{equation}\label{e:1.2}
J(dx,dy)=J(x, y)\,\mu(dx)\, \mu (dy),
\end{equation} and
\begin{equation}\label{jsigm}
 \frac{c_1}{V(x,d(x, y)) \psi (d(x, y))}\le J(x, y) \le \frac{c_2}{V(x,d(x, y)) \psi (d(x, y))}.
 \end{equation}
We say that $\J_{\psi,\le}$ (resp. $\J_{\psi,\ge}$) if \eqref{e:1.2} holds and the upper bound (resp. lower bound) in \eqref{jsigm} holds.}
\end{definition}

Note that, without loss of generality,
we may and do assume that in condition $\J_{\psi}$ ($\J_{\psi, \geq}$ and $\J_{\psi, \leq}$, respectively)
that \eqref{jsigm} (and the corresponding inequality) holds for every $x, y \in M$. Note also that, under $\VD$,
 the bounds in condition \eqref{jsigm}
are consistent with the symmetry of $J(x, y)$. See \cite[Remark 1.3]{CKW1} for more details.

Since $\phi(r)\le \phi_j(r)$ for all $r>0$,  $\J_{\phi_j,\le}$ implies $\J_{\phi,\le}$; that is,
condition $\J_{\phi,\le}$ is weaker than condition $\J_{\phi_j,\le}$. We will frequently use this fact
in the paper.

\begin{definition} \rm
Let $U \subset V$ be open sets of $M$ with
$U \subset \ol U \subset V$, and  $\kappa\ge 1$.
We say a non-negative bounded measurable function $\vp$ is a {\it $\kappa$-cut-off function for $U \subset V$},
if $\vp \ge 1$ on $U$,  $\vp=0$ on $V^c$ and $0\leq \vp \leq \kappa$ on $M$. Any 1-cut-off function is simply referred to as a {\it cut-off function}.
\end{definition}

It is obvious that for any $\kappa$-cut-off function $\vp$ for $U \subset V$, $1\wedge\vp$ is a cut-off function  for $U \subset V$.

Motivated by \cite{GHH} for
pure jump Dirichlet forms, we formulate the generalized capacity condition for non-local Dirichlet forms that have diffusive parts. For this, we consider the following function space
$$
 \sF'_b:=\{u+a: u\in \sF_b,\,  a\in \bR\},
$$
where $\sF_b:=\sF \cap L^\infty (M; \mu)$.

\begin{definition}\rm  We say that the {\it generalized capacity inequality} $\Gcap(\phi)$ holds, if there exist constants $\kappa\ge1$ and $C>0$ such that for every $0<r<R$, any $f\in \sF'_b$  and for almost all
$x_0\in M$, there is a $\kappa$-cut-off function $\vp\in \sF_b$ for $B(x_0,R)\subset B(x_0,R+r)$ so that
$$
\sE(f^2\vp,\vp)\le \frac{C}{\phi(r)} \int_{B(x_0,R+r)}f^2\,d\mu.
$$
\end{definition}

Recall that for any subsets $A\subset B$,
the relative capacity
$\dcap(A,B)$ is defined by
$$\dcap(A,B)=\inf\{\sE(\vp,\vp):
\vp \in \sF_b
\hbox{ is a cut-off function
for } A\subset B\}.
$$
Following \cite[Definition 1.7]{GHH}, for any
subsets $A\subset B$,
$f\in \sF'_b$ and a constant
$\kappa\ge1$,
we define the generalized relative capacity $\dcap_f^{(\kappa)}(A,B)$ by
$$
\dcap_f^{(\kappa)}(A,B)=\inf \left\{\sE(f^2\vp,\vp):
\vp \in \sF_b
\hbox{ is a }\kappa\hbox{-cut-off function
for } A\subset B \right\}.
$$
In particular, when $f=1$ and $\kappa=1$,
$\dcap_f^{(\kappa)}(A,B)=\dcap(A,B).$ As mentioned in \cite[Remarks
1.8 and 1.9]{GHH}, the quantity $\sE(f^2\vp,\vp)$ in the definition
of the generalized capacity is well defined,
and the generalized capacity can take negative values.
With this notation, the generalized
capacity inequality $\Gcap(\phi)$ is equivalent to that there
exist constants $\kappa\ge1$ and $C>0$ such that for every $0<r<R$, any $u\in \sF'_b$ and almost all $x_0\in M$ so that
$$\dcap_u^{(\kappa)}(B(x_0,R), B(x_0,R+r))\le \frac{C}{\phi(r)} \int_{B(x_0,R+r)}u^2\,d\mu.$$

Denote by $C_c(M)$ the space of continuous functions on $M$ with compact support.
It is well known that for any
$f\in \sF_b$, there exist unique positive Random measures $\Gamma(f,f)$ and $\Gamma_c(f, f)$ on $M$ so that
for every $g\in \sF\cap C_c(M)$,
$$
\int_M g \, d\Gamma(f,f)=\sE(f, fg)-\frac 12\sE(f^2,g),
$$
and
$$
\int_M g \, d\Gamma_c(f,f)=\sE^{(c)}(f, fg)-\frac 12\sE^{(c)}(f^2,g).
$$
The energy measures $\Gamma(f,f)$  and $\Gamma_c (f, f)$ can be uniquely extended to any $f\in \sF$
as the increasing limit of $\Gamma (f_n, f_n)$ and $\Gamma_c (f_n, f_n)$, respectively, where
$f_n:=((-n)\vee f)\wedge n$.
The measure $\Gamma(f,f)$ (resp. $\Gamma_c (f, f)$) is called the \emph{energy measure} of $f$ (which is also called the
 {\it carr\'e du champ} in the literature) for $\sE$ (resp. its strongly local part $\sE^{(c)}$).

To make use  of
the generalized capacity inequality
$\Gcap(\phi)$, we need to introduce
a version of a cut-off Sobolev inequality
that controls
the energy of cut-off functions.

\begin{definition} \rm We say that condition $\CS(\phi)$ holds if there exist constants $C_0\in (0,1]$ and $C_1, C_2>0$
such that for every
$0<r\le R$, almost all $x_0\in M$ and any $f\in \sF$, there exists
a cut-off function $\vp\in \sF_b$ for $B(x_0,R) \subset B(x_0,R+r)$ so that the following holds:
\be \label{e:csj1} \begin{split}
 &\int_{B(x_0,R+(1+C_0)r)} f^2 \, d\Gamma (\vp,\vp)\\
&\le C_1 \bigg(\int_{B(x_0,R+r)} \vp^2\,d\Gamma_c(f,f)\\
&\qquad\quad+\int_{B(x_0,R+r)\times B(x_0,R+(1+C_0)r)}\vp^2(x) (f(x)-f(y))^2\,J(dx,dy)\bigg) \\
&\quad + \frac{C_2}{\phi(r)}  \int_{B(x_0,R+(1+C_0)r)} f^2  \,d\mu.
\end{split}
\ee
\end{definition}

\begin{remark}\label{rek:scj}\rm
\begin{itemize}
\item[(i)]
Clearly, unlike $\Gcap (\phi)$,
condition $\CS (\phi)$ is stable under rough isometry.
$\CS(\phi)$ is a combination of $\CSA(\phi)$ for strongly local Dirichlet
forms and $\CSJ(\phi)$ for
pure jump
Dirichlet forms.  $\CSA(\phi)$ was introduced in \cite{AB} for strongly local Dirichlet
forms as a weaker version of the so called cut-off Sobolev
inequality in \cite{BB2,BBK1}; while $\CSJ(\phi)$, as a
counterpart
of $\CSA(\phi)$ for
pure jump Dirichlet form, was given in \cite{CKW1}. As pointed out in \cite[Remark 1.6(ii)]{CKW1}, the main difference
between $\CSJ(\phi)$ and $\CSA(\phi)$ is that the
integrals in the left hand side and in the second term of the right
hand side of the inequality \eqref{e:csj1} are over $B(x_0,R+(1+C_0)r)$ instead of over $B(x_0,R+r)$ for \cite{AB}. Note that the integral
over $B(x_0,R+r)^c$ is zero in the left hand side of \eqref{e:csj1} for the
case of strongly local Dirichlet forms. As we see in \cite{CKW1} for the arguments of the stability of heat kernel estimates for jump processes, it is important to
enlarge
the ball $B(x_0,R+r)$ and integrate over
$B(x_0,R+(1+C_0)r)$ rather than over $B(x_0,R+r)$. In the present setting, we will deal with Dirichlet forms
having both local and non-local
parts
(i.e., the associated Hunt process have both the diffusive and jumping parts),
and so it is natural to use the formula similar to $\CSJ(\phi)$
of \cite{CKW1}. As $\supp [\varphi]\subset B(x_0, R+r)$,
we could  replace $B(x_0, R+r)$   by $B(x_0, R+(1+C_0)r)$
in the integral region of the first term on the right hand side of
\eqref{e:csj1}.

 \item[(ii)] Denote by $\sF_{loc}$ the space of functions
locally in $\sF$; that is, $f\in \sF_{loc}$
if and only if for any relatively compact open set $U\subset M$ there exists $g\in \sF$ such that $f=g$ $\mu$-a.e.\ on $U$.
 Since each ball is relatively compact and
 \eqref{e:csj1} uses the property of $f$ on $B(x_0,R+(1+C_0)r)$ only, $\CS(\phi)$ also holds for any $f\in \sF_{loc}.$

 \item[(iii)] As mentioned in \cite[page 1492]{GHL3} and \cite[Remark 1.7]{CKW1}, if the non-negative cut-off function $\vp$ in $\CS(\phi)$ can be chosen as a Lipschitz continuous  function, then $\CS(\phi)$ always holds under  $\VD$, \eqref{polycon} and $\J_{\phi,\le}$. For example, this is the case that any geodesically complete Riemannian manifold with $\phi_c(r)=r^2$ and $\phi_j(r)=r^\alpha$ with some $\alpha\in (0,2)$.  See also Example \ref{Lip-dom}.

\end{itemize}
\end{remark}

For $\alpha >0$,  define
$$
\sE_\alpha (f, g)=\sE(f, g)+\alpha \int_M f(x)g(x)\,\mu (dx) \qquad \hbox{for }
f, g\in \sF.
$$
We next introduce the Faber-Krahn inequality.
For any open set $D \subset M$, let $\sF^D$ be the
$\sE_1$-closure of  $\sF\cap C_c(D)$ in $\sF$.
Define
\be \label{e:lam1}
 \lam_1(D)
= \inf \left\{ \sE(f,f):  \,  f \in \sF^D \hbox{ with }  \|f\|_2 =1 \right\},
\ee
the bottom of the Dirichlet spectrum of the corresponding self-adjoint operator in $D$.

\begin{definition}\label{Faber-Krahnww}
{\rm We say that the {\em Faber-Krahn
inequality} $\FK(\phi)$ holds if there exist positive constants $C$ and
$\nu$ such that for any ball $B(x,r)$ and any open set $D \subset
B(x,r)$, \be \label{e:fki}
 \lam_1 (D) \ge \frac{C}{\phi(r)} (V(x,r)/\mu(D))^{\nu}.
\ee
} \end{definition}

Since $V(x,r)\ge \mu(D)$
for $D\subset B(x, r)$,
 if \eqref{e:fki} holds for some $\nu=\nu_0>0$, it then holds
for every $\nu \in (0, \nu_0)$.  So without loss of generality, we may and do assume $0<\nu<1$.

\begin{definition} {\rm
We say that
the {\em $($weak$)$ Poincar\'e inequality}
$\PI(\phi)$
holds if there exist constants $C>0 $ and $\kappa\ge1$ such that
for any  ball $B_r=B(x,r)$ with $x\in M$  and for any $f \in \sF_b$,
\begin{equation}\label{eq:PIn}
\int_{B_r} (f-\ol f_{B_r})^2\, d\mu \le C \phi(r)\left(\int_{B_{\kappa r}}\,\Gamma_c(f,f)+ \int_{{B_{\kappa r}}\times {B_{\kappa r}}} (f(y)-f(x))^2\,J(dx,dy)\right),
\end{equation}
where $\Gamma_c$ is the energy measure of local bilinear form of $(\sE,\sF)$,
and $\ol f_{B_r}= \frac{1}{\mu({B_r})}\int_{B_r} f\,d\mu$ is the average value of $f$ on ${B_r}$.} \end{definition}
 If the integral on the right hand side of \eqref{eq:PIn} is over ${B_{r}}\times {B_r}$ (i.e. $\kappa=1$), then it is called strong Poincar\'e inequality.
If the metric is geodesic, it is known that
(weak) Poincar\'e inequality implies
strong Poincar\'e inequality (see for instance \cite[Section 5.3]{Sa1}), but in general they are not the same. In this paper, we only use weak Poincar\'e inequality.
Note also that the left hand side of \eqref{eq:PIn} is
equal to $\inf_{a \in \bR} \int_{B_r} (f-a)^2\, d\mu$.

Recall that $X=\{X_t\}_{t\ge0}$ is the Hunt process associated with
the regular Dirichlet form $(\sE,\sF)$ on $L^2(M; \mu)$ with
properly exceptional set $\sN$, and $M_0:= M\setminus \sN$. For a
set $A\subset M$, define the exit time $\tau_A = \inf\{ t >0 : X_t
\notin A\}.$

\begin{definition}{\rm (i) We say that $\E_\phi$ holds if
there is a constant $c_1>1$ such that for all $r>0$ and  all $x\in M_0$,
$$
c_1^{-1}\phi(r)\le \bE^x [ \tau_{B(x,r)} ] \le c_1\phi(r).
$$
We say that $\E_{\phi,\le}$ (resp. $\E_{\phi,\ge}$) holds
if the upper bound (resp. lower bound) in the inequality above holds.

(ii) We say $\EP_{\phi,\le}$
holds if there is a constant $c>0$ such that for all $r,t>0$ and  all $x\in M_0$,
 $$
 \bP^x(  \tau_{B(x,r)} \le t ) \le \frac{ct}{\phi(r)}.
 $$
We say  $\EP_{\phi,\le, \varepsilon}$ holds,
 if there exist constants $\varepsilon, \delta\in (0,1)$ such that for
any $x_0\in M$ and $r>0$,
$$
\bP^x (  \tau_{B(x_0,r)} \le \delta \phi(r) ) \le \eps
\quad \hbox{for all } x\in  B(x_0, r/4)  \cap M_0 .
$$
}\
\end{definition}

 It is clear that $\EP_{\phi,\le}$ implies $\EP_{\phi,\le,\eps}$.
It is also easy to see that, under \eqref{polycon},
$\E_\phi$ implies $\EP_{\phi,\le, \eps}$; see Proposition \ref{P:2.4} below.

\medskip

We use $\phi_c^{-1}(t)$ (resp. $\phi_j^{-1}(t)$) to denote the inverse function of the strictly increasing function
$t\mapsto \phi_c (t)$ (resp. $t\mapsto \phi_j(t)$).
Throughout the paper,
we write
$f(s, x)\simeq g(s, x)$,
if there exist constants $c_{1},c_{2}>0$ such that
$
c_{1}g(s, x)\leq f(s, x)\leq c_{2}g(s, x)
$
for the specified range of the argument $(s, x)$. Similarly, we write
$f(s, x)\asymp g(s, x)$,
if there exist constants $c_k>0$, $k=1, \cdots, 4$, such that
$
c_1g(c_2 s, x)\leq f(s, x)\leq c_3 f(c_4 s, x)
$
for the specified range of $(s, x)$.

\medskip

We consider
the following two-sided estimates of heat kernel for the local Dirichlet forms.
Define
\begin{equation}\label{eq:fibie2}
p^{(c)}(t,x,y):=\frac{1}{V(x,\phi_c^{-1}(t))} \exp\left(- \sup_{s>0}\left\{\frac{d(x,y)}{s}-\frac{t}{\phi_c(s)}\right\}\right),\quad t>0,x,y\in M_0.
\end{equation}
This kernel arises in the two-sided estimates of heat kernel for strongly local Dirichlet forms;
see e.g.\ \cite{AB}. In the literature (see \cite{GT,HK}), there is another expression of two-sided heat kernel estimates for the local Dirichlet forms
given by
\begin{equation}\label{eq:fibie3}
p^{(c)}(t,x,y)= \frac{1}{V(x, \phi_c^{-1}(t))}
\exp\left(-\frac{d(x, y)}{\bar \phi_c^{-1}(t/d(x, y))} \right),
 \quad t>0, x,
y\in M_0.
\end{equation}
Note that
\begin{equation}\label{e:1.22}
m(t, r) := \frac{r}{\bar \phi_c^{-1}(t/r)}
\end{equation}
is the unique solution of
\begin{equation}\label{e:scdf}
 \bar \phi_c \left(\frac{r}{m(t,r)}\right)=\frac{t}{r},
\quad t,r>0.
\end{equation}
We will
show
that \eqref{eq:fibie2} and \eqref{eq:fibie3} are
equivalent to each other in our setting in the sense that there are constants $c_k>0$, $k=1, \cdots, 4$, so that
for $p^{(c)}(t, x, y)$ given by \eqref{eq:fibie2},
\begin{equation}\label{e:1.20}\begin{split}
\frac{c_1}{ V(x,\phi_c^{-1}(t))} \exp\left(- \frac{c_2 d(x,y)}{\bar \phi_c^{-1}(t/d(x,y))}\right)&
\leq
p^{(c)}(t,x,y)\\
& \leq \, \frac{c_3}{V(x,\phi_c^{-1}(t))} \exp\left(- \frac{c_4 d(x,y)}{\bar \phi_c^{-1}(t/d(x,y))}\right)
\end{split}\end{equation}
for every $t>0$ and $x,y\in M_0$.
See Lemma \ref{L:diff} and Corollary \ref{C:diff} for the proofs.
On the other hand, we note that, in all the literature we know,
for example
\cite{BB1}, \cite{HK} and \cite[Page 1217--1218]{GT},
the lower bound in the estimate \eqref{eq:fibie3} (more explicitly, the lower bound of off-diagonal estimate in \eqref{eq:fibie3})
was established under assumptions that include $(M,d,\mu ) $ being connected and satisfying the chain condition; that is,
there exists a constant $C>0$ such that, for any $x,y\in
M$ and for any $n\in {\mathbb N}$, there exists a sequence
$\{x_{i}\}_{i=0}^{n}\subset M$ such that $x_{0}=x$, $x_{n}=y$ and
$
d(x_{i},x_{i+1})\leq C { d(x,y)}/{n}$ for all $i=0,1,\cdots,n-1$.

In the following, we designate  \eqref {eq:fibie3} as the expression of $p^{(c)}(t, x, y)$.
Note that
$$\frac{ d(x,y)}{\bar \phi_c^{-1}(t/d(x,y))} \preceq 1
\quad \hbox{if and only if} \quad  \phi_c(d(x, y))/t\preceq 1.
$$
So for each fixed $a>0$,
\begin{equation}\label{e:1.25}
p^{(c)}(t, x, y)
\simeq 1/V(x, \phi_c^{-1}(t))
\quad \hbox{when } d(x, y)\leq a \phi_c^{-1}(t).
\end{equation} Set
\begin{equation}\label{e:1.26}
p^{(j)}(t, x, y):= \frac{1}{V(x, \phi_j^{-1}(t))}\wedge \frac{t}{V(x, d(x, y)) \phi_j (d(x, y))}.
\end{equation}
It is easy to see that for each fixed $a>0$,
\begin{equation}\label{e:1.27}
p^{(j)}(t, x, y)
\simeq \frac{1}{V(x, \phi_j^{-1}(t))}
\quad \hbox{when } d(x, y) \leq a \phi^{-1}_j (t).
\end{equation}

\begin{definition}\label{D:1.11}   \rm \begin{description}
\item{\rm (i)} We say that
$\HK(\phi_c, \phi_j)$
holds if there exists a heat kernel $p(t, x,y)$
of the semigroup $\{P_t\}_{t\ge0}$ associated with $(\sE,\sF)$
and the following estimates hold
for all $t>0$ and all $x,y\in M_0$,
\begin{equation}\label{HKjum}\begin{split}
&   c_1\Big(\frac 1{V(x,\phi_c^{-1}(t))}\wedge \frac 1{V(x,\phi_j^{-1}(t))} \wedge
 \big(p^{(c)}(c_2 t,x,y)+p^{(j)}(t,x,y)\big)\Big) \\
 & \le \   p(t, x,y)  \\
  &\le c_3\Big(\frac 1{V(x,\phi_c^{-1}(t))}\wedge \frac 1{V(x,\phi_j^{-1}(t))} \wedge
\big(p^{(c)}(c_4 t,x,y)+p^{(j)}(t,x,y)\big)\Big),
\end{split}
\end{equation}
where $c_k>0$, $k=1, \cdots, 4$,  are constants independent of $x,y\in M_0$ and $t>0$.
For simplicity and by abusing the notation, we abbreviate the two-sided estimate \eqref{HKjum} as
$$
p(t, x, y) \asymp \frac 1{V(x,\phi_c^{-1}(t))}\wedge \frac 1{V(x,\phi_j^{-1}(t))} \wedge
\left( p^{(c)}(t,x,y)+p^{(j)}(t,x,y) \right).
$$

\item{(ii)} We say
$\HK_- (\phi_c, \phi_j)$
holds if the upper bound in \eqref{HKjum} holds but the lower bound is replaced by
the following: there are constants $c_0, c_1>0$ so that for all $x,y\in M_0$,
\begin{equation}\label{e:1.31}\begin{split}
p(t, x, y) \geq c_0\bigg( &\frac{1}{V(x,\phi^{-1}(t))}{\bf 1}_{\{d(x, y) \leq c_1 \phi^{-1}(t)\}}\\
&+\frac{t}{V(x,d(x,y))\phi_j(d(x,y))}{\bf 1}_{\{d(x, y) >c_1 \phi^{-1}(t)\}} \bigg).\end{split}
\end{equation}

\item{(iii)} We say $\UHK (\phi_c, \phi_j)$ (resp. $\LHK (\phi_c, \phi_j)$) holds if the
upper bound (resp. the lower bound) in \eqref{HKjum} holds.

\item{(iv)}  We say $\UHKD(\phi)$ holds if there is a constant $c>0$ such that for all $t>0$ and all $x\in M_0$,
$$p(t, x,x)
\leq  \frac{c}{V(x, \phi^{-1}(t))} .
$$

\item{(v)}
We say a near-diagonal lower bound heat kernel estimate
$\NL(\phi)$ holds if there are constants $c_1,c_2>0$ such that for all $t>0$ and all $x,y\in M_0$ with $d(x,y)\le c_1\phi^{-1}(t)$,
$$p(t, x,y)\ge  \frac{c_2}{V(x,\phi^{-1}(t))}.$$

\item{(vi)}
Denote by $(P_t^D)_{t\ge0}$ the (Dirichlet) semigroups of $(\sE,\sF^D)$, and by $p^D(t,x,y)$ the corresponding (Dirichlet) heat kernel.
We say that \emph{a
near diagonal lower bounded estimate for Dirichlet heat kernel} $\NDL(\phi)$ holds, if there exist $\eps\in (0,1)$ and $c_1>0$ such that for any $x_0\in M$, $r>0$,
$0<t\le \phi(\eps r)$ and $B=B(x_0,r)$,
\be\label{NDLdf1} p^{B}(t, x
,y )\ge \frac{c_1}{V(x_0, \phi^{-1}(t))},\quad x ,y\in B(x_0,
\eps\phi^{-1}(t)) \cap M_0.
\ee
\end{description}
\end{definition}

\begin{figure}[t]
\centerline{\epsfig{file=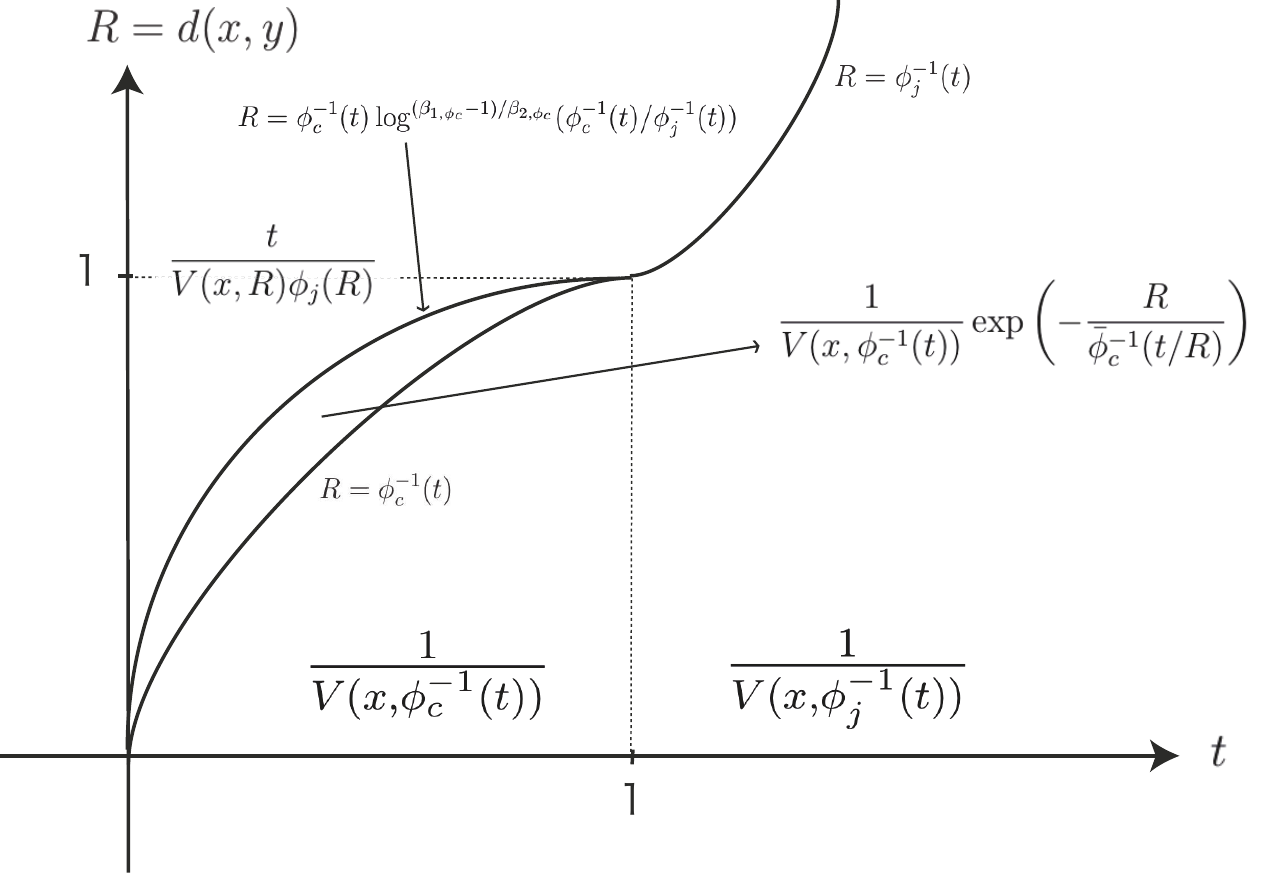, height=2in}}
\caption{Dominant term in the heat kernel estimates
$\HK (\phi_c, \phi_j)$ for  $p(t,x,y)$}\label{domfig}
\end{figure}

\begin{remark}\label{R:1.22} \rm
We have
five
remarks about this definition.

\begin{itemize}
\item[(i)] Note that the scaling of $\HK (\phi_c, \phi_j)$ (also $\HK_-(\phi_c, \phi_j)$, $\UHK (\phi_c, \phi_j)$ and $\LHK(\phi_c,\phi_j)$) is
NOT completely
determined by $\phi$. Indeed, it also includes the information of
$\phi_j(r)$ for $0<r\le 1$. Yet, we use this notation since $\phi$ gives us the space-time relation of the heat kernel estimates.
Furthermore, it follows from Theorem \ref{C:1.25} that $\HK_-(\phi_c, \phi_j)$
(and so $\HK(\phi_c, \phi_j)$) is  stronger than $\PHI(\phi)$, and, consequently,  $\PHR(\phi)$;
see Definition \ref{def:PHI}(ii) and Definition \ref{PER}(i) for precise definitions. In particular, this implies that $\HK_-(\phi_c, \phi_j)$ and $\HK(\phi_c, \phi_j)$ hold for all $x,y \in M$ (not only for all $x,y\in M_0$).

\item[(ii)]
Since $\phi (r):= \phi_c (r) \wedge \phi_j(r)$, $\phi^{-1}(r)=\phi^{-1}_c(r) \vee
\phi^{-1}_j (r)$ and so
$$
 \frac{1}{V(x,\phi^{-1}(t))} = \frac 1{V(x,\phi_c^{-1}(t))
}\wedge \frac 1{V(x,\phi_j^{-1}(t))} .
$$
It follows from \eqref{e:1.25} and \eqref{e:1.27} that
under
  $\HK_-(\phi_c,\phi_j)$,
 for each fixed $c>0$,
$$
p(t, x, y)
\simeq 1/ V(x,\phi^{-1}(t))
\quad \hbox{when } d(x, y) \leq c \phi^{-1} (t).
$$
In particular,  $\NL (\phi)$ holds under
 $\HK_-(\phi_c,\phi_j)$, and so under $\HK(\phi_c,\phi_j)$.
 The off-diagonal estimates of $\HK (\phi_c, \phi_j)$ is expressed by the factor
$p^{(c)}(c_kt,x,y)
+p^{(j)}(t,x,y)$. In particular,
$\HK (\phi_c, \phi_j)$ is equivalent to
\begin{align*}
&p(t,x,y) \\
&  \asymp \, \frac{1}{V(x,\phi^{-1}(t))}
\wedge \left(\frac{t}{V(x,d(x,y))\phi_j(d(x,y))}+\frac{1}{V(x,\phi_c^{-1}(t))}  \exp\left(-
\frac{ d(x,y)}
{\bar \phi_c^{-1}(t/d(x,y))}\right)\right)\\
& \asymp \, \frac{1}{V(x,\phi^{-1}(t))}\wedge \left( \frac{1}{V(x,d(x,y))}\left(\frac{t}{\phi_j(d(x,y))}+\exp\left(-
\frac{ d(x,y)}
{\bar \phi_c^{-1}(t/d(x,y))}\right)\right)\right).
\end{align*}

\item[(iii)]
We can express
$\HK (\phi_c, \phi_j)$ in the following way. For $0< t\le 1$,
\begin{align*}
p(t,x,y)\asymp\begin{cases}\frac{1}{V(x,\phi_c^{-1}(t))},\quad &d(x,y)\le c_1\phi_c^{-1}(t),\\
\frac{t}{V(x,d(x,y))\phi_j(d(x,y))}+\frac{1}{V(x,\phi_c^{-1}(t))}  \exp\left(-
\frac{ d(x,y)}
{\bar \phi_c^{-1}(t/d(x,y))}\right),\quad&d(x,y)\ge c_1\phi_c^{-1}(t); \end{cases}\end{align*}for $t\ge1$,
\begin{align*}p(t,x,y)
\simeq p^{(j)}(t,x,y)
\simeq\begin{cases}\frac{1}{V(x,\phi_j^{-1}(t))},\quad &d(x,y)\le c_2\phi_j^{-1}(t),\\
\frac{t}{V(x,d(x,y))\phi_j(d(x,y))},\quad&d(x,y)\ge c_2\phi_j^{-1}(t). \end{cases}\end{align*}
In particular, for $t\ge 1$, heat kernel estimates are dominated by the non-local part of Dirichlet form $(\sE, \sF)$.
Furthermore, for $0<t\le 1$, we
have the following more explicit expression for $\HK (\phi_c, \phi_j)$:
\begin{align*}p(t,x,y)\asymp\begin{cases}\frac{1}{V(x,\phi_c^{-1}(t))},\quad &d(x,y)\le c_1\phi_c^{-1}(t),\\
\frac{1}{V(x,\phi_c^{-1}(t))} \exp\left(-
\frac{ d(x,y)}
{\bar \phi_c^{-1}(t/d(x,y))}\right),\quad&  c_1\phi_c^{-1}(t)\le d(x,y)\le
t_*,\\
\frac{t}{V(x,d(x,y))\phi_j(d(x,y))},\quad &d(x,y)\ge t_*,
\end{cases}\end{align*} where $t_*$ satisfies that
$$c_3\phi_c^{-1}(t)\log^{(\beta_{1,\phi_c}-1)/\beta_{2,\phi_c}}(\phi^{-1}_c(t)/\phi^{-1}_j(t))\le
t_*\le c_4
\phi_c^{-1}(t)\log^{(\beta_{2,\phi_c}-1)/\beta_{1,\phi_c}}(\phi^{-1}_c(t)/\phi^{-1}_j(t)),
$$ and $\beta_{1,\phi_c}$ and $\beta_{2,\phi_c}$ are given in
\eqref{polycon}. See the proof of Proposition \ref{thm:ujeuhkds} for
more details. Figure \ref{domfig} indicates which term is the dominant one for the estimate of $p(t,x,y)$ in each region.

\item[(iv)]
 For any $D\subset M$, it holds for  $x,y\in M_0$ and $t>0$ that
$p(t,x,y)\ge p^D(t,x,y),$ and so $\NDL(\phi)$ is stronger than $\NL(\phi)$. Furthermore, under $\VD$ and \eqref{polycon} we can prove that $\NL(\phi)$ together with $\UHK (\phi_c, \phi_j)$ implies $\NDL(\phi)$, see Lemma \ref{near-}.
We also note that, under $\VD$, $V(x_0, \phi^{-1}(t))$ in the definition of $\NDL(\phi)$ can be replaced
by either $V(x, \phi^{-1}(t))$ or $V(y, \phi^{-1}(t))$.
Under \eqref{polycon}, we may also replace $ \phi(\eps r)$ and
$\eps\phi^{-1}(t)$ in the definition of $\NDL(\phi)$ by $\eps\phi(r)$ and
$\phi^{-1}(\eps t)$, respectively.

\item[(v)] If in the lower bound for the definition of $\HK_- (\phi_c, \phi_j)$, we assume
$$
p(t, x, y) \geq c_0\left(\frac 1{V(x,\phi_c^{-1}(t))} \wedge p^{(j)}(t,x,y)  \right)
$$
instead of  \eqref{e:1.31}, then we only have
$$
p(t, x, y) \geq c_2\left( \frac 1{V(x,\phi_c^{-1}(t))
}\wedge \frac 1{V(x,\phi_j^{-1}(t))}   \right) = \frac{c_2}{V(x, \phi^{-1}(t))}
\quad \hbox{for } d(x, y) \leq c_1 \phi_j^{-1}(t)
$$
Note that as $\phi^{-1}(t)= {\bf1}_{[0, 1]}(t) \phi_c^{-1} (t) + {\bf 1}_{(1, \infty)} \phi_j^{-1} (t)$
and $\phi_c^{-1}(t) \geq \phi_j^{-1}(t)$ on $[0, 1]$,  the above inequality is weaker than $\NL (\phi)$
(for instance when $\phi_c (r)=r^2$ and $\phi_j (r)=r^\alpha$ with $0<\alpha<2$).
\end{itemize}
\end{remark}

We say
{\it $(\sE, \sF)$ is conservative}   if its associated Hunt process $X$
has infinite lifetime.  This is equivalent to $P_t 1 =1$ a.e.\ on $M_0$ for every $t>0$.
It follows from
\cite[Proposition 3.1(ii)]{CKW1} that $\VD$ and $\NL(\phi)$ imply
that $(\sE, \sF)$ is conservative.

\begin{thm} \label{T:main}
Assume that the metric measure space $(M, d , \mu)$ satisfies
$\VD$ and $\RVD$, and that
the scale functions $\phi_c$ and $\phi_j$ satisfy  \eqref{polycon} and \eqref{e:1.11}. Let $\phi:=\phi_c\wedge \phi_j$.
The following are equivalent:
\begin{itemize}
\item[\rm (i)] $\HK_- (\phi_c, \phi_j)$.

\item[\rm (ii)] $\UHK (\phi_c, \phi_j)$, $\NL(\phi)$ and $\J_{\phi_j}$.

\item[\rm (iii)] $\UHKD(\phi)$, $\NDL(\phi)$ and $\J_{\phi_j}$.

\item[\rm (iv)] $\PI(\phi)$, $\J_{\phi_j}$ and $\Gcap(\phi)$.

\item[\rm (v)] $\PI(\phi)$, $\J_{\phi_j}$ and $\CS(\phi)$.
\end{itemize}
If, additionally, $(M,d,\mu)$ is connected and satisfies the chain condition, then all the conditions above are equivalent to:
\begin{itemize}
\item[\rm (vi)] $\HK (\phi_c, \phi_j)$.
\end{itemize}
\end{thm}

In the process of establishing  Theorem \ref{T:main}, we also obtain the following characterizations for $\UHK(\phi_c, \phi_j)$.

\begin{thm} \label{T:main-1}
Assume that the metric measure space $(M, d, \mu)$ satisfies $\VD$ and $\RVD$, and
that the scale functions $\phi_c$ and $\phi_j$ satisfy  \eqref{polycon} and \eqref{e:1.11}. Let $\phi:=\phi_c\wedge \phi_j$.
Then the following are equivalent:
\begin{itemize}
\item[\rm (i)]  $\UHK(\phi_c, \phi_j)$ and $(\sE, \sF)$ is conservative.

\item[\rm (ii)]  $\UHKD(\phi)$, $\J_{\phi_j,\le}$ and $\E_\phi$.

\item[\rm (iii)]  $\FK(\phi)$, $\J_{\phi_j,\le}$ and $\Gcap(\phi)$.

\item[\rm (iv)]  $\FK(\phi)$, $\J_{\phi_j,\le}$ and $\CS(\phi)$.
\end{itemize}
\end{thm}

The proof of Theorem \ref{T:main-1} is given at the end of Section \ref{section4},
while the proof of Theorem \ref{T:main} is given at the end of Section \ref{S:5}.
We point out that $\UHK(\phi_c, \phi_j)$ alone
does not imply the conservativeness of the associated Dirichlet form $(\sE, \sF)$.
See \cite[Proposition 3.1 and Remark 3.2]{CKW1} for more details.
Under $\VD$, $\RVD$ and \eqref{polycon}, $\NDL(\phi)$ implies $\E_\phi$
(see Proposition \ref{P:4.1}(ii) below),
and so $ {\rm (iii)}$ in Theorem \ref{T:main} is stronger than ${\rm (ii)}$ in Theorem \ref{T:main-1}.
 We also note that $\RVD$
 is only used
in the implications of $\UHKD(\phi)\Longrightarrow \FK(\phi)$ and  $\PI(\phi)\Longrightarrow
 \FK(\phi)$;
see Proposition \ref{P:3.1} below.
In particular,
 $ {\rm (iii)}\Longrightarrow  {\rm (iv)} \Longrightarrow  {\rm (i)}$
  in Theorem
 \ref{T:main-1} holds true under $\VD$ and \eqref{polycon}.
 See Figure \ref{diagfig} below for various relations among $\UHK(\phi_c,\phi_j)$, $\HK_-(\phi_c,\phi_j)$ and $\HK(\phi_c,\phi_j)$.
 \begin{figure}[t]
\centerline{\epsfig{file=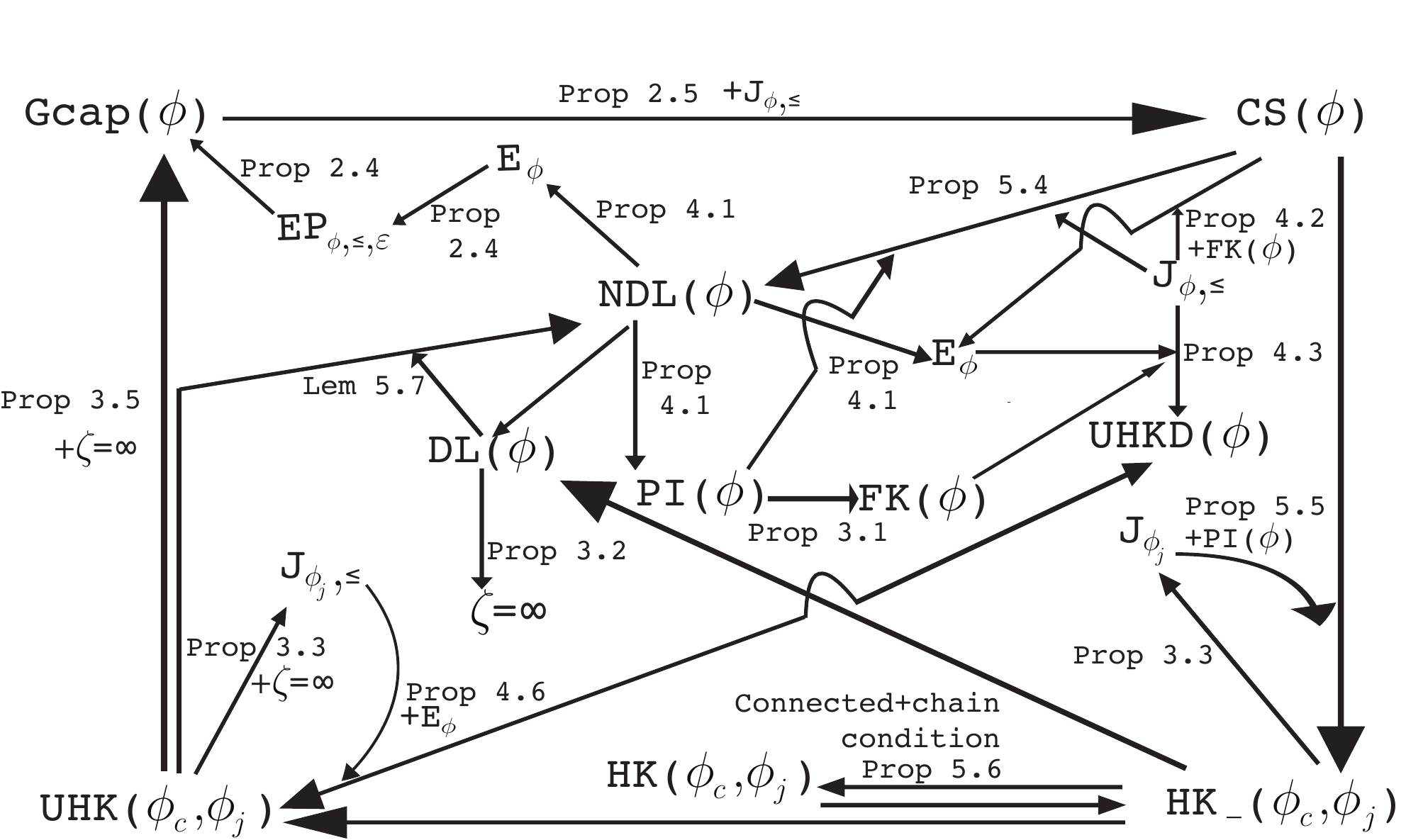, height=3in}}
\caption{Diagram for heat kernel estimates}\label{diagfig}
\end{figure}

 We emphasize again that the connectedness and the chain condition of the underlying metric measure space $(M,d,\mu)$
are only used to derive optimal
lower bounds
off-diagonal estimates for heat kernel when the time is small (i.e., from $\HK_- (\phi_c, \phi_j)$ to $\HK (\phi_c, \phi_j)$), while for other statements in the two main results above,
the metric measure space $(M, d,\mu)$ is only assumed to satisfy the general VD and RVD; that is, neither do we assume
$M$ to be connected nor $(M, d)$ to be geodesic.
Furthermore, we do not assume the uniform comparability of volume of balls; that is, we do not assume
the existence of a non-decreasing function $V$ on $[0, \infty)$ with $V(0)=0$ so that
 $\mu(B(x,r))\asymp V(r)$ for all $x\in M$ and $r>0$.

\subsection{Parabolic Harnack inequalities}

Let $Z:=\{V_s,X_s\}_{s\ge0}$ be the   space-time process corresponding to $X$,
where $V_s=V_0-s$.
The augmented
filtration generated by $Z$ satisfying the usual conditions will be denoted by $\{\widetilde{\mathcal{F}}_s;s\ge0\}$. The law of the space-time process $s\mapsto Z_s$ starting from $(t,x)$ will be denoted by $\bP^{(t,x)}$. For every open subset $D$ of $[0,\infty)\times M$, define
$\tau_D=\inf\{s>0:Z_s\notin D\}.$

\begin{definition}\label{def:PHI} \rm \begin{description}
\item{(i)} We say that a Borel measurable function $u(t,x)$ on
$[0,\infty)\times M$ is \emph{parabolic} (or \emph{caloric}) on
$D=(a,b)\times B(x_0,r)$ for the process $X$ if there is a properly
exceptional set $\mathcal{N}_u$ associated with the process $X$ so that for every
relatively compact open subset $U$ of $D$,
$u(t,x)=\bE^{(t,x)}u(Z_{\tau_{U}})$ for every $(t,x)\in
U\cap([0,\infty)\times (M\backslash \mathcal{N}_u)).$

\item{(ii)}
We say that the \emph{parabolic Harnack inequality} ($\PHI(\phi)$) holds for the process $X$, if there exist constants $0<C_1<C_2<C_3<C_4$,  $C_5>1$ and  $C_6>0$ such that for every $x_0 \in M $, $t_0\ge 0$, $R>0$ and for
every non-negative function $u=u(t,x)$ on $[0,\infty)\times M$ that is parabolic on cylinder $Q(t_0, x_0,\phi(C_4R),C_5R):=(t_0, t_0+\phi(C_4R))\times B(x_0,C_5R)$,
\be\label{e:phidef}
  \esssup_{Q_- }u\le C_6 \,\essinf_{Q_+}u,
 \ee where $Q_-:=(t_0+\phi(C_1R),t_0+\phi(C_2R))\times B(x_0,R)$ and $Q_+:=(t_0+\phi(C_3R), t_0+\phi(C_4R))\times B(x_0,R)$.
\end{description}
\end{definition}

The above $\PHI(\phi)$ is called a weak parabolic Harnack inequality in \cite{BGK},
in the sense that \eqref{e:phidef} holds for some $C_1, \cdots, C_5>0$.
It is called a parabolic Harnack inequality in \cite{BGK}
if \eqref{e:phidef} holds for any choice of positive constants $C_1, \cdots, C_5$
with $C_6=C_6(C_1, \dots, C_5)<\infty$. Since our underlying metric measure space may not be geodesic,
one can not expect to deduce  parabolic Harnack inequalities from weak  parabolic Harnack inequalities.

The following definition
was initially
 introduced in \cite{BBK2} in the setting of graphs.
See \cite{CKK2}
for the general setting of metric measure spaces.

\begin{definition}\label{thm:defUJS} {\rm
We say that $\UJS$ holds if
there is a non-negative symmetric function $J(x, y)$ on $M\times M$ so that
for $\mu\times \mu$ almost all $x,y\in M$, \eqref{e:1.2} holds, and that there is a constant $c>0$ such that for $\mu$-a.e. $x, y\in M$ with $x\not= y$,
\begin{equation}\label{ujs}
J(x,y)\le  \frac{c}{V(x,r)}\int_{B(x,r)}J(z,y)\,\mu(dz)
\quad\hbox{for every }
0<r\le   d(x,y) /2.
\end{equation}
} \end{definition}

The following are the main results for parabolic Harnack inequalities.
See Section \ref{S:6}
and Remark \ref{Heat:remark} for  notations appeared in the statement.

\begin{theorem}\label{C:1.25} Suppose that the metric measure space  $(M, d,  \mu)$ satisfies $\VD$ and $\RVD$, and
that the scale functions $\phi_c$ and $\phi_j$ satisfy  \eqref{polycon} and \eqref{e:1.11}. Let $\phi:=\phi_c\wedge \phi_j$.
Then
the following statements are equivalent.
 \begin{itemize}
 \item[\rm (i)]  $\PHI(\phi)$ .

 \item[\rm (ii)]  $\UHK_{weak} (\phi) + \NDL (\phi) +\UJS $.

 \item[\rm (iii)] $\PHR(\phi)+\E_\phi+ \UJS +\J_{\phi,\le}$.

 \item[\rm (iv)] $\EHR +\E_\phi +\UJS+\J_{\phi,\le }$.

 \item[\rm (v)]  $\PI(\phi)+\J_{\phi,\le }+\Gcap(\phi)+\UJS$.

 \item[\rm (vi)] $\PI(\phi)+\J_{\phi,\le }+\CS(\phi)+\UJS$.
 \end{itemize}
Consequently, we have
\begin{equation}\label{e:HKPHI}
\HK_- (\phi_c, \phi_j)\Longleftrightarrow \PHI(\phi)+\J_{\phi_j}.
\end{equation}
 If in additional, the metric measure space $(M,d,\mu)$ is connected and satisfies the chain condition, then
\begin{equation}\label{eq:5-6Fuz}
\HK (\phi_c, \phi_j)\Longleftrightarrow \PHI(\phi)+\J_{\phi_j}.
\end{equation}
 \end{theorem}

The equivalence between {\rm (i)} and {\rm (ii)} will be  proved in Theorem \ref{Thm: prop1},   the equivalence between
{\rm (i)},  {\rm (iii)} and  {\rm (iv)}  will be established in Theorem \ref{Thm: prop2}, while the equivalence between
{\rm (i)},  {\rm (v)} and  {\rm (vi)} will be given in Theorem  \ref{Thm:anal}.  The last
two assertions of Theorem \ref{C:1.25} follow
from the equivalence between {\rm (i)},  {\rm (v)} and Theorem \ref{T:main}.

We emphasize that, different from the purely non-local setting as
studied in \cite{CKW2}, $\PHI(\phi)$ alone
can only imply $\J_{\phi,\le }$ but not the stronger $\J_{\phi_j,\le}$.
See Example
\ref{PHIexm} for a counterexample.

 \medskip

The rest of the paper is organized as follows. In the next section, we present some preliminary results about $\Gcap(\phi)$.
We show in Proposition \ref{P:csp} that
$\Gcap(\phi)$ along with $\J_{\phi,\le}$ yields
$\CS(\phi)$.
This immediately yields  $ {\rm (iv)} \Longrightarrow  {\rm (v)}$ in Theorem \ref{T:main} and
$ {\rm (iii)} \Longrightarrow  {\rm (iv)}$ in Theorem \ref{T:main-1}.
Furthermore,
$\CS(\phi)$ enjoys the self-improving property, and enables us to make full use of the ideas in \cite{CKW1, CKW2}.
For example, via them we can obtain the $L^1$-mean value inequalities in the present setting, which play a key tool to obtain $\E_{\phi}$.  In Section \ref{section3}, we
investigate
consequences of $\UHK(\phi_c, \phi_j)$, and establish ${\rm (i)} \Longrightarrow  {\rm (iii)}$ of Theorem \ref{T:main-1}. Section \ref{section4} is devoted to obtaining $\UHK(\phi_c, \phi_j)$, which is the most difficult part of the paper.
The crucial step is to apply rough tail probability estimates to derive sharp $\UHK(\phi_c, \phi_j)$, which
requires detailed analysis of the roles of the
local and non-local parts in different time and space regions.
The proof of Theorem \ref{T:main-1} is given at the end of Section \ref{section4}.
Section \ref{S:5} is devoted to the two-sided heat kernel estimates and
the proof of Theorem \ref{T:main}.
Various characterizations of $\PHI(\phi)$ are given in Section \ref{S:6}.  In the last section, some examples are shown to illustrate the applications of our results, and a counterexample is also given to indicate that alone $\PHI(\phi)$ does not imply $\J_{\phi_j,\le}.$

\medskip

Throughout this paper, we will use $c$, with or without subscripts,
to denote strictly positive finite constants whose values are
insignificant and may change from line to line.
For $p\in [1, \infty]$, we will use $\| f\|_p$ to denote the $L^p$-norm in $L^p(M;\mu)$. For $a,b\in \bR_+$, $a\wedge b=:\min\{a,b\}$ and $a\vee b=:\max\{a,b\}.$
For $B=B(x_0, r)$ and $a>0$, we use $a B$ to denote
the ball $B(x_0, ar)$,
and $\bar{B}:=\{x\in M:d(x,x_0)\le r\}$. For any subset $D$ of $M$,
$D^c$ denotes its complement in $M$.

\section{Preliminaries}\label{section2}

In this section, we mainly present some preliminary results about $\Gcap(\phi)$.  For our later use to establish the characterizations of parabolic Harnack inequalities, we always assume that $\J_{\phi,\le}$ is satisfied in this section. Since $\phi(r)\le \phi_j(r)$ for all $r>0$, $\J_{\phi,\le}$ is weaker than $\J_{\phi_j,\le}$, and so all the results in this section still hold true with $\J_{\phi,\le}$ replaced by $\J_{\phi_j,\le}$.

\subsection{Properties of $\phi_c$ and $\phi_j$}

We recall the following statement from \cite{CKW1}.

 \begin{lemma}\label{intelem}$(${\rm\cite[Lemma 2.1]{CKW1}}$)$ Assume that $\VD$ and
 \eqref{polycon} hold. If $\J_{\phi_j,\le}$ $($resp. $\J_{\phi,\le}$$)$ holds,
then there exists a constant $c_1>0$ such that
\[
 \int_{B(x,r)^c}
J(x,y)\,\mu (dy)\le \frac{c_1}{\phi_j (r)}
\quad\mbox{for every }
x\in M \hbox{ and } r>0.
\]
\[\left ( resp.\
 \int_{B(x,r)^c}
J(x,y)\,\mu (dy)\le \frac{c_1}{\phi (r)}
\quad\mbox{for every }
x\in M \hbox{ and } r>0.
\right)\]
\end{lemma}

The following lemma is concerned with the exponential function in estimates for $p^{(c)}(t,x,y)$.

\begin{lemma}\label{L:diff}
Under  \eqref{polycon},
 for any $c_0>0$ there exists a constant $C:=C(c_0)\ge1$ so that for any $t,r>0$,
$$ \frac{r}{C\bar \phi_c^{-1}(t/r)}\le \sup_{s>0}\left\{\frac{r}{s}-\frac{c_0t}{\phi_c(s)}\right\}\le \frac{Cr}{\bar \phi_c^{-1}(t/r)}.$$ \end{lemma}

\begin{proof} We fix $c_0>0$ throughout the proof.
Let $c_2\geq c_1>0$ be the positive constants in \eqref{e:1.12}. For any $b>c_0/c_2+1$ and $t,r>0$,
\begin{align*}
\sup_{s>0}\left\{\frac{r}{s}-\frac{c_0t}{\phi_c(s)}\right\}
&\ge  \frac{r}{\bar\phi_c^{-1}(bt/r)}-\frac{c_0t}{\phi_c(\bar\phi_c^{-1}(bt/r))}\\
&\geq \frac{r}{\bar\phi_c^{-1}(bt/r)}-\frac{c_0r}{c_2 b\bar\phi_c^{-1}(bt/r)}=\frac{(bc_2 -c_0)r}{bc_2\bar\phi_c^{-1}(bt/r)}\\
&\ge\frac{1}
{c_{2,\bar \phi_c} b^{1+1/(\beta_{1,\phi_c}-1)}}\frac{r}{ \bar\phi_c^{-1}(t/r)},
\end{align*}
where we used \eqref{e:1.12} in the second inequality and \eqref{e:effdiff} in the last inequality. On the other hand,
for any $0<a\le 1$ and $t,r>0$,
\begin{align*}
\sup_{s>0}\left\{\frac{r}{s}-\frac{c_0t}{\phi_c(s)}\right\}
&\le \sup_{0<s\le a \bar\phi_c^{-1}(t/r)}\left\{\frac{r}{s}-\frac{c_0t}{\phi_c(s)}\right\}+ \sup_{s> a \bar\phi_c^{-1}(t/r)}\left\{\frac{r}{s}-\frac{c_0t}{\phi_c(s)}\right\}\\
&\le   \sup_{0<s\le a \bar\phi_c^{-1}(t/r)}\left\{\frac{r}{s}-\frac{c_0 r \bar \phi_c(s/a)}{\phi_c(s)}\right\}+\frac{r}{a \bar\phi_c^{-1}(t/r)}\\
&\leq
 \sup_{0<s\le a \bar\phi_c^{-1}(t/r)}\left\{\frac{r}{s}\left(1-\frac{a c_0  c_1\phi_c(s/a)}{\phi_c(s)}\right)\right\} +\frac{r}{a \bar\phi_c^{-1}(t/r)}\\
&\le \sup_{0<s\le a
\bar\phi_c^{-1}(t/r)}\left\{\frac{r}{s}\left(1-c_0 c_1 c_{1,\phi_c}a^{-\beta_{1,\phi_c}+1}\right)\right\}+\frac{r}{a
\bar\phi_c^{-1}(t/r)},
\end{align*}
where we used \eqref{e:1.12} in the third inequality and
\eqref{polycon} in the last inequality. Taking $a\in (0,1]$ such
that $a=a_0:=(1+(1/(c_0c_1 c_{1,\phi_c})))^{-1/(\beta_{1,\phi_c}-1)}$ in
the inequality above, we find that
$$\sup_{s>0}\left\{\frac{r}{s}-\frac{c_0t}{\phi_c(s)}\right\}\le \frac{r}{a_0 \bar\phi_c^{-1}(t/r)}.$$

The desired assertion follows from the above two  conclusions. \qed
\end{proof}

\begin{corollary}\label{C:diff} Under  \eqref{polycon},  the expressions \eqref{eq:fibie2} and \eqref{eq:fibie3} for $p^{(c)}(t,x,y)$ are equivalent
in the sense that \eqref{e:1.20} holds.
\end{corollary}

\begin{proof}
This is a direct consequence of Lemma \ref{L:diff}.
\qed\end{proof}

\subsection{$\EP_{\phi, \leq, \eps} \Longrightarrow \Gcap (\phi)$ and
$\Gcap(\phi)+\J_{\phi,\le}\Longrightarrow \CS(\phi)$.}\label{section2.2}

Recall that  $\Gamma(f,f)$ (resp. $\Gamma_c (f, f)$) is the \emph{energy measure} of $f$ for $\sE$ (resp. its strongly local part $\sE^{(c)}$).
For any $f,g\in \sF$, the signed measure $\Gamma(f,g)$ is defined by
$$
\Gamma(f,g)=\frac{1}{2}\big(\Gamma(f+g,f+g)-\Gamma(f,f)-\Gamma(g,g)\big).
$$
Similar definition applies to $\Gamma_c (f, g)$.
The measure $ \Gamma(f,g)$ is symmetric and bilinear in  $(f,g)$.
The following  Cauchy-Schwarz inequality holds:
\begin{equation}\label{e:sc-ine}\begin{split}
\left|\int_M fg\,d\Gamma(u,v)\right|\le & \left(\int_M f^2\,d\Gamma(u,u)\right)^{1/2}\left(\int_M g^2\,d\Gamma(v,v)\right)^2\\
\le& \frac{1}{2\lambda}\int_M f^2\,d\Gamma(u,u)+\frac{\lambda}{2}\int_M g^2\,d\Gamma(v,v)
\end{split}\end{equation}
for all $u,v\in \sF$, bounded $f, g$ on $M$, and $\lambda>0$.
When the Dirichlet form $(\sE, \sF)$ admits no killings as the one given by \eqref{e:1.1},
$$
\sE(f,g)= \Gamma(f,g)(M) \quad \hbox{for } f, g \in \sF.
$$

The following Leibniz and chain rules hold for the energy  measure
$\Gamma_c$ for the strongly local part $\sE^{(c)}$ of
the Dirichlet form $(\sE, \sF)$: for all $f,g,h\in \sF_b$,
$$
\Gamma_c(fg,h)(dx) =g (x) \Gamma_c(f,h) (dx) +f (x) \Gamma_c(g,h) (dx)
$$
and
$$
\Gamma_c(\Phi(f),g) (dx)= \Phi'(f) (x) \Gamma_c(f,g) (dx),
$$
where $\Phi:\bR\to \bR$ is any smooth function with $\Phi(0)=0$.
The measure $\Gamma_c$  has the strong local property in the sense that  if $f$ is constant on a  set $F$,
then
\begin{equation}\label{e:comment-1}
\Gamma_c(f,f ) (F)=0 .
\end{equation}
see \cite[Theorem 4.3.8 and Exercise 4.3.12]{CF}.
For $f, g\in \sF$,  define
$$
\Gamma_j (f, g) (dx) := \int_{y\in M} (f(x)-f(y))(g(x)-g(y))\,J(d x,d y).
$$
It is easy to check that
 the following chain rule holds:
 $$
\int_M\,d\Gamma_j(f g,h)=\int_Mf \, d\Gamma_j(g,h)+\int_Mg\, d\Gamma_j(f,h)
\quad \hbox{for  } f,g,h \in \sF_b .
$$
See \cite[Lemma 3.5 and Theorem 3.7]{CKS} for more details.

\medskip

The following proposition extends \cite[Lemma 2.8]{GHH}
from strongly local Dirichet forms on metric measure spaces
to symmetric Dirichlet forms having both local and nonlocal terms
on general metric measure spaces.

\begin{proposition}\label{P:2.4}
Suppose that
\eqref{polycon} holds. Then $\EP_{\phi,\le, \eps}
\Longrightarrow \Gcap (\phi)$. Consequently,
$$\EP_{\phi,\leq}\Longrightarrow\EP_{\phi,\le, \eps}\Longrightarrow \Gcap (\phi),$$ and
$$\E_\phi
\Longrightarrow\EP_{\phi,\le, \eps}\Longrightarrow \Gcap (\phi).$$
\end{proposition}

\begin{proof} The proof of  $\EP_{\phi,\le, \eps}
\Longrightarrow \Gcap (\phi)$ is the same as that of \cite[Lemma 2.8]{GHH}.
We note that while \cite[Leamm 2.8]{GHH} concerns
with purely non-local Dirichlet forms, its proof does not use any character of pure jump Dirichlet forms and
works for general symmetric Dirichlet forms.
Clearly, $\EP_{\phi, \leq}$ implies $\EP_{\phi,\le, \eps}$. By the same proof as that of
\cite[Lemma 4.16]{CKW1},
we have $\E_\phi\Longrightarrow \EP_{\phi,\le, \eps}$.
This establishes the last assertion.
\qed
\end{proof}

\medskip

The next proposition
extends \cite[Lemma 2.4]{GHH} from strongly local Dirichet forms on metric measure spaces
satisfying the $d$-set upper bound condition
to
symmetric Dirichlet forms having both local and nonlocal terms
on general metric measure spaces with $\VD$ condition. It
in particular
gives the implication $ {\rm (iv)} \Longrightarrow  {\rm (v)}$ in Theorem \ref{T:main}
and $ {\rm (iii)} \Longrightarrow  {\rm (iv)}$ in Theorem \ref{T:main-1}.

\begin{proposition}\label{P:csp} Under $\VD$ and \eqref{polycon},
$$\Gcap(\phi) \hbox{ and } \J_{\phi,\le}\Longrightarrow \CS(\phi).$$ \end{proposition}

To prove Proposition \ref{P:csp}, we need the following
lemma.

\begin{lemma} \label{L:dirichlet00}
${\rm(i)}$ For each $f,g\in \sF_b$ and $\eta>0$,
$$(1-\eta^{-1})\int f^2d\Gamma_c(g,g)\le \int d\Gamma_c(g,gf^2)+\eta\int g^2d\Gamma_c(f,f).$$
${\rm(ii)}$ For each $f,g\in \sF_b$, $\eta>0$ and any subset $D \subset M$,
\be\label{eq:enowec}\begin{split}
&(1-\eta^{-1})\int_{D\times D} f^2(x)(g(x)-g(y))^2\,J(dx,dy)\\
&\le \int_{D\times D}(g(x)f^2(x)-g(y)f^2(y))(g(x)-g(y))\,J(dx,dy)\\
 &\quad +\eta\int_{D\times D} g^2(x)(f(x)-f(y))^2\,J(dx,dy).
\end{split}\ee
\end{lemma}

\begin{proof}\eqref{eq:enowec} has been proved in \cite[Lemma 3.5]{CKW1}, and so we only need to show (i). For any $f,g\in \sF_b$ and $\eta>0$, by Leibniz and chain rules  and the Cauchy-Schwarz inequality \eqref{e:sc-ine},
\begin{align*}\int \,d\Gamma_c(g,f^2g)=&2\int g f\,d\Gamma_c(f,g)+\int f^2\,d\Gamma_c(g,g)\\
\ge&-\eta^{-1}\int f^2\,d\Gamma_c(g,g)-\eta\int g^2\,d\Gamma_c(f,f)+\int f^2\,d\Gamma_c(g,g),\end{align*}proving the assertion (i).  \qed\end{proof}

\noindent
{\bf Proof of Proposition \ref{P:csp}.}\quad  For fixed $x_0\in M$, $0<r\le R$ and $C_0\in (0,1]$, set $B_0=B(x_0,R)$, $B_1=B(x_0,R+r/2)$, $B_2=B(x_0,R+r)$ and $B_3=B(x_0, R+(1+C_0)r)$. For any $f\in \sF$,
by $\Gcap(\phi)$
there is a $\kappa$-cut-off function $\vp$ for $B_0\subset B_1$ such that
\begin{equation}
\sE(f^2\vp,\vp)\le \frac{C_1}{\phi(r)}\int_{B_1}f^2\,d\mu.
\end{equation}

Since $\vp=0$ on $B_1^c$, we have
\begin{align*}
\sE(f^2\vp,\vp)
&=\sE^{(c)}(f^2\vp,\vp)+\sE^{(j)}(f^2\vp,\vp)\\
&=\sE^{(c)}(f^2\vp,\vp)\\
&\quad+\left(\int_{B_2\times B_2}+\int_{B_2^c\times B_2}+\int_{B_2\times B_2^c}\right)(\vp(x)-\vp(y))(f^2(x)\vp(x)-f^2(y)\vp(y))\,J(dx,dy)\\
&\ge\int_{B_2}d\Gamma_c(f^2\vp,\vp)+\int_{B_2\times B_2}(\vp(x)-\vp(y))(f^2(x)\vp(x)-f^2(y)\vp(y))\,J(dx,dy).
\end{align*}
This along with Lemma \ref{L:dirichlet00} with $\eta=2$
and $\Gcap(\phi)$ yields that
\begin{align*}&\int_{B_2}f^2\,d\Gamma_c(\vp,\vp)+\int_{B_2\times B_2}f^2(x)(\vp(x)-\vp(y))^2\,J(dx,dy)\\
&\le 2\left(\int_{B_2}d\Gamma_c(f^2\vp,\vp)+\int_{B_2\times B_2}
(\vp(x)-\vp(y))(f^2(x)\vp(x)-f^2(y)\vp(y))\,J(dx,dy)\right)\\
&\quad+ 4\left(\int_{B_2}\vp^2\,d\Gamma_c(f,f)+\int_{B_2\times B_2}\vp^2(x)(f(x)-f(y))^2\,J(dx,dy)\right)\\
 &\le 2 \sE(f^2\vp,\vp)+ 4\left(\int_{B_2}\vp^2\,d\Gamma_c(f,f)+\int_{B_2\times B_2}\vp^2(x)(f(x)-f(y))^2\,J(dx,dy)\right)\\
&\le \frac{2C_1}{\phi(r)}\int_{B_1}f^2\,d\mu + 4\left(\int_{B_2}\vp^2\,d\Gamma_c(f,f)+\int_{B_2\times B_2}\vp^2(x)(f(x)-f(y))^2\,J(dx,dy)\right).
\end{align*}
Replacing \cite[(2.4) on page 447]{GHH} with the inequality above, and following the proof of \cite[Lemma 2.4]{GHH} from \cite[(2.4) on page 447]{GHH} to the end, we can obtain that
\begin{align*}&\int_{B_3}f^2\,d\Gamma_c(\vp,\vp)+\int_{B_3\times B_3}f^2(x)(\vp(x)-\vp(y))^2\,J(dx,dy)\\
&=\int_{B_2}f^2\,d\Gamma_c(\vp,\vp)+\int_{B_3\times B_3}f^2(x)(\vp(x)-\vp(y))^2\,J(dx,dy)\\
&\le \frac{C_2}{\phi(r)}\int_{B_3}f^2\,d\mu + 4\left(\int_{B_2}\vp^2\,d\Gamma_c(f,f)+
\int_{B_2\times B_2}\vp^2(x)(f(x)-f(y))^2\,J(dx,dy)\right).
\end{align*}
(Indeed, this can be seen by replacing $B$ and $\Omega$ in the proof of  \cite[Lemma 2.4]{GHH} with $B_2$ and $B_3$, respectively.)
 We note that for the argument above, we used Lemma \ref{intelem}.
Again by  Lemma \ref{intelem}, we can further improve the inequality above into
 \begin{align*}\int_{B_3}f^2\,d\Gamma(\vp,\vp)
=& \int_{B_3}f^2\,d\Gamma_c(\vp,\vp)+\int_{B_3\times B_3}f^2(x)(\vp(x)-\vp(y))^2\,J(dx,dy)\\
&+\int_{B_3\times B_3^c}f^2(x)(\vp(x)-\vp(y))^2\,J(dx,dy)\\
\le& \int_{B_3}f^2\,d\Gamma_c(\vp,\vp)+\int_{B_3\times B_3}f^2(x)(\vp(x)-\vp(y))^2\,J(dx,dy)\\
&+\int_{B_1\times B_3^c}f^2(x)\,J(dx,dy)\\
 \le & 4\left(\int_{B_2}\vp^2\,d\Gamma_c(f,f)+\int_{B_2\times B_2}\vp^2(x)(f(x)-f(y))^2\,J(dx,dy)\right)+\frac{C_3}{\phi(r)}\int_{B_3}f^2\,d\mu.
\end{align*}This proves the desired assertion.
 \qed

\subsection{$\rho$-truncated version of $\CS(\phi)$}

To deal with processes with long rang jumps, we will frequently use the truncation. Fix $\rho>0$ and define  a bilinear form $(\sE^{(\rho)}, \sF)$ by
\begin{align*}\sE^{(\rho)}(u,v)=&\sE^{(c)}(u,v)+\int_{M\times M} (u(x)-u(y))(v(x)-v(y)){\bf 1}_{\{d(x,y)\le \rho\}}\, J(dx,dy)\\
=&:\sE^{(c)}(u,v)+\sE^{(\rho)}_j(u,v).\end{align*}
Clearly, the form $\sE^{(\rho)}(u,v)$ is well defined for $u,v\in \sF$, and $\sE^{(\rho)}(u,u)\le \sE(u,u)$ for all $u\in \sF$. Assume that $\VD$, \eqref{polycon}  and $\J_{\phi,\le}$ hold. Then  we have by  Lemma \ref{intelem}
  that for all $u\in \sF$,
\begin{align*}
   \sE(u,u)-\sE^{(\rho)}(u,u)&= \int_{M\times M}(u(x)-u(y))^2{\bf 1}_{\{d(x,y)>\rho\}}\,J(dx,dy)\\
&\le 4\int_Mu^2(x)\,\mu(dx)\int_{B(x,\rho)^c}J(x,y)\,\mu(dy)\le \frac{c_0\|u\|_{2}^2 }{\phi(\rho)}.
\end{align*}
Thus $\sE_1 (u, u)$ is equivalent to $\sE^{(\rho)}_1(u,u):=
\sE^{(\rho)}(u,u)+ \|u\|_2^2$ for every $u\in \sF$. Hence $(\sE^{(\rho)},
\sF)$ is a regular Dirichlet form on $L^2(M; \mu)$.
Throughout this paper, we call $(\sE^{(\rho)}, \sF)$  $\rho$-truncated
Dirichlet form. The Hunt process associated with $(\sE^{(\rho)}, \sF)$, denoted by $(X_t^{(\rho)})_{t\ge0}$, can be identified in distribution with
the Hunt process of the original Dirichlet form $(\sE, \sF)$ by removing those jumps of size larger than
 $\rho$.

 Define
 $J(x,dy)=J(x,y)\,\mu(dy)$.
   Let $J^{(\rho)}(dx,dy)={\bf 1}_{\{d(x,y)\le \rho\}} J(dx,dy)$, $J^{(\rho)}(x,dy)={\bf 1}_{\{d(x,y)\le \rho\}} J(x,dy)$, and  $\Gamma^{(\rho)}_j(f,g)$ be the
    carr\'e du champ of the non-local part $\sE^{(\rho)}_j$ for the $\rho$-truncated Dirichlet form $(\sE^{(\rho)}, \sF)$; namely,
$$ \sE^{(\rho)}_j(f,g)=\int_M \,\mu(dx)\int_M(f(x)-f(y))(g(x)-g(y))\,J^{(\rho)}(x,dy)
=:\int_M d\Gamma^{(\rho)}_j(f,g). $$ We also set
$$\Gamma^{(\rho)}(f,g):=\Gamma_c(f,g)+\Gamma^{(\rho)}_j(f,g).$$

\begin{lemma}\label{L:cs-trun}Under $\VD$, \eqref{polycon} and $\J_{\phi,\le}$, if $\CS(\phi)$ holds, then $\CS^{(\rho)}(\phi)$ holds too, i.e.,\  there exist constants $C_0\in (0,1]$ and $C_1, C_2>0$
such that for every
$0<r\le R$, almost all $x_0\in M$ and any $f\in \sF$, there exists
a cut-off function $\vp\in \sF_b$ for $B(x_0,R) \subset B(x_0,R+r)$ so that the following holds for all $\rho\in (0,\infty]$:
 \begin{align*}
 &\int_{B(x_0,R+(1+C_0)r)} f^2 \, d\Gamma^{(\rho)} (\vp,\vp)\\
&\le C_1 \left(\int_{B(x_0,R+r)}\vp^2\, d\Gamma_c(f,f)+\int_{B(x_0,R+r)\times B(x_0,R+(1+C_0)r)}\vp^2(x)(f(x)-f(y))^2\,J^{(\rho)}(dx,dy)\right) \\
&\quad + \frac{C_2}{\phi(r\wedge\rho)}  \int_{B(x_0,R+(1+C_0)r)} f^2  \,d\mu.\end{align*}  \end{lemma}
\begin{proof} The proof is based on Lemma \ref{intelem}, and we omit it here. See \cite[Proposition 2.3(1)]{CKW1}.
\qed \end{proof}

We also note that, by the proof of \cite[Proposition 2.3(4)]{CKW1}, under $\VD$, \eqref{polycon} and $\J_{\phi,\le}$, if $\CS(\phi)$ (resp.\ $\CS^{(\rho)}(\phi)$) holds for some $C_0\in (0,1]$, then for any $C_0'\in (C_0,1]$, there exist constants $C_1, C_2>0$ (where $C_2$ depends on $C_0'$) such that $\CS(\phi)$ (resp.\ $\CS^{(\rho)}(\phi)$) holds for $C_0'$.

\begin{proposition}\label{CSJ-equi}
Assume that $\VD$, \eqref{polycon} and $\J_{\phi,\le}$ hold.
 If $\CS(\phi)$ holds, then there is a constant $c_0>0$ such that for  every
$0<r\le R$, $\rho>0$ and almost all $x\in M$,
$$
\mbox{\rm Cap} ^{(\rho)} (B(x,R),B(x,R+r))\le c_0\frac{V(x,R+r)}{\phi(r\wedge \rho)}.
$$
In particular, we have
\be\label{eq:fneoobo3}
\mbox{\rm Cap} (B(x,R),B(x,R+r))\le c_0\frac{V(x,R+r)}{\phi(r)}.
\ee

\end{proposition}

\begin{proof}
According to Lemma \ref{L:cs-trun} and the remark below its proof,  $\CS^{(\rho)}(\phi)$ holds for every $\rho >0$ and we
may and do take
$C_0=1$ in \eqref{e:csj1}. Fix $x_0\in M$ and
write $B_s:=B(x_0,s)$ for $s\ge 0$. Let
$f\in \sF$ with $0\le f\le 1$ such that $f|_{B_{R+2r}}=1$ and $f|_{B_{R+3r}^c}=0$. For any $0<r\le R$, let $\vp\in \sF_b$ be the cut-off function for $B_R\subset B_{R+r}$ associated with $f$ in $\CS^{(\rho)}(\phi)$. Then for any $\rho>0$,
\begin{align*}
\mbox{{\rm Cap}}^{(\rho)} (B_R, B_{R+r})\le& \int_{B_{R+2r}}\,d\Gamma^{(\rho)}(\vp,\vp)+\int_{B_{R+2r}^c}\,d\Gamma^{(\rho)}(\vp,\vp)\\
=&\int_{B_{R+2r}}f^2\,d\Gamma^{(\rho)}(\vp,\vp)+\int_{B_{R+2r}^c}\,d\Gamma^{(\rho)}(\vp,\vp)\\
\le &c_1\left(\int_{B_{R+r}} \,d\Gamma_c(f,f)+\int_{B_{R+r}\times B_{R+2r}}(f(x)-f(y))^2\,J^{(\rho)}(dx,dy)\right)
\\
&+\frac{c_2}{\phi(r\wedge \rho)}\int_{B_{R+2r}}f^2\,d\mu+\int_{B_{R+2r}^c}\,\mu(dx)\int_{B_{R+r}} \vp^2(y)J(x,y)\,\mu(dy)\\
\leq & \, 0+0+
\frac{c_2\mu(B_{R+2r})}{\phi(r\wedge \rho)}+\frac{c_3\mu(B_{R+r})}{\phi(r)}\\
\le &\frac{c_4\mu(B_{R+r})}{\phi(r\wedge \rho)},
\end{align*}
where we used $\CS^{(\rho)}(\phi)$ in the second inequality, and applied Lemma
\ref{intelem} and $\VD$ in the third inequality.

Now let $f_\rho$ be
the potential whose $\sE^{(\rho)}$-norm gives the capacity. Then the
Ces\`{a}ro mean of a subsequence of $f_\rho$ converges in
$\sE_1$-norm, say to $f$, and $\sE(f, f)$ is no less than the
capacity corresponding to $\rho =\infty$.  So \eqref{eq:fneoobo3} is
proved.\qed\end{proof}

\subsection{Self-improvement of $\CS(\phi)$}
We next show that the leading  constant in $\CS^{(\rho)}(\phi)$ is self-improving
in the following sense.

\begin{proposition} \label{L:cswk}
Suppose that $\VD$, \eqref{polycon} and $\J_{\phi,\le}$ hold. If $\CS^{(\rho)}(\phi)$ holds, then there exists a constant $C_0\in (0,1]$ so that for any $\varepsilon>0$, there exists a constant $c_1(\varepsilon)>0$ such that for every $0<r\le R$, almost all $x_0\in M$ and any $f\in \sF$, there exists a cut-off function $\vp\in \sF_b$ for $B(x_0,R)\subset B(x_0,R+r)$ so that the following holds for all $\rho>0$:
\begin{equation}\label{e:csj3} \begin{split}
 &\int_{B(x_0,R+(1+C_0)r)} f^2 \, d\Gamma^{(\rho)} (\vp,\vp)\\
&\le\varepsilon \bigg(\int_{B(x_0,R+r)} \vp^2\,d\Gamma_c(f,f)\\
&\qquad\quad+\int_{B(x_0,R+r)\times B(x_0,R+(1+C_0)r)}\vp^2(x)(f(x)-f(y))^2\,J^{(\rho)}(dx,dy)\bigg) \\
&\quad + \frac{c_1(\varepsilon)}{\phi(r\wedge\rho)}  \int_{B(x_0,R+(1+C_0)r)} f^2  \,d\mu.\end{split} \end{equation}
\end{proposition}

\begin{proof} By the remark before Proposition \ref{CSJ-equi},
we may and do assume that $\CS^{(\rho)}(\phi)$ holds with $C_0=1$.
Fix $x_0\in M$, $0<r\le R$ and $f\in \sF$. For $s >0$, set $B_{s}=B(x_0,s)$.
The goal is to construct a cut-off function $\vp\in \sF_b$ for  $B_R
\subset B_{R+r}$
so that \eqref{e:csj3} holds.
Without loss of generality, in the following we may and do assume that $\int_{B_{R+2r}} f^2\,d\mu>0$; otherwise, \eqref{e:csj3} holds trivially.

For $\lam>0$ whose exact value to be determined later,
let
$$ s_n = c_0 r e^{- n \lam/(2\beta_{2,\phi})}, $$
where $c_0:=c_0(\lam)$ is chosen so that $\sum_{n=1}^\infty s_n =r $, and $\beta_{2,\phi}:=\beta_{2,\phi_c}\vee \beta_{2,\phi_j}$ is given in \eqref{polycon000}.
Set $r_0=0$ and
$$ r_n = \sum_{k=1}^n s_k,\quad n\ge 1.$$
Clearly, $R < R+r_1 < R+r_2 < \dots < R+r$.
For any $n\ge 0$, define $U_n:= B_{R+r_{n+1}} \setminus B_{R+r_n}$, and
$U_n^*:=B_{R+r_{n+1}+s_{n+1} } \setminus  B_{R+r_n-s_{n+1}}$.
Let $\theta>0$,
whose value also to be determined later,
 and define $f_\theta:=|f|+\theta$.
By $\CS^{(\rho)}(\phi)$
(with $R=R+r_n$, $r=r_{n+1}-r_n=s_{n+1}$),
 there exists a cut-off function $\vp_n$ for
$B_{R+r_n} \subset B_{R+r_{n+1}}$
such that
\begin{align*}
 &\int_{B_{R+r_{n+1}+s_{n+1}}} f_\theta^2 \, d\Gamma^{(\rho)}(\vp_n,\vp_n)\\
 &\le C_1 \left( \int_{B_{R+r_{n+1}}}\vp_n^2\,d\Gamma_c(f_\theta,f_\theta)+\int_{B_{R+r_{n+1}}\times B_{R+r_{n+1}+s_{n+1}}} \vp_n^2(x)(f_\theta(x)-f_\theta(y))^2\, J^{(\rho)}(dx,dy) \right) \\
&\quad+ \frac{C_2 }{\phi(s_{n+1}\wedge \rho)} \int_{B_{R+r_{n+1}+s_{n+1}}}   f_\theta^2 \,d\mu,\end{align*}
 where $C_1,C_2$ are positive constants independent of $f_\theta$ and $\vp_n$.
Here, we mention that since $(\sE, \sF)$ is a regular Dirichlet form on $L^2(M,\mu)$, $f_\theta\in \sF_{loc}$, and so, by Remark \ref{rek:scj}(ii), $\CS^{(\rho)}(\phi)$ can be applied to $f_\theta.$

Let $b_n = e^{-n \lam}$ and define
\be \label{phi}
 \vp = \sum_{n=1}^\infty (b_{n-1} - b_n) \vp_n.
\ee
 Then  $\vp$ is a cut-off
function for $B_R \subset B_{R+r}$, because $\vp=1$ on $B_R$ and  $\vp=0$ on $B_{R+r}^c$.
Hence, combining the proof of \cite[Lemma 5.1]{AB} with that of \cite[Proposition 2.4]{CKW1}, we can verify that the function $\vp$ defined by \eqref{phi} satisfies \eqref{e:csj3}
with $C_0=1$ and
$\vp\in \sF_b$. (In particular, by \cite[(5.7) in the proof of Lemma 5.1]{AB}, we can insert the function $\vp^2$ in front of $d\Gamma_c(f,f)$.) The details are omitted here.
\qed
\end{proof}

As a direct consequence of Lemma \ref{L:cs-trun} and Proposition \ref{L:cswk}, we have the following corollary.
\begin{corollary} \label{C:cswk-1}
Suppose that $\VD$, \eqref{polycon}, $\J_{\phi,\le}$ and $\CS(\phi)$ hold.
 Then  there exists a constant $c_1>0$
such that for every $0<r\le R$, almost all $x_0 \in M$ and any $f\in \sF$, there exists
a cut-off function $\vp\in \sF_b$ for $B(x_0,R) \subset B(x_0,R+r)$ so that the following holds for all $\rho\in (0,\infty]$,
\be \label{e:csa2} \begin{split}
 &\int_{B(x_0,R+2r)} f^2 \, d\Gamma^{(\rho)}(\vp,\vp)\\
 & \le \frac 18 \bigg(\int_{B(x_0,R+r)} \vp^2\,d\Gamma_c(f,f)+\int_{B(x_0,R+r)\times B(x_0,R+2r)}\vp^2(x)(f(x)-f(y))^2\,J^{(\rho)}(dx,dy)\bigg)\\
&\quad + \frac{c_1}{\phi(r\wedge \rho)}  \int_{B(x_0,R+2r)} f^2 \, d\mu.
\end{split}\ee  \end{corollary}

\subsection{Consequences of $\CS(\phi)$: Caccioppoli and $L^1$-mean value inequalities} \label{s:cac}

 In this subsection, we establish mean value inequalities for subharmonic functions.
For
stability results for heat kernel estimates, we only need  mean value inequalities for the $\rho$-truncated Dirichlet form $(\sE^{(\rho)},\sF)$, while the mean value inequalities for the original Dirichlet form $(\sE,\sF)$ will be used as one of the key tools in the study of
 characterizations of parabolic Harnack inequalities.
We
will first
 present these inequalities for subharmonic functions of the original Dirichlet form $(\sE,\sF)$
 and then indicate similar inequalities for subharmonic functions
 of the $\rho$-truncated Dirichlet form $(\sE^{(\rho)},\sF)$.

\begin{definition}\rm
Let ${D}$ be an open subset of $M$.
\begin{description}
\item{\rm (i)} We say that a bounded nearly Borel measurable function $u$ on $M$
is \emph{$\sE$-subharmonic}  (resp. \emph{$\sE$-harmonic, $\sE$-superharmonic})
in ${D}$ if $u\in\sF^{D}_{loc}$ satisfies
\begin{equation*}\label{an-har}
\sE(u,\varphi)\le 0 \quad (\textrm{resp.}\ =0, \ge0)
\end{equation*}
for any $0\le\varphi\in\sF^{D}.$

\item{\rm (ii)} A nearly Borel measurable function $u$ on $M$ is said to be \emph{subharmonic}  (resp. \emph{harmonic, superharmonic})
in ${D}$ (with respect to the process $X$)
if for any relatively compact subset $U\subset D$,
$t\mapsto   u (X_{t\wedge \tau_U}) $ is a uniformly integrable submartingale
(resp. martingale, supermartingale) under $\bP^x$ for q.e. $x\in M$.
\end{description}
\end{definition}

The following result is established in \cite[Theorem 2.11  and Lemma 2.3]{Chen} first for harmonic functions,
and then extended in \cite[Theorem 2.9]{ChK} to subharmonic functions.

\begin{theorem}\label{equ-har}Let ${D}$ be an open subset of $M$, and let $u$ be a bounded function.
Then $u$ is $\sE$-harmonic $($resp.  $\sE$-subharmonic$)$ in ${D}$ if and only if $u$ is  harmonic
 $($resp. subharmonic$)$
 in ${D}$.
 \end{theorem}

To establish the Caccioppoli inequality, we also need the following definition.

\begin{definition} \rm Let $\psi:\bR_+\to \bR_+$.
For a Borel measurable function $u$ on $M$, we define its
\emph{nonlocal tail}   in the ball
$B(x_0,r)$ with respect to the function $\psi$ by
\begin{equation} \label{def-T}
\T_\psi\, (u; x_0,r)=\psi(r)\int_{B(x_0,r)^c}\frac{|u(z)|}{V(x_0,d(x_0,z))\psi(d(x_0,z))}\,\mu(dz).
\end{equation}
\end{definition}

Suppose that $\VD$ and \eqref{polycon} hold.
By
Lemma \ref{intelem},  both $\T_{\phi_j}\, (u; x_0,r)$  and $\T_{\phi}\, (u; x_0,r)$  are finite if $u$ is bounded.

 \ \

We first show that $\CS(\phi)$ enables us to prove
a Caccioppoli
inequality for $\sE$-subharmonic functions.

\begin{lemma} \label{L:ci}
{\bf (Caccioppoli inequality)}\,
Suppose that $\VD$, \eqref{polycon}, $\CS(\phi)$ and
$\J_{\phi,\le}$ hold.
For $x_0 \in M$ and $s>0$, let $B_s= B(x_0, s)$.
For $0<r <  R$, let $u$ be an
$\sE$-subharmonic function on $B_{R+r}$ for the Dirichlet
form $(\sE, \sF)$, and $v := (u-\theta)^+$ for $\theta>0$.
Let $\vp$ be the cut-off function for $B_{R-r} \subset B_{R}$ associated with $v$ in
$\CS(\phi)$. Then there exists a constant $c>0$ independent of
$x_0, R, r$ and $\theta$ such that
\begin{equation}\label{e:cacc}\begin{split}
\int_{B_{R+r}}\,d\Gamma(v\vp,v\vp)\le  \frac{c }{\phi(r)} \left[ 1+\frac{1}{\theta}\left(1+\frac{R}{r}\right)^{d_2+\beta_{2,\phi}-\beta_{1,\phi}}
\T_\phi\,(u; x_0, R+r)\right]\int_{B_{R+r}} u^2\,d\mu,
\end{split}\end{equation} where $\beta_{1,\phi}$ and $\beta_{2,\phi}$ are given in \eqref{polycon000}.
\end{lemma}

\begin{proof} (i) Since $u$ is $\sE$-subharmonic in $B_{R+r}$ for the Dirichlet
form $(\sE, \sF)$ and $\vp^2v\in \sF^{B_{R}}$,
we have $u\in \FF_{B_{R+r}}^{loc}$ and $\sE(u,\vp^2v)=\sE^{(c)}(u,\vp^2v)+\sE^{(j)}(u,\vp^2v)\le 0.$

As $u-v=u 1_{\{u\leq \theta\}} + \theta 1_{\{u>\theta\}}$ and $v=0$ on $\{u\leq \theta\}$, we have by \eqref{e:comment-1} that
$\Gamma_c (u-v, v)=0$ on $M$.
Hence by  the Leibniz and chain rules as well as
 the Cauchy-Schwarz inequality \eqref{e:sc-ine},
 \begin{align*}
 \sE^{(c)}(u,\vp^2v)&=\int_M \vp^2\,d\Gamma_c(u,v)+2\int_M \vp v\,d\Gamma_c(u,\vp)\\
&=\int_M \vp^2\,d\Gamma_c(v,v)+2\int_M \vp v\,d\Gamma_c(v,\vp)\\
&\ge \int_M \vp^2\,d\Gamma_c(v,v)-\left( 4\int_M v^2\,d\Gamma_c(\vp,\vp)+\frac{1}{4}\int_M \vp^2\,d\Gamma_c(v,v)\right)\\
&\ge \frac{3}{4}\int_M \vp^2\,d\Gamma_c(v,v)-4\int_M v^2\,d\Gamma_c(\vp,\vp)\\
&=\frac{3}{4}\int_{B_{R}} \vp^2\,d\Gamma_c(v,v)-4\int_{B_{R}} v^2\,d\Gamma_c(\vp,\vp).
\end{align*}
On the other hand, by \cite[(4.5)]{CKW1}, we have
 \begin{align*}
\sE^{(j)}(u,\vp^2v)&\ge \int_{B_{R}\times B_{R+r}} \vp^2(x)(v(x)-v(y))^2 \, J(dx,dy)-4\int_{B_{R+r}} v^2\, d\Gamma_j(\vp,\vp)\\
&\quad- \frac{c_1 }{\theta\phi(r)}
\bigg[ \left(1+\frac{R}{r}\right)^{d_2+\beta_{2,\phi}-\beta_{1,\phi}} \T_\phi\,(u; x_0,
R+r)\bigg] \int_{B_{R}} u^2\,d\mu.
\end{align*}
Combining all the estimates above, we arrive at
\begin{equation}\label{e:ci2}\begin{split}
0\le &4\int_{B_{R+r}} v^2\, d\Gamma(\vp,\vp)\\
&
-\frac{3}{4}\Big(\int_{B_{R}} \vp^2\,d\Gamma_c(v,v)+\int_{B_{R}\times B_{R+r}} \vp^2(x)(v(x)-v(y))^2 \, J(dx,dy)\Big) \\
& +\frac{c_1 }{\theta\phi(r)}
\bigg[ \left(1+\frac{R}{r}\right)^{d_2+\beta_{2,\phi}-\beta_{1,\phi}} \T_\phi\,(u; x_0,
R+r)\bigg] \int_{B_{R}} u^2\,d\mu.
\end{split}\end{equation}

(ii) It is easy to see
 from  the Leibniz rule and
 the Cauchy-Schwarz inequality \eqref{e:sc-ine} that
$$
\int_{B_{R+r}}\,d\Gamma_c(v\vp, v\vp)\le 2\int_{B_{R+r}} v^2\,d\Gamma_c(\vp,\vp)+2 \int_{B_{R+r}}\vp^2\,d\Gamma_c(v,v).
$$
According to \cite[(4.6)]{CKW1},
\begin{align*}
\int_{B_{R+r}}\,d\Gamma_j(v\vp,v\vp)
\le& 2\int_{B_{R+r}}v^2 \,d\Gamma_j(\vp,\vp)
+2\int_{B_{R}\times B_{R+r}}\vp^2(x)(v(x)\!-\!v(y))^2\,J(dx,dy)\\
& +\frac{c_2}{\phi(r)}\int_{B_{R}}u^2\,d\mu.
\end{align*}
Hence,
\begin{equation}\label{e:cac33}\begin{split}
\int_{B_{R+r}}\,d\Gamma(v\vp,v\vp)
\le&   2\int_{B_{R+r}}v^2 \,d\Gamma(\vp,\vp)
+2 \int_{B_{R}\times B_{R+r}}\vp^2(x)(v(x)\!-\!v(y))^2\,J(dx,dy) \\
& +2 \int_{B_{R}}\vp^2\,d\Gamma_c(v,v)
 +\frac{c_2}{\phi(r)}\int_{B_{R}}u^2\,d\mu.
\end{split} \end{equation}
Combining \eqref{e:ci2} with \eqref{e:cac33}, we have for $a>0$,
\begin{equation}\label{eq:nefow}\begin{split}
&\, a\int_{B_{R+r}}\,d\Gamma(v\vp,v\vp)\\
&\le (2a+4)\int_{B_{R+r}}v^2\,d\Gamma(\vp,\vp)\\
&\quad
 +\left(2a-\frac{3}{4}\right)\left(\int_{B_{R}}\vp^2\,d\Gamma_c(v,v)+\int_{B_{R}\times B_{R+r}}\vp^2(x)(v(x)\!-\!v(y))^2\,J(dx,dy)\right)\\
&\quad+\frac{c_3(1+a) }{\phi(r)}
\bigg[ 1+\frac{1}{\theta}\left(1+\frac{R}{r}\right)^{d_2+\beta_{2,\phi}-\beta_{1,\phi}}
\T_\phi\,(u; x_0, R+r)\bigg] \int_{B_{R}} u^2\,d\mu.
\end{split}\end{equation}

 Next by using \eqref{e:csa2} for $v$ with $\rho=\infty$,  we have
\begin{align*}
& \int_{B_{R+r}}v^2\,d\Gamma(\vp,\vp)\\
& \le \frac 18 \left[\int_{B_{R}}\vp^2\,d\Gamma_c(v,v)+\int_{B_{R} \times
B_{R+r}}\vp^2(x)(v(x)-v(y))^2\,J(dx,dy) \right] + \frac{c_0}{\phi(r)} \int_{B_{R+r}}
v^2  \,d\mu\\
& \le\frac 18 \left[\int_{B_{R}}\vp^2\,d\Gamma_c(v,v)+\int_{B_{R} \times
B_{R+r}}\vp^2(x)(v(x)-v(y))^2\,J(dx,dy) \right] + \frac{c_0}{\phi(r)} \int_{B_{R+r}}
u^2  \,d\mu.
\end{align*}
Plugging this into \eqref{eq:nefow} with $a=1/9$ (so that $(4+2a)/8+(2a-(3/4))=0$), we obtain
\begin{align*}
\frac 19\int_{B_{R+r}}\,d\Gamma(v\vp,v\vp)\le  \frac{c_4 }{\phi(r)}
\bigg[1+\frac{1}{\theta}\left(1+\frac{R}{r}\right)^{d_2+\beta_{2,\phi}-\beta_{1,\phi}}
\T_\phi\,(u; x_0, R+r)\bigg]\int_{B_{R+r}} u^2\,d\mu,
\end{align*} which proves the desired assertion.
\qed\end{proof}

\begin{proposition} \label{P:mvi2g}
{\bf($L^2$-mean value inequality)}\,
Assume $\VD$, \eqref{polycon}, $\FK(\phi)$, $\CS(\phi)$ and
$\J_{\phi,\le}$ hold.  For $x_0\in M$ and $R>0$,  let $u$
be a  bounded
$\sE$-subharmonic function
in $B(x_0,R)$.  Then  for any $\delta>0$,
$$
 \esssup_{B(x_0,R/2)} u
 \le  c_1\left[ \left(\frac{(1+\delta^{-1})^{1/\nu}}{V(x_0,R)}\int_{B(x_0,{R})} u^2\,d\mu \right)^{1/2}+\delta\T_\phi\, (u; x_0,R/2)
  \right] ,
$$
 where $\nu$ is the constant appearing in the $\FK(\phi)$
inequality \eqref{e:fki}, and $c_1>0$ is a constant independent of
$x_0$, $R$ and $\delta$.
In particular, there is a constant $c>0$ independent of
$x_0$ and $R$ so that
$$
 \esssup_{B(x_0,R/2)} u
 \le  c\left[  \left(\frac{1}{V(x_0,R)}\int_{B(x_0,{R})} u^2\,d\mu \right)^{1/2}+ \T_\phi\, (u; x_0,R/2)\right].
$$
\end{proposition}

\begin{proof} With the aid of \eqref{e:cacc}, one can see that the comparison results over balls as stated in \cite[Lemma 4.8]{CKW1} still hold true.
We can then follow the proof of \cite[Proposition 4.10]{CKW1} line by the line to obtain the desired assertion. We omit details here.   \qed
\end{proof}

In the following, we consider $L^2$ and $L^1$ mean value
inequalities for $\sE$-subharmonic functions for truncated Dirichlet
forms.
In the truncated situation we can no longer
 use the nonlocal tail of subharmonic functions. The  remedy is to enlarge the integral regions
 in the right hand side of the mean value inequalities.
Since the proof is almost the same as these of \cite[Proposition 4.11 and Corollary 4.12]{CKW1} (with some necessary modifications as done in the proof of Lemma \ref{L:ci}), we omit it here.

\begin{proposition} \label{P:mvi2}{\bf($L^2$ and $L^1$ mean value inequalities for $\rho$-truncated Dirichlet forms)}
Assume $\VD$, \eqref{polycon}, $\FK(\phi)$,
$\CS(\phi)$ and $\J_{\phi,\le}$ hold.
There are positive constants $c_1, c_2>0$ so that for $x_0\in M$, $\rho,
R>0$, and for any bounded $\sE^{(\rho)}$-subharmonic function $u$ on $B(x_0,R)$ for the $\rho$-truncated
Dirichlet form $(\sE^{(\rho)}, \sF)$, we have
\begin{itemize}
\item[{\rm(i)}]
$$
 \esssup_{B(x_0,R/2)} u^2  \le  \frac{c_1
 }{V(x_0,R)}\left(1+\frac{\rho}{R}\right)^{d_2/\nu} \Big(1+\frac{R}{\rho}\Big)^{\beta_{2,\phi}/\nu}
    \int_{B(x_0,{R+\rho})} u^2\,d\mu;
$$
\item[{\rm(ii)}] \be \label{e:mvi}
 \esssup_{B(x_0,R/2)} u \le  \frac{c_2}{V(x_0,R)} \Big(1+\frac{\rho}{R}\Big)^{d_2/\nu}\Big(1+\frac{R}{\rho}\Big)^{\beta_{2,\phi}/\nu}
    \int_{B(x_0,R+\rho)} u \,d\mu.
\ee \end{itemize}
 Here, $\nu$ is the constant in $\FK(\phi)$, $d_2$ and  $\beta_{2,\phi}$ are the exponents in \eqref{e:vd2} from $\VD$
  and \eqref{polycon000} respectively.
  \end{proposition}

\section{Implications of heat kernel estimates} \label{section3}

First we note that by the same proof of \cite[Proposition 7.6]{CKW1}, we have the following.

\begin{proposition}\label{P:3.1}
Under $\VD$, $\RVD$ and \eqref{polycon},
$$
\UHKD(\phi) \Longrightarrow\FK(\phi) \qquad \hbox{and} \qquad \PI (\phi) \Longrightarrow\FK (\phi).
$$
\end{proposition}

\smallskip

Denote by $\zeta$ the lifetime of the Hunt process $X$ associated with the regular Dirichlet form $(\sE, \sF)$
on $L^2(M; \mu)$.
We have the following two facts.

\begin{proposition}\label{P:3.2}
\begin{description}
\item{\rm (i)}
 Under $\VD$, $\NL (\phi) \Longrightarrow \zeta =\infty$.

\item{\rm (ii)}  Assume that $\VD$, \eqref{polycon}, $\E_\phi$ and $\J_{\phi, \leq }$ hold.
Then $(\sE, \sF)$ is conservative; that is, $\zeta =\infty$.
\end{description}
\end{proposition}

\begin{proof}
(i) is taken directly from \cite[Proposition 3.1(ii)]{CKW1}, which holds for any symmetric
Markov process.

(ii) can be proved by exactly the same argument as that
of \cite[Lemma 4.21]{CKW1}.
The details
are
omitted here.  \qed
\end{proof}

\subsection{$\UHK(\phi_c, \phi_j) +
(\sE, \sF) \hbox{ is conservative}
\Longrightarrow \J_{\phi_j,\le}$ and  $\HK_- (\phi_c, \phi_j) \Longrightarrow \J_{\phi_j}$}

\begin{proposition}\label{l:jk}
Under $\VD$ and \eqref{polycon},
$$
\UHK(\phi_c, \phi_j)
\hbox{ and }
(\sE, \sF) \hbox{ is conservative}
\Longrightarrow \J_{\phi_j,\le}
$$ and  $$\HK_- (\phi_c, \phi_j)  \Longrightarrow \J_{\phi_j}.$$
In particular,
$\HK (\phi_c, \phi_j)  \Longrightarrow \J_{\phi_j}$.
\end{proposition}

\begin{proof}  We only prove that case that
$\HK_- (\phi_c, \phi_j)\Longrightarrow
\J_{\phi_j}$, and the other two cases can be verified similarly.
For $t>0$, consider the bilinear form $\sE^{(t)}(f,g):=\langle f-P_{t}
f,g\rangle /t$. Since $(\sE, \sF)$ is
conservative by
Proposition \ref{P:3.2}(i),
we can write
\begin{equation*}
\sE^{(t)}(f,g)=\frac{1}{2t}\int_{M}\int_{M}(f(x)-f(y))(g(x)-g(y))p(t,x,y)\,\mu
(dx)\,\mu (dy).
\end{equation*}
It is well known that $\lim_{t\to 0}\sE^{(t)}(f,g)=
\sE(f,g)$ for all $f,g\in \sF$. Let $A$, $B$ be disjoint compact sets, and
take $f,g\in \sF$ such that $\mathrm{supp}\,f\subset A$ and $\mathrm{supp}\, g\subset B$. Then in view of the strongly local property of  $\sE^{(c)}(f,g)$,
\begin{equation*}
\sE^{(t)}(f,g)=-\frac{1}{t}\int_{A}\int_{B}f(x)g(y)p(t,x,y)\,\mu (dy)\,\mu (dx)
\overset{t\rightarrow 0}{\longrightarrow }-\int_{A}\int_{B}f(x)g(y)\,J(dx,dy).
\end{equation*}
Let
$r_0:=d(A,B)$. For any $0<t\le \phi_c(r_0)$, by $\VD$, \eqref{polycon}, \eqref{e:1.12} and \eqref{e:effdiff},
\begin{align*}&\sup_{x\in A,y\in B} \frac{1}{V(x,\phi_c^{-1}(t))}\exp\left(-c_1\frac{d(x,y)}{\bar \phi_c^{-1}(t/d(x,y))}\right)\\
&\le c_2 \sup_{x\in A,y\in B} \frac{1}{V(x,d(x,y))}\left(\frac{d(x,y))}{\phi_c^{-1}(t)}\right)^{d_2}\exp\left(-c_3\left(\frac{\phi_c(d(x,y))}{t}\right)^{1/(\beta_{2,\phi_c}-1)}\right)\\
&\le c_4 \sup_{x\in A,y\in B} \frac{1}{V(x,d(x,y))}\left(\frac{\phi_c(d(x,y)))}{t}\right)^{d_2\beta_{2,\phi_c}}\exp\left(-c_3\left(\frac{\phi_c(d(x,y))}{t}\right)^{1/(\beta_{2,\phi_c}-1)}\right)\\
&\le  c_5 \sup_{x\in A,y\in B}\frac{1}{V(x,d(x,y))} \exp\left(-\frac{c_3}{2}\left(\frac{\phi_c(d(x,y))}{t}\right)^{1/(\beta_{2,\phi_c}-1)}\right)\\
&\le c_6\sup_{x\in A}\frac{1}{V(x,r_0)}\exp\left(-\frac{c_3}{2}\left(\frac{\phi_c(r_0))}{t}\right)^{1/(\beta_{2,\phi_c}-1)}\right),\end{align*} where in the third inequality we used the fact that
$$r^{d_2\beta_{2,\phi_c}}\le c_7\exp\left(\frac{c_3}{2}r^{1/(\beta_{2,\phi_c}-1)}\right),\quad r\ge 1.$$ The inequality above
yields that $$\lim_{t\to0}\frac{1}{t}\int_{A}\int_{B}f(x)g(y)p^{(c)}(t,x,y)\,\mu (dy)\,\mu (dx)=0.$$
Hence, using
$\HK_- (\phi_c, \phi_j)$, we obtain
\begin{equation*}
\int_{A}\int_{B}f(x)g(y)\,J(dx,dy)\asymp\int_{A}\int_{B}\frac{f(x)g(y)}{
V(x,d(x,y))\phi_j(d(x,y))}\,\mu (dy)\,\mu (dx)
\end{equation*}
for all $f,g\in \sF$ such that $\mathrm{supp}\,f\subset A$ and $\mathrm{supp}\,g\subset B$.
Since $A$, $B$ are arbitrary disjoint compact sets,  it follows that
$J(dx,dy)$ is absolutely continuous w.r.t. $\mu (dx)\,\mu (dy)$, and
$\J_{\phi_j}$ holds.  \qed
\end{proof}

\subsection{$\UHK(\phi_c, \phi_j)
\hbox{ and } (\sE, \sF) \hbox{  is conservative}
\Longrightarrow\Gcap(\phi)$}

In this subsection, we give the proof that $\UHK(\phi_c, \phi_j)$
and the conservativeness of $(\sE, \sF)$ imply $\Gcap(\phi)$.

We begin with the following lemma.

\begin{lemma}\label{Conserv} Assume that
$\VD$, \eqref{polycon} and $\UHK(\phi_c, \phi_j)$ hold and that $(\sE, \sF)$ is conservative.
Then $\EP_{\phi,\le}$ holds.
\end{lemma}

\begin{proof}
We first verify that there is a constant $c_1>0$ such that for each $t,r>0$ and for almost all $x\in M$,
$$
\int_{B(x,r)^c} p(t, x,y)\,\mu(dy)\le \frac{c_1 t}{\phi(r)}.
$$
 Indeed, we only need to consider the case that $\phi(r)>t$; otherwise, the inequality above holds trivially with $c_1=1$.
 According to $\UHK(\phi_c, \phi_j)$,
 $\VD$, \eqref{polycon}, \eqref{e:1.12} and \eqref{e:effdiff}, for any $t,r>0$ with $\phi(r)>t$ and almost all $x\in M$,
\begin{align*}
&\int_{B(x,r)^c}p(t, x,y)\,\mu(dy)\\
&=\sum_{i=0}^\infty\int_{B(x,2^{i+1}r)\setminus B(x,2^ir)}p(t, x,y)\,\mu(dy)\\
&\le \sum_{i=0}^\infty \frac{c_2 V(x,2^{i+1}r)}{V(x,\phi_c^{-1}(t))}\exp\left(-\frac{c_3 2^i r}{\bar\phi_c^{-1}(t/(2^ir))}\right)+\sum_{i=0}^\infty \frac{c_2t V(x,2^{i+1}r)}{V(x,2^{i}r)\phi_j(2^ir)}\\
&\le c_4\sum_{i=0}^\infty \left(\frac{2^ir}{\phi_c^{-1}(t)}\right)^{d_2}\exp\left(-\frac{c_5 2^i r}{\bar\phi_c^{-1}(t/r)}\right)+ \frac{c_4t}{\phi_j(r)}\sum_{i=0}^\infty 2^{-i\beta_{1,\phi_j}}\\
&\le c_6\sum_{i=0}^\infty\left(2^i\left(\frac{\phi_c(r)}{t}\right)^{1/\beta_{1,\phi_c}}\right)^{d_2}\exp\left(-c_72^i\left(\frac{\phi_c(r)}{t}\right)^{1/(\beta_{2,\phi_c}-1)}\right)+\frac{c_6t}{\phi_j(r)}\\
&\le c_6\sum_{i=0}^\infty\left(2^{i(\beta_{2,\phi_c}-1)}\frac{\phi_c(r)}{t}\right)^{d_2(1+1/(\beta_{2,\phi_c}-1))}\exp\left(-c_7\left(2^{i(\beta_{2,\phi_c}-1)}\frac{\phi_c(r)}{t}\right)^{1/(\beta_{2,\phi_c}-1)}\right)+\frac{c_6t}{\phi_j(r)}\\
&\le c_8\sum_{i=0}^\infty\exp\left(-\frac{c_7}{2}\left(2^{i(\beta_{2,\phi_c}-1)}\frac{\phi_c(r)}{t}\right)^{1/(\beta_{2,\phi_c}-1)}\right)+\frac{c_6t}{\phi_j(r)}\\
&\le c_{9}t\left(\frac{1}{\phi_c(r)}+\frac{1}{\phi_j(r)}\right)\le  \frac{c_{10}t}{\phi(r)},
\end{align*} where in the arguments above we used $\phi_c(r)\wedge \phi_j(r)=\phi(r)$ for all $r>0$, in the second inequality we used the fact that $\bar \phi_c(r)$ is increasing on $(0,\infty)$, in the fourth inequality we used the fact that $\beta_{1,\phi_c}>1$, and in the fifth and the sixth inequalities we used the facts that there is a constant $c_{11}>0$ such that
$$s^{d_2(1+1/(\beta_{2,\phi_c}-1))}\le c_{11}\exp\left(\frac{c_7}{2}s^{1/(\beta_{2,\phi_c}-1)}\right),\quad s\ge 1$$ and
$$\sum_{i=0}^\infty \exp\Big(-\frac{c_7}{2} (2^is)^{1/(\beta_{2,\phi_c}-1)}\Big)\le c_{11}/s,\quad s\ge 1,$$ respectively.

Now, since $(\sE, \sF)$ is conservative, by the strong Markov property, for any each $t,r>0$ and for almost all $x\in M$,
\begin{align*}\bP^x(\tau_{B(x,r)}\le t)&=\bP^x(\tau_{B(x,r)}\le t, X_{2t}\in B(x,r/2)^c)+\bP^x(\tau_{B(x,r)}\le t, X_{2t}\in B(x,r/2))\\
&\le\bP^x( X_{2t}\in B(x,r/2)^c)+ \sup_{z\notin B(x,r)^c, s\le t}\bP^z( X_{2t-s}\in B(z,r/2)^c)\\
&\le \frac{c_{13}t}{\phi(r)},\end{align*} which yields
$\EP_{\phi,\le}$. (Note that the conservativeness of $(\sE, \sF)$ is used in the equality above. Indeed, without the conservativeness,
there must be an extra term $\bP^x(\tau_{B(x,r)}\le t, \zeta\le 2t)$
in the right hand side of the above equality, where $\zeta$ is the
lifetime of $X$.)\qed\end{proof}

\begin{proposition} \label{T:gdCS}
Suppose that $\VD$, \eqref{polycon} and $\UHK(\phi_c, \phi_j)$
hold, and that $(\sE, \sF)$ is conservative.
Then $\Gcap(\phi)$ holds.
\end{proposition}

\begin{proof}
According to Lemma \ref{Conserv}, $\EP_{\phi,\le}$ hold true.
Thus the desired assertion follows from Proposition \ref{P:2.4}.
\qed\end{proof}

\section{Implications of $\FK(\phi)$, $\NDL (\phi)$, $\CS (\phi)$  and  $\J_{\phi_j, \le}$}\label{section4}

\subsection{
$\NDL(\phi)\Longrightarrow \PI (\phi) + \E_{\phi}$}\label{S:4.1}

\begin{proposition}\label{P:4.1}
Assume that $\VD$ and \eqref{polycon} hold.

\begin{description}
\item{\rm (i)}  $\NDL (\phi) \Longrightarrow \PI (\phi) + \E_{\phi, \geq}$.

\item{\rm (ii)} If in addition $\RVD$ holds,
 then  $\NDL (\phi) \Longrightarrow \E_\phi$.
\end{description}
\end{proposition}

\begin{proof} The proof is the same as that
of \cite[Proposition 3.5]{CKW2}, so it is omitted.
\qed
\end{proof}

\subsection{$\FK(\phi)+\J_{\phi,\le}+\CS(\phi)\Longrightarrow \E_\phi$  and $\FK(\phi)+\J_{\phi,\le}+ \E_\phi\Longrightarrow \UHKD(\phi)$}\label{S:4.2}

The next two propositions can be proved by following the arguments in \cite{CKW1}, which are based on the probabilistic ideas and
are valid for general symmetric  Dirichlet forms.

\begin{proposition} \label{P:exit}
Assume $\VD$, \eqref{polycon}, $\FK(\phi)$, $\J_{\phi,\le}$ and $\CS(\phi)$ hold.  Then  $\E_{\phi}$ holds.
\end{proposition}

\begin{proof} According to \cite[Lemma 4.14]{CKW1}, $\E_{\phi,\le} $ holds under $\VD$, \eqref{polycon} and $\FK(\phi)$.
On the other hand, by Proposition \ref{P:mvi2}, we have the $L^1$-mean value
inequality \eqref{e:mvi} under the assumptions of Proposition \ref{P:exit}.
Then by the proofs of \cite[Lemmas 4.15, 4.16 and 4.17]{CKW1}, we  obtain $\E_{\phi,\ge}$.  \qed
\end{proof}

\begin{proposition}\label{ehi-ukdd}  Suppose that  $\VD$, \eqref{polycon}, $\FK(\phi)$, $\E_\phi$ and $\J_{\phi, \le}$ hold. Then $\UHKD(\phi)$ is satisfied, i.e., there is a constant $c>0$ such that for all $x\in M_0$ and $t>0$,
$$p(t, x,x)\le \frac{c}{V(x,\phi^{-1}(t))}.$$
 \end{proposition}

\begin{proof} The proof
is the same as that for \cite[Theorem 4.25]{CKW1}, and so we omit the details here.
\qed\end{proof}

\subsection{$\UHKD(\phi)+\J_{\phi_j,\le}+ \E_\phi\Longrightarrow \UHK(\phi_c, \phi_j)$}\label{S:4.3}

First we have the following  by the proof of \cite[Lemma 5.2]{CKW1}.

\begin{lemma}\label{P:truncated} Under $\VD$ and \eqref{polycon}, if
$\UHKD(\phi)$, $\J_{\phi,\le}$ and $\E_{\phi}$ hold, then the $\rho$-truncated Dirichlet form $(\sE^{(\rho)}, \sF)$ has the heat kernel $q^{(\rho)}(t,x,y)$, and it satisfies that for any $t> 0$ and all $x,y\in M_0$,
$$q^{(\rho)}(t,x,y)\le c_1\bigg(\frac{1}{V(x,\phi^{-1}(t))}+ \frac{1}{V(y,\phi^{-1}(t))}\bigg) \exp\left(c_2\frac{t}{\phi(\rho)}-c_3\frac{d(x,y)}{\rho}\right),$$ where $c_1,c_2,c_3$ are positive constants independent of $\rho$.
Consequently, for any $t> 0$ and all $x,y\in M_0$,
$$q^{(\rho)}(t,x,y)\le \frac{c_4}{V(x,\phi^{-1}(t))}\bigg(1+ \frac{d(x,y)}{\phi^{-1}(t)}\bigg)^{d_2} \exp\left(c_2\frac{t}{\phi(\rho)}-c_3\frac{d(x,y)}{\rho}\right).$$
\end{lemma}

 \ \

From now till the end of this subsection, we will assume that $\J_{\phi_j,\le}$ is satisfied.
Note that since $\phi(r)\le \phi_j(r)$ for all $r>0$,
condition $\J_{\phi_j,\le}$ implies $\J_{\phi,\le}$.

Recall that
there is close relation between $p(t,x,y)$ and $q^{(\rho)}(t,x,y)$  via
Meyer's decomposition, e.g.\ see \cite[Section 7.2]{CKW1}. In particular, according to \cite[(4.34) and Proposition 4.24 (and its proof)]{CKW1}, for any $t,\rho>0$ and all $x,y\in M_0$,
\begin{equation}\label{e:trun-org}
p(t,x,y)\le q^{(\rho)}(t,x,y)+ \frac{c_1t}{V(x,\rho)\phi_j(\rho)}\exp\left(\frac{c_1t}{\phi(\rho)}\right),
\end{equation}
where $c_1$ is independent of $t,\rho>0$ and $x,y\in M_0.$

The following lemma is essentially taken from \cite[Lemma 4.3]{BKKL}, which is partly motivated by the proof of \cite[Proposition 5.3]{CKW1}. For the sake of the completeness and for further applications,
we spell out its proof here.

\begin{lemma}\label{L:selftail}
Assume that $\VD$, \eqref{polycon},
$\UHKD(\phi)$, $\J_{\phi_j,\le}$ and $\E_\phi$ hold. Let $f:\bR_+\times \bR_+\to \bR_+$ be a measurable function such that $t\mapsto f(r,t)$ is non-increasing for all $r>0$, and that $r\mapsto f(r,t)$ is non-decreasing for all $t>0$. Suppose that the following hold:
\begin{itemize}
\item[\rm (i)] For each $b>0$, $\sup_{t>0} f(b\phi^{-1}(t),t)<\infty$.
\item [\rm (ii)] There exist constants $\eta\in (0,\beta_{1,\phi_j}]$ and $a_1,c_1>0$ such that for all $x\in M_0$ and $r,t>0$,
$$\int_{B(x,r)^c} p(t,x,y)\,\mu(dy)\le c_1\left(\frac{\phi_j^{-1}(t)}{r}\right)^\eta+c_1\exp\left(-a_1 f(r,t)\right).$$\end{itemize}
Then  there exist constants $k,c_0>0$ such that for all $x,y\in M_0$ and $t>0$,
\begin{align*}p(t,x,y)\le &\frac{c_0t}{V(x,d(x,y))\phi_j(d(x,y))}+\frac{c_0}{V(x,\phi^{-1}(t))}\left(1+\frac{d(x,y)}{\phi^{-1}(t)}\right)^{d_2}\exp\left(-a_1k f(d(x,y)/(16 k),t)\right).\end{align*}
Furthermore, the conclusion still holds true for any $t\in (0,T]$ or
$t\in [T,\infty)$ with some $T>0$, if assumptions $(i)$ and $(ii)$
above are restricted on the corresponding time interval.
\end{lemma}

\begin{proof} We only consider the case that $t\in (0,\infty)$, since the other cases can be treated similarly.
We first note that, by $\E_\phi$ and $\J_{\phi_j,\le}$, the Dirichlet form $(\sE, \sF)$ is conservative
by Proposition \ref{P:3.2}(ii).
For fixed $x_0\in M$,   set $B_s:=B(x_0,s)$ for all $s>0$. By the conservativeness of the Dirichlet form $(\sE, \sF)$,
the strong Markov property, assumption (ii) and the fact that $t\mapsto f(r,t)$ is non-increasing, we have that for any $x\in B_{r/4}\cap M_0$ and $t\ge 0$,
\begin{align*}
 \bP^x(\tau_{B_r}\le t)\le c_1\left(\frac{\phi_j^{-1}(t)}{r}\right)^\eta+c_1\exp\left(-a_1 f(r/4,t)\right);
\end{align*}
see the end of the  proof for Lemma \ref{Conserv}.  For any $\rho>0$ and any subset $D\subset M$, denote by $\tau_D^{(\rho)}=\inf\{t>0: X^{(\rho)}_t\notin D\}$,
 where $(X^{(\rho)}_t)_{t\ge0}$ is the symmetric Hunt process associated with the $\rho$-truncated Dirichlet form $(\sE^{(\rho)},\sF)$.
 By \cite[Lemma 7.8]{CKW1}, $\J_{\phi_j,\le}$ and Lemma \ref{intelem}, we have that for all $x\in B_{r/4}\cap M_0$ and $r>0$,
$$
\bP^x (\tau_{B_r}^{(r)}\le t ) \le c_1\left(\frac{\phi_j^{-1}(t)}{r}\right)^\eta+c_1\exp\left(-a_1 f(r/4,t)\right)+\frac{c_2t}{\phi_j(r)}=:\Phi(r,t).
$$
This together with \cite[Lemma 7.1]{CKW1}   yields that for all $t>0$, $x\in M_0$ and $k\ge 1$,
\begin{equation}\label{e:ietrat}
\int_{B(x,2kr)^c}q^{(\rho)}(t,x,y)\,\mu(dy)\le \Phi(r,t)^k.
\end{equation}

 Let $k=[(\beta_{2,\phi_j}+2d_2)/\eta]+1$. For any $x,y\in M_0$ and $t>0$ with $4k\phi^{-1}(t)\ge d(x,y)$,  it follows from  the assumption that
 $r\mapsto f(r,t)$ is non-decreasing and assumption (i) that $f(d(x,y)/(16 k),t)\le f(\phi^{-1}(t)/4,t)\le A<\infty.$
  Thus, for any $x,y\in M_0$ and $t>0$ with
 $4k\phi^{-1}(t)\ge d(x,y)$, according to $\UHKD(\phi)$ and $\VD$,
\begin{align*}p(t,x,y)\le &c_3\left(\frac{1}{V(x,\phi^{-1}(t))}+\frac{1}{V(y,\phi^{-1}(t))}\right)\\
\le& \frac{c_4e^{a_1kA}}{V(x,\phi^{-1}(t))}\left(1+\frac{d(x,y)}{\phi^{-1}(t)}\right)^{d_2} \exp\left(-a_1k f(d(x,y)/(16 k),t)\right) .\end{align*}
Next  we consider the case that $x,y\in M_0$ and $t>0$ with $4k\phi^{-1}(t)\le d(x,y)$. Letting $r=d(x,y)$ and $\rho=r/(4k)$, it holds
\begin{align*}q^{(\rho)}(t,x,y)&=\int_M q^{(\rho)}(t/2,x,z)q^{(\rho)}(t/2,z,y)\,\mu(dz)\\
&\le\left(\int_{B(x,r/2)^c}+\int_{B(y,r/2)^c}\right)q^{(\rho)}(t/2,x,z)q^{(\rho)}(t/2,z,y)\,\mu(dz)\\
&\le \frac{c_5}{V(x,\phi^{-1}(t))}\left(1+\frac{d(x,y)}{\phi^{-1}(t)}\right)^{d_2} \Phi(\rho, t/2)^k. \end{align*}
Here, the last inequality follows from \eqref{e:ietrat} and $\VD$ as well as
$$q^{(\rho)}(t,x,z)\le \frac{c_0}{V(x,\phi^{-1}(t))} e^{c_0t/\phi(\rho)}\le \frac{c_0 e^{c_0}}{V(x,\phi^{-1}(t))},\quad x,z\in M_0, t>0\textrm{ with }\phi^{-1}(t)\le \rho ,$$ which follows from \cite[Lemma 5.1]{CKW1} (based on $\UHKD(\phi)$ and $\E_\phi$) and \cite[Lemma 7.8]{CKW1} (based on $\J_{\phi_j,\le}$ and the fact that $\phi_j(r)\ge \phi(r)$ for all $r>0$).
Since $\rho\ge \phi^{-1}(t)\ge \phi_j^{-1}(t)$ and $k\beta_{1,\phi_j}\ge k\eta\ge \beta_{2,\phi_j}+2d_2$, by \eqref{polycon},
\begin{align*}\left(\frac{\phi_j^{-1}(t)}{\rho}\right)^{\eta k}+\left(\frac{t}{\phi_j(\rho)}\right)^k&\le c_6\left[\left(\frac{\phi_j^{-1}(t)}{\rho}\right)^{\beta_{2,\phi_j}+2d_2}+ \left(\frac{\phi_j^{-1}(t)}{\rho}\right)^{k\beta_{1,\phi_j}}\right]\le
c_7\left(\frac{\phi_j^{-1}(t)}{\rho}\right)^{\beta_{2,\phi_j}+2d_2}.
\end{align*}
Thus  for all $x,y\in M_0$ and $t,\rho>0$ with $\rho\ge
\phi^{-1}(t)\ge\phi_j^{-1}(t)$,
\begin{align*}&\frac{1}{V(x,\phi^{-1}(t))}\left(1+\frac{d(x,y)}{\phi^{-1}(t)}\right)^{d_2}\left[\left(\frac{\phi_j^{-1}(t)}{\rho}\right)^{\eta k}+\left(\frac{t}{\phi_j(\rho)}\right)^k\right]\\
&\le \frac{c_7}{V(x,\phi^{-1}(t))}\left(1+\frac{d(x,y)}{\phi^{-1}(t)}\right)^{d_2}\left(\frac{\phi_j^{-1}(t)}{\rho}\right)^{\beta_{2,\phi_j}+2d_2}\le  \frac{c_8}{V(x,\rho)}\left(\frac{\rho}{\phi_j^{-1}(t)}\right)^{2d_2}\left(\frac{\phi_j^{-1}(t)}{\rho}\right)^{\beta_{2,\phi_j}+2d_2}\\
&\le \frac{c_9 t}{V(x,\phi^{-1}_j(t))\phi_j(d(x,y))}, \end{align*}
where the second inequality follows from $\VD$ and the fact that
$\rho\ge \phi^{-1}(t)\ge \phi_j^{-1}(t)$,  while in the last inequality
we used \eqref{polycon}. Hence, for any $x,y\in M_0$ and $t>0$ with
$4k\phi^{-1}(t)\le d(x,y),$
\begin{align*}
 &q^{(\rho)}(t,x,y) \\
 &\le  \frac{c_{10}}{V(x,\phi^{-1}(t))}\left(1+\frac{d(x,y)}{\phi^{-1}(t)}\right)^{d_2}
   \left(\left(\frac{\phi_j^{-1}(t)}{\rho}\right)^{\eta k}+\left(\frac{t}{\phi_j(\rho)}\right)^k+\exp\left(-a_1k f(\rho/4,t)\right)\right)\\
& \le\frac{c_{11}t}{V(x,d(x,y))\phi_j(d(x,y))}
+\frac{c_{11}}{V(x,\phi^{-1}(t))}\left(1+\frac{d(x,y)}{\phi^{-1}(t)}\right)^{d_2}\exp\left(-a_1k
f(d(x,y)/(16 k),t)\right).
\end{align*}
The desired assertion now follows from
\eqref{e:trun-org}, $\J_{\phi_j,\le}$ and Lemma \ref{intelem}.
    \qed\end{proof}

\begin{proposition}\label{thm:ujeuhkds}
Under $\VD$ and \eqref{polycon}, if $\UHKD(\phi)$, $\J_{\phi_j,\le}$ and $\E_{\phi}$ hold, then we have  $\UHK(\phi_c, \phi_j)$.
\end{proposition}

To prove Proposition \ref{thm:ujeuhkds},
we will use Lemma \ref{L:selftail} and need the following lemma.

\begin{lemma}\label{L:lemmastep1}
Assume that $\VD$, \eqref{polycon},
$\UHKD(\phi)$, $\J_{\phi_j,\le}$ and $\E_\phi$ hold.
Then there exist constants $a, c>0$ and $N\in \bN$ such that for all $x,y\in M_0$ and $t>0$,
$$p(t,x,y)\le \frac{ct}{V(x,d(x,y))\phi_j(d(x,y))}+ \frac{c}{V(x,\phi^{-1}(t))}\exp\left(-\frac{a d(x,y)^{1/N}}{\phi^{-1}(t)^{1/N}}\right).$$
\end{lemma}

\begin{proof}
We   claim
that there exist $a_1,c_1>0$ such that for all $x\in M_0$ and $t,r>0$
 \begin{equation}\label{e:set1}\int_{B(x,r)^c}p(t,x,y)\,\mu(dy)\le c_1\left(\frac{\phi_j^{-1}(t)}{r}\right)^\eta+c_1\exp\left(-\frac{a_1 r^{1/N}}{\phi^{-1}(t)^{1/N}}\right),\end{equation} where $\eta=\beta_{1,\phi_j}-((d_2+\beta_{1,\phi_j})/N)\in (0,\beta_{1,\phi_j}]$ (by taking $N$ large enough).
 If \eqref{e:set1} holds, then the assertion follows from Lemma \ref{L:selftail} by taking $f(r,t)=(r/\phi^{-1}(t))^{1/N}.$

 When $r\le \phi^{-1}(t)$, \eqref{e:set1} holds trivially with $c_1=e^{a_1}$.
 So it suffices to consider
 the case that $r\ge \phi^{-1}(t)$. According to \eqref{e:trun-org} and Lemma \ref{P:truncated}, for any $t,\rho>0$ and $x,y\in M_0$,
  $$
 p(t,x,y)\le  \frac{c_2}{V(x,\phi^{-1}(t))}\left(1+\frac{d(x,y)}{\phi^{-1}(t)} \right)^{d_2}\exp\left(\frac{c_2t}{\phi(\rho)}-\frac{c_3d(x,y)}{\rho}\right)
   +\frac{c_2t}{V(x,\rho)\phi_j(\rho)}\exp\left(\frac{c_2t}{\phi(\rho)}\right).
$$
Take $\alpha\in (d_2/(d_2+\beta_{1,\phi_j}),1)$, and define
$$
\rho_n:=\rho_n(r,t)=2^{n\alpha} r^{1-1/N} \phi^{-1}(t)^{1/N},\quad n\in \bN\cup\{0\}.
$$
Since $r\ge \phi^{-1}(t)$, $\phi^{-1}(t)\le \rho_n\le 2^{n\alpha}r$ for all $n\in \bN\cup\{0\}$. In particular, $t\le \phi(\rho_n)$ for all $n\in \bN\cup\{0\}$.
Thus, for any $x\in M_0$ and $t,r>0$,
 \begin{align*}\int_{B(x,r)^c}p(t,x,y)\,\mu(dy)&=\sum_{n=0}^\infty\int_{B(x,2^{n+1}r)\backslash B(x,2^nr)} p(t,x,y)\,\mu(dy)\\
 &\le \frac{c_2e^{c_2}}{V(x,\phi^{-1}(t))}\sum_{n=0}^\infty \left(1+\frac{2^{n+1}r}{\phi^{-1}(t)}\right)^{d_2}\exp\left(-\frac{c_32^nr}{\rho_n}\right)V(x,2^{n+1}r)\\
 &\quad+ c_2e^{c_2}t\sum_{n=0}^\infty \frac{V(x,2^{n+1}r)}{V(x,\rho_n)\phi_j(\rho_n)}\\
 &=:I_1+I_2. \end{align*} On the one hand, by the definition of $\rho_n$ and $\VD$,
 \begin{align*}I_1&=\frac{c_2e^{c_2}}{V(x,\phi^{-1}(t))}\sum_{n=0}^\infty\left(1+\frac{2^{n+1}r}{\phi^{-1}(t)}\right)^{d_2}\exp\left(-c_32^{n(1-\alpha)}
 \left(\frac{r}{\phi^{-1}(t)}\right)^{1/N}\right)V(x,2^{n+1}r)\\ &\le  c_4\sum_{n=0}^\infty\left(1+\frac{2^{n+1}r}{\phi^{-1}(t)}\right)^{2d_2}\exp\left(-c_32^{n(1-\alpha)}\left(\frac{r}{\phi^{-1}(t)}\right)^{1/N}\right)\\
 &\le c_5\sum_{n=0}^\infty \exp\left(-\frac{c_3}{2}2^{n(1-\alpha)}\left(\frac{r}{\phi^{-1}(t)}\right)^{1/N}\right)\le c_6\exp\left(-c_7\left(\frac{r}{\phi^{-1}(t)}\right)^{1/N}\right),  \end{align*} where in the second inequality we used the fact that there is a constant $c_8>0$ such that for all $n\ge1$ and $s\ge1$,
 $$\left(1+ 2^{n+1}s\right)^{2d_2}\le c_8\exp\left(\frac{c_3}{2}2^{n(1-\alpha)}s^{1/N}\right).$$  On the other hand, according to  $\VD$ and \eqref{polycon},
 \begin{align*}I_2&\le c_9\sum_{n=0}^\infty\left(\frac{2^nr}{\rho_n}\right)^{d_2}\left(\frac{\phi_j^{-1}(t)}{\rho_n}\right)^{\beta_{1,\phi_j}}\\
 &\le c_{10}\sum_{n=0}^\infty 2^{n((1-\alpha)d_2-\alpha\beta_{1,\phi_j})}\left(\frac{\phi_j^{-1}(t)}{r}\right)^{\beta_{1,\phi_j}-((d_2+\beta_{1,\phi_j})/N)}\le c_{11} \left(\frac{\phi_j^{-1}(t)}{r}\right)^\eta,\end{align*} where the first and the second inequalities follow from the facts that $\phi(t)\le \phi_j(t)$ and so $\phi^{-1}(t)\ge \phi^{-1}_j(t)$, and  in the last inequality we used  $\eta=\beta_{1,\phi_j}-((d_2+\beta_{1,\phi_j})/N)$ and
 $\alpha\in (d_2/(d_2+\beta_{1,\phi_j}),1)$.

 Combining estimates for $I_1$ and $I_2$, we obtain \eqref{e:set1}. The proof is complete.
 \qed\end{proof}

Now, we   give the

\medskip

\noindent{\bf Proof of Proposition \ref{thm:ujeuhkds}.}\quad The proof is split into two cases.

(1) We first consider the case that $t\geq 1$.
  By $\UHKD(\phi)$, we only need to check the case that $x,y\in M_0$ and $t\ge1$ with $d(x,y)\ge \phi_j^{-1}(t)$.  First, according to Lemma \ref{L:lemmastep1}, there exist constants $a, c>0$ and $N\in \bN$ such that for all $x,y\in M_0$ and $t\ge1$,
$$p(t,x,y)\le \frac{ct}{V(x,d(x,y))\phi_j(d(x,y))}+ \frac{c}{V(x,\phi_j^{-1}(t))}\exp\left(-\frac{a d(x,y)^{1/N}}{\phi_j^{-1}(t)^{1/N}}\right).$$
Furthermore, we find by $\VD$ and \eqref{polycon} that
\begin{align} \frac{1}{V(x,\phi_j^{-1}(t))}\exp\left(-\frac{a d(x,y)^{1/N}}{\phi_j^{-1}(t)^{1/N}}\right)
&\le \frac{c_1}{V(x,d(x,y))}\left(\frac{d(x,y)}{\phi_j^{-1}(t)}\right)^{d_2}\exp\left(-\frac{a d(x,y)^{1/N}}{\phi_j^{-1}(t)^{1/N}}\right) \nonumber\\
&\le  \frac{c_2}{V(x,d(x,y))}\exp\left(-\frac{a}{2}\frac{ d(x,y)^{1/N}}{\phi_j^{-1}(t)^{1/N}}\right) \nonumber\\
&\le \frac{c_3}{V(x,d(x,y))}\left(\frac{\phi_j^{-1}(t)}{d(x,y)}\right)^{\beta_{2,\phi_j}} \nonumber\\
&\le \frac{c_3t}{V(x,d(x,y))\phi_j(d(x,y))},  \label{e:4.4}
\end{align}
where in the second and the third inequalities we used
the facts  that
$$
r^{d_2} \le c_4 e^{ar^{1/N}/2} \quad \hbox{and} \quad  e^{-ar^{1/N}/2}\le
c_5r^{-\beta_{2,\phi_j}} \quad \hbox{for } r\ge 1,
$$
 respectively. Hence  for
all $x,y\in M_0$ and $t\ge1$ with $d(x,y)\ge \phi_j^{-1}(t)$,
$$p(t,x,y)\le \frac{c_6t}{V(x,d(x,y))\phi_j(d(x,y))}.$$

We
claim that for any $t\ge1$ and $x,y\in M_0$ with $d(x,y)\ge \phi_j^{-1}(t)$,
$$\frac{1}{V(x,\phi_c^{-1}(t))}\exp\left(-c_7\frac{d(x,y)}{\bar \phi_c^{-1}(t/d(x,y))}\right)\le \frac{c_{11}t}{V(x,d(x,y))\phi_j(d(x,y))} .$$
Indeed, since $\phi_c(t)\ge \phi_j(t)$ for all $t\ge 1$, for any $t\ge1$ and $x,y\in M_0$ with $d(x,y)\ge \phi_j^{-1}(t)$,
\begin{align*}&\frac{1}{V(x,\phi_c^{-1}(t))}\exp\left(-c_7\frac{d(x,y)}{\bar \phi_c^{-1}(t/d(x,y))}\right)\\
&\le \frac{c_8}{V(x,d(x,y))}\left(1+\frac{d(x,y)}{\phi_c^{-1}(t)}\right)^{d_2}\exp\left(-c_7\frac{d(x,y)}{\bar \phi_c^{-1}(t/d(x,y))}\right)\\
&\le \frac{c_8}{V(x,d(x,y))}\left(1+\frac{\phi_c(d(x,y))}{t}\right)^{d_2/\beta_{1,\phi_c}}\exp\left(-c_9\left(\frac{\phi_c(d(x,y))}{t}\right)^{1/(\beta_{2,\phi_c}-1)}\right)\\
&\le \frac{c_{10}}{V(x,d(x,y))}\exp\left(-\frac{c_9}{2}\left(\frac{\phi_c(d(x,y))}{t}\right)^{1/(\beta_{2,\phi_c}-1)}\right)\\
&\le \frac{c_{10}}{V(x,d(x,y))}\exp\left(-\frac{c_9}{2}\left(\frac{\phi_j(d(x,y))}{t}\right)^{1/(\beta_{2,\phi_c}-1)}\right)\\
&\le \frac{c_{11}t}{V(x,d(x,y))\phi_j(d(x,y))},  \end{align*} where
in the first inequality we used $\VD$, the second inequality follows
from \eqref{polycon}, \eqref{e:1.12} and \eqref{e:effdiff}, and in the third and the
last inequalities we applied the following two inequalities
$$(1+r)^{d_2/\beta_{1,\phi_c}}\le c_{12} \exp\left(\frac{c_9}{2} r^{1/(\beta_{2,\phi_c}-1)}\right),\quad r\ge1$$
and $$\exp\left(-\frac{c_9}{2}r^{1/(\beta_{2,\phi_c}-1)}\right)\le
c_{13} r^{-1},\quad r\ge1,$$ respectively.
This establishes $\UHK(\phi_c, \phi_j)$ for the case that $t\ge 1$.

(2) We next consider the case of $t\leq 1$.
 It suffices to
 consider the case
 when
  $x,y\in M_0$ and $0<t\le 1$ with $d(x,y)\ge \phi_c^{-1}(t)$.
    By Lemma \ref{L:selftail}, it is enough to show
    that there exist constants $\eta\in (0,\beta_{1,\phi_j}]$ and $c_1,c_2>0$ such that for any $x\in M_0$, $0<t\le 1$ and $r>0$,
 \begin{equation}\label{e:diffjump1}\int_{B(x,r)^c}p(t,x,y)\,\mu(dy)\le c_1\left(\frac{\phi_j^{-1}(t)}{r}\right)^{\eta}+c_1\exp\left(-c_2\frac{r}{\bar \phi_c^{-1}(t/r)}\right).\end{equation} Indeed, by \eqref{e:diffjump1}
 and Lemma \ref{L:selftail} with $f(r,t)=r/\bar\phi_c^{-1}(t/r)$, we have that for any $x,y\in M_0$ and $0<t\le 1$,
 $$p(t,x,y)\le c_{3} \left(\frac{t}{V(x,d(x,y))\phi_j(d(x,y))}+\frac{1}{V(x,\phi_c^{-1}(t))}  \exp\left(- \frac{c_0^*d(x,y)}{\bar \phi_c^{-1}(t/d(x,y))}\right)\right).$$
The desired assertion follows.
In the following, we will prove \eqref{e:diffjump1}. For this, we will consider four different cases.

(i) If $r\le C_0\phi_c^{-1}(t)$ for some
constant $C_0>0$ whose exact valued will be determined
in the step (ii) below, then by the non-decreasing property of
$\bar \phi_c$, \eqref{e:1.12} and \eqref{e:effdiff},
 \begin{align*}\exp\left(c_2\frac{r}{\bar \phi_c^{-1}(t/r)}\right)\le &\exp\left(c_2\frac{C_0\phi_c^{-1}(t)}{\bar \phi_c^{-1}(t/(C_0\phi_c^{-1}(t)))}\right)\le e^{c_4}.\end{align*} Hence, \eqref{e:diffjump1} holds trivially by taking $c_1=e^{c_4}.$

(ii) Suppose that $C_0\phi_c^{-1}(t)\le r\le r_*(t)$, where $r_*(t)$
is to be determined later. For any $n\in \bN\cup\{0\}$, define
$$\rho_n=c_*2^{n\alpha}\bar\phi_c^{-1}\left(\frac{ t}{r}\right)$$ with $d_2/(d_2+\beta_{1,\phi_j})<\alpha<1$, where $c_*$ is also determined later.  Then by
 \eqref{e:trun-org} and Lemma  \ref{P:truncated}, for any $x,y\in M_0$ with
$2^nr\le d(x,y)\le 2^{n+1}r$ and $0<t\le 1$,
\begin{align*}p(t,x,y)\le & \frac{c_{1}}{V(x,\phi_c^{-1}(t))}\left(1+\frac{2^{n+1}r}{\phi_c^{-1}(t)}\right)^{d_2}\exp\left(\frac{c_{1} t}{\phi(\rho_n)}-\frac{c_{2} 2^nr}{\rho_n}\right)+\frac{c_{1} t}{V(x,\rho_n)\phi_j(\rho_n)}\exp\left(\frac{c_{1}t}{\phi(\rho_n)}\right)\\
\le &\frac{c_{3}}{V(x,\phi_c^{-1}(t))}\left(1+\frac{2^{n+1}r}{\phi_c^{-1}(t)}\right)^{d_2}\exp\left(\frac{c_{1} t}{\phi_c(\rho_n)}-\frac{c_{2} 2^nr}{\rho_n}\right)+\frac{c_{3} t}{V(x,\rho_n)\phi_j(\rho_n)}\exp\left(\frac{c_{1}t}{\phi_c(\rho_n)}\right),\end{align*}
where in the last inequality we used the fact that if $\rho_n\ge1$, then $\phi(\rho_n)=\phi_j(\rho_n)\ge \phi_j(1)=1\ge t$ for all $t\in (0,1]$. Hence, for any $x\in M_0$, $0<t\le
1$ and $C_0\phi_c^{-1}(t)\le r\le r_*(t)$,
\begin{align*}&\int_{B(x,r)^c} p(t,x,y)\,\mu(dy)\\
&=\sum_{n=0}^\infty \int_{B(x,2^{n+1}r)\backslash B(x,2^nr)}p(t,x,y)\,\mu(dy)\\
&\le\sum_{n=0}^\infty\frac{c_{4}
V(x,2^{n+1}r)}{V(x,\phi_c^{-1}(t))}\left(1+\frac{2^{n+1}r}{\phi_c^{-1}(t)}\right)^{d_2}\exp\left(\frac{c_5t}{c_*^{\beta_{1,\phi_c}}2^{n\alpha\beta_{1,\phi_c}} \phi_c(\bar \phi_c^{-1}(t/r))}-\frac{c_22^{n(1-\alpha)} r}{c_*\bar\phi_c^{-1}(t/r)}\right) \\
&\quad+\sum_{n=0}^\infty \frac{c_{4} t V(x,2^{n+1}r)}{V(x,\rho_n)\phi_j(\rho_n)}\exp\left(\frac{c_5t}{c_*^{\beta_{1,\phi_c}}2^{n\alpha\beta_{1,\phi_c}} \phi_c(\bar \phi_c^{-1}(t/r))}\right)\\
&=:I_1+I_2,  \end{align*} where the inequality above follows from \eqref{polycon}, and
$c_1,\cdots, c_5$
are independent of $c_*$.
On the one hand, by taking
$c_*=(1+(2c^*_2c_5/c_2))^{1/(\beta_{1,\phi_c}-1)}$ (with $c^*_2$ being the constant $c_2$ in \eqref{e:1.12}), we find by $\VD$ and \eqref{e:1.12} that
\begin{align*}I_1\le & c_{6}\sum_{n=0}^\infty \left(\frac{2^nr}{\phi_c^{-1}(t)}\right)^{2d_2} \exp\left(-c_{7}2^{n(1-\alpha)}\frac{r}{\bar\phi_c^{-1}(t/r)}\right)\le c_{8}\exp\left(-c_{9}\frac{r}{\bar\phi_c^{-1}(t/r)}\right). \end{align*} Here and in what follows, the constants will depend on $c_*$.  On the other hand, according to $\VD$ and \eqref{polycon} again,
\begin{align*}I_2&\le c_{10}\exp\left(\frac{c_{11}r}{\bar \phi_c^{-1}(t/r)}\right)\sum_{n=0}^\infty \left(\frac{2^n r}{\rho_n}\right)^{d_2} \frac{\phi_j(r)}{\phi_j(\rho_n)}\frac{t}{\phi_j(r)}\\
&\le c_{12}\exp\left(\frac{c_{11}r}{\bar \phi_c^{-1}(t/r)}\right)\left(\frac{r}{\bar\phi_c^{-1}(t/r)}\right)^{d_2+\beta_{2,\phi_j}}\frac{t}{\phi_j(r)}\sum_{n=0}^\infty
2^{n(d_2 -\alpha (d_2+\beta_{1,\phi_j}))}\\
&\le c_{13}\exp\left(\frac{c_{14}r}{\bar \phi_c^{-1}(t/r)}\right)\frac{t}{\phi_j(r)}, \end{align*} where in the last inequality we used the condition $d_2/(d_2+\beta_{1,\phi_j})<\alpha<1$ and the fact that $$r^{d_2+\beta_{2,\phi_j}}\le  c_{15} e^r,\quad r\ge c_{16}>0.$$

We note
that the argument up to here is
independent of the definition of $r_*(t)$ and the choice of the
constant $C_0$. Now, according to Lemma \ref{L:keyexist} below, we
can find constants $C_0, c_{17}>0$
and a unique $r_*(t)\in (C_0\phi_c^{-1}(t),\infty)$ such that
\begin{equation}\label{e:constant1-}\exp\left(\frac{2c_{14}r}{\beta_{1,\phi_j}\bar \phi_c^{-1}(t/r)}\right)
\le \frac{c_{17} r}{\phi_j^{-1}(t)} ,\quad C_0\phi_c^{-1}(t)\le r\le
r_*(t)\end{equation} and
$$\exp\left(\frac{2c_{14}r}{\beta_{1,\phi_j}\bar
\phi_c^{-1}(t/r)}\right)\ge \frac{c_{17} r}{\phi_j^{-1}(t)},\quad
r\ge r_*(t), $$ as well as
\begin{equation}\label{e:constant1--}\exp\left(\frac{2c_{14}r_*(t)}{\beta_{1,\phi_j}\bar
\phi_c^{-1}(t/r_*(t))}\right)= \frac{c_{17}
r_*(t)}{\phi_j^{-1}(t)}.\end{equation}
(Here,   without loss of generality we may and do  assume that
${2c_{14}} \geq {\beta_{1,\phi_j}}$.)
 Then due to
\eqref{polycon} again,
$$I_2\le c_{18}\left(\frac{r}{\phi_j^{-1}(t)}\right)^{\beta_{1,\phi_j}/2}\left(\frac{t}{\phi_j(r)}\right)\le c_{19}\left(\frac{\phi_j^{-1}(t)}{r}\right)^{\beta_{1,\phi_j}/2}.$$
Putting $I_1$ and $I_2$ together, we
obtain
\eqref{e:diffjump1}.

Next we estimate $r_*(t)$ from above and below since they are needed in steps (iii) and (iv).
We first consider the lower bound for $r_*(t)$.
By
\eqref{polycon}, \eqref{e:1.12} and \eqref{e:effdiff}, we have
\begin{equation}\label{e:lowerbounduse}\exp\left(\frac{2c_{14} r}{\beta_{1,\phi_j}\bar\phi_c^{-1}(t/r)}\right)
\le
\exp\left(c_{20}\left(\frac{r}{\phi_c^{-1}(t)}\right)^{\beta_{2,\phi_c}/(\beta_{1,\phi_c}-1)}\right).\end{equation}
Hence, \eqref{e:constant1-} holds if
$$\exp\left(c_{20}\left(\frac{r}{\phi_c^{-1}(t)}\right)^{\beta_{2,\phi_c}/(\beta_{1,\phi_c}-1)}\right)\le \frac{c_{17} r}{\phi_j^{-1}(t)},\quad
C_0\phi_c^{-1}(t)\le r\le r_*(t);$$
namely,
\begin{align*}&\log c_{17}+  \log \frac{r}{\phi_c^{-1}(t)}+  \log \frac{\phi_c^{-1}(t)}
{\phi_j^{-1}(t)}\ge
c_{20}\left(\frac{r}{\phi_c^{-1}(t)}\right)^{\beta_{2,\phi_c}/(\beta_{1,\phi_c}-1)}\end{align*}
holds for all $ C_0\phi_c^{-1}(t)\le r\le r_*(t)$.
Hence, we have
$$r_*(t)\ge C_1
\phi_c^{-1}(t)\log^{(\beta_{1,\phi_c}-1)/\beta_{2,\phi_c}}(\phi^{-1}_c(t)/\phi^{-1}_j(t))$$
for some constant $C_1>0$ which is independent of $t$.  For the upper
bound of $r_*(t)$, similar to the argument for
\eqref{e:lowerbounduse}, we have
$$\exp\left(\frac{2c_{14} r}{\beta_{1,\phi_j}\bar\phi_c^{-1}(t/r)}\right)
\ge
\exp\left(c_{21}\left(\frac{r}{\phi_c^{-1}(t)}\right)^{\beta_{1,\phi_c}/(\beta_{2,\phi_c}-1)}\right).$$
Hence, we have $$r_*(t)\le C_2
\phi_c^{-1}(t)\log^{(\beta_{2,\phi_c}-1)/\beta_{1,\phi_c}}(\phi^{-1}_c(t)/\phi^{-1}_j(t)),$$
where $C_2>0$ is also independent of $t$.

 (iii)
Suppose that $r\ge r^*(t):=C_3
\phi_c^{-1}(t)\log^{N}(\phi^{-1}_c(t)/\phi^{-1}_j(t))$, where $N\in
\bN$ is given in Lemma \ref{L:lemmastep1}, and $C_3>0$ is determined
later.
 According to Lemma \ref{L:lemmastep1}, there exist constants $a, c>0$ and $N\in \bN$ such that for all $x,y\in M_0$ and $0<t\le 1$ with $d(x,y)\ge r^*(t)$,
\begin{align*}p(t,x,y)\le &\frac{ct}{V(x,d(x,y))\phi_j(d(x,y))}+ \frac{c}{V(x,\phi_c^{-1}(t))}\exp\left(-\frac{a d(x,y)^{1/N}}{\phi_c^{-1}(t)^{1/N}}\right)\\
\le&\frac{ct}{V(x,d(x,y))\phi_j(d(x,y))}+ \frac{c_1}{V(x,d(x,y))}\left(\frac{d(x,y)}{\phi_c^{-1}(t)}\right)^{d_2}\exp\left(-\frac{a d(x,y)^{1/N}}{\phi_c^{-1}(t)^{1/N}}\right)\\
\le& \frac{ct}{V(x,d(x,y))\phi_j(d(x,y))}+ \frac{c_2}{V(x,d(x,y))}\exp\left(-\frac{a}{2}\frac{ d(x,y)^{1/N}}{\phi_c^{-1}(t)^{1/N}}\right),\end{align*} where in the second inequality we used $\VD$ and the last inequality follows from the fact that
$$r^{d_2}\le c_3 e^{ar^{1/N}/2},\quad r\ge c_4>0.$$
Without loss of generality, we may and do assume that $N\ge
(\beta_{2,\phi_c}-1)/\beta_{1,\phi_c}$. In particular, $r^*(t)\ge
r_*(t)$ by the upper bond for $r_*(t)$ mentioned at the end of step
(ii) and also by choosing $C_3>0$ large enough in the definition of
$r^*(t)$ if necessary.

 Next, suppose
that there is a constant $C_3>0$ in the definition of $r^*(t)$ such
that
\begin{equation}\label{e:constant1+}\exp\left(\frac{a}{2}\frac{r^{1/N}}{\phi_c^{-1}(t)^{1/N}}\right)\ge\frac{ c_5 \phi_j(r)}{t},\quad r\ge r^*(t)\end{equation}
holds for some $c_5>0$. Then,
for all $x,y\in M_0$ and $0<t\le 1$ with $d(x,y)\ge r^*(t)$,
we have
$$p(t,x,y)\le \frac{c_6t}{V(x,d(x,y))\phi_j(d(x,y))}.$$ Hence, by  \eqref{polycon} and Lemma \ref{intelem}, for all $x\in M_0$, $t\in (0,1]$ and $r\ge r^*(t)$,
$$
\int_{B(x,r)^c}p(t,x,y)\,\mu(dy)\le c_6 \int_{B(x,r)^c}\frac{t}{V(x,d(x,y))\phi_j(d(x,y))}\,\mu(dy) \le  \frac{c_7t}{\phi_j(r)}\le
c_{8}\left(\frac{\phi_j^{-1}(t)}{r}\right)^{\beta_{1,\phi_j}},
$$
proving \eqref{e:diffjump1}.

Finally, we verify that  \eqref{e:constant1+} indeed holds
by using  the idea of the argument for the lower bound of
$r_*(t)$ in the end of step (ii).
By \eqref{polycon},
$$ \frac{ c_5 \phi_j(r)}{t}\le c_9 \left(\frac{r}{\phi_j^{-1}(t)}\right)^{\beta_{2,\phi_j}}.
$$
Hence
\eqref{e:constant1+} is a consequence of the following inequality
$$\exp\left(\frac{a}{2}\frac{r^{1/N}}{\phi_c^{-1}(t)^{1/N}}\right)\ge  c_9 \left(\frac{r}{\phi_j^{-1}(t)}\right)^{\beta_{2,\phi_j}},\quad r\ge r^*(t);$$ that is,
$$ \frac{a}{2}\left(\frac{r}{\phi_c^{-1}(t)}\right)^{1/N}\ge \log c_9+ \beta_{2,\phi_j} \log \frac{r}{\phi_c^{-1}(t)}
+\beta_{2,\phi_j}\log \frac{\phi_c^{-1}(t)}{\phi_j^{-1}(t)},\quad
r\ge r^*(t).$$
The above inequality  clearly is true by  a suitable choice of $C_3>0$
so  that \eqref{e:constant1+}
holds true.

(iv) Let $r\in [r_*(t), r^*(t)]$. For any $x\in M_0$, $0<t\le 1$ and $r\in [r_*(t),r^*(t)]$, we find by the conclusion in step (ii) that
\begin{align*}\int_{B(x,r)^c}p(t,x,y)\,\mu(dy)&\le \int_{B(x,r_*(t))^c} p(t,x,y)\,\mu(dy)\\
&\le c_1\left(\frac{\phi_j^{-1}(t)}{r_*(t)}\right)^{\beta_{1,\phi_j}/2}+c_1\exp\left(-c_2\frac{r_*(t)}{\bar\phi_c^{-1}(t/r_*(t))}\right)\\
&=:I_1+I_2.\end{align*}
It follows from
$r_*(t)\ge C_1 \phi_c^{-1}(t)\log^{(\beta_{1,\phi_c}-1)/\beta_{2,\phi_c}}(\phi^{-1}_c(t)/\phi^{-1}_j(t))$ and $\phi_j^{-1}(t)\le \phi_c^{-1}(t)$ that
\begin{align*}I_1&\le c_3\left(\frac{\phi_j^{-1}(t)}{\phi_c^{-1}(t)\log^{(\beta_{1,\phi_c}-1)/\beta_{2,\phi_c}}(\phi^{-1}_c(t)/\phi^{-1}_j(t))}\right)^{\beta_{1,\phi_j}/2}\\
&\le c_3\left(\frac{\phi_j^{-1}(t)}{\phi_c^{-1}(t)\log^N(\phi^{-1}_c(t)/\phi^{-1}_j(t))}\right)^{((\beta_{1,\phi_c}-1)\beta_{1,\phi_j})/(2N\beta_{2,\phi_c})}\le c_4
\left(\frac{\phi_j^{-1}(t)}{r}\right)^{((\beta_{1,\phi_c}-1)\beta_{1,\phi_j})/(2N\beta_{2,\phi_c})}.
\end{align*} On the other hand, without loss of generality we may and do assume that $c_2\in (0,1)$ in the term $I_2$. By
\eqref{e:constant1--} and the fact that we can assume in
\eqref{e:constant1--} that
$ {2c_{14}} \geq {\beta_{1,\phi_j}}$,
there is a constant $\theta\in (0,1)$ such that
\begin{align*}I_2&\le c_5\left(\frac{\phi_j^{-1}(t)}{r_*(t)} \right)^{\theta}\le c_6\left(\frac{\phi_j^{-1}(t)}{\phi_c^{-1}(t)\log^{(\beta_{1,\phi_c}-1)/\beta_{2,\phi_c}}(\phi^{-1}_c(t)/\phi^{-1}_j(t))}\right)^{\theta}\\
&\le c_7\left(\frac{\phi_j^{-1}(t)}{\phi_c^{-1}(t)\log^N(\phi^{-1}_c(t)/\phi^{-1}_j(t))}\right)^{(\theta(\beta_{1,\phi_c}-1))/(N\beta_{2,\phi_c})}\le
c_8\left(\frac{\phi_j^{-1}(t)}{r}\right)^{(\theta(\beta_{1,\phi_c}-1))/(N\beta_{2,\phi_c})}.\end{align*}

Combining all the estimates above,
we get \eqref{e:diffjump1}
  with
$$\eta=\min\{((\beta_{1,\phi_c}-1)\beta_{1,\phi_j})/(2N\beta_{2,\phi_c}),
(\theta(\beta_{1,\phi_c}-1))/(N\beta_{2,\phi_c})\}.$$
This completes the proof.   \qed

The following lemma has been used in the proof above.

\begin{lemma}\label{L:keyexist}For any $C_*>0$, there exist constants $C_0,C^*>0$ such that
\begin{itemize}
\item[\rm (i)] for any $t\in (0,1]$ and $r\ge C_0\phi_c^{-1}(t)$, the function
$$r\mapsto F_{1,t}(r):=\frac{\exp\left(C_*r/\bar\phi_c^{-1}(t/r)\right)}{r/\bar\phi_c^{-1}(t/r)}$$ is strictly increasing.

\item[\rm (ii)] for any $t\in (0,1]$, there is a unique $r_*(t)\in (C_0\phi_c^{-1}(t),\infty)$ such
that $ F_{1,t}(r_*(t))=F_{2,t}(r_*(t)),$
$$ F_{1,t}(r)<F_{2,t}(r),\quad r\in (C_0\phi_c^{-1}(t), r_*(t))$$
and $$F_{1,t}(r)>F_{2,t}(r),\quad r\in ( r_*(t),\infty),$$ where
$$F_{2,t}(r)=\frac{C^*\bar\phi_c^{-1}(t/r)}{\phi_j^{-1}(t)}.$$
\end{itemize}
\end{lemma}

\begin{proof} (i) We know from \eqref{polycon} and \eqref{e:effdiff} that, if $r\ge C_0\phi_c^{-1}(t)$ with $C_0$ large enough, then
$$\frac{r}{\bar\phi_c^{-1}(t/r)}\ge c_1 \left(\frac{\phi_c(r)}{t}\right)^{1/(\beta_{2,\phi_c}-1)}\ge c_2\left(\frac{r}{\phi_c^{-1}(t)}\right)
^{\beta_{1,\phi_c}/(\beta_{2,\phi_c}-1)}\ge c_2
C_0^{\beta_{1,\phi_c}/(\beta_{2,\phi_c}-1)},$$ where the constants
$c_1$ and $c_2$ are independent of $C_0$. Note further that the
function $r\mapsto r/\bar\phi_c^{-1}(t/r)$ is strictly increasing
due to the strictly increasing property of the function
$\bar\phi_c(r)$ on $\bR_+$. Then the first required assertion
follows from the fact that
the function $s\mapsto s^{-1}{e^{C_*s}}$ is strictly increasing on $[c_2
C_0^{\beta_{1,\phi_c}/(\beta_{2,\phi_c}-1)},\infty)$, by choosing
$C_0$ (depending on $C_*$) large enough.

(ii) We fix $t\in (0,1]$. As seen from (i) and its proof that the function $F_{1,t}(r)$
is strictly increasing on $[C_0\phi_c^{-1}(t),\infty)$ with
$F_{1,t}(\infty)=\infty$. On the other hand, due to the strictly
increasing property of the function $\bar\phi_c(r)$ on $\bR_+$ and
\eqref{e:effdiff}, we know that the function $F_{2,t}(r)$ is
strictly decreasing on $[C_0\phi_c^{-1}(t),\infty)$ with
$F_{2,t}(C_0\phi_c^{-1}(t))=\frac{C^*\bar\phi_c^{-1}(t/(C_0\phi_c^{-1}(t)))}{\phi_j^{-1}(t)}$
and $F_{2,t}(\infty)=0.$

Furthermore, according to \eqref{polycon} and \eqref{e:effdiff}
again, we can obtain that
$F_{1,t}(C_0\phi_c^{-1}(t))\le c_3$ and
$$F_{2,t}(C_0\phi_c^{-1}(t))=\frac{C^*\bar\phi_c^{-1}(t/(C_0\phi_c^{-1}(t)))}{\phi_j^{-1}(t)}\ge C^*\bar\phi_c^{-1}(t/(C_0\phi_c^{-1}(t))) \ge
C^*c_4,$$ where $c_3, c_4$ are independent of $C^*$.

Combining with both conclusions above and taking $C^*=2c_3/c_4$, we
then prove the second desired assertion. \qed
\end{proof}

\begin{remark}\label{Heat:remark}\rm
Assume that $\VD$, \eqref{polycon},
$\UHKD(\phi)$, $\J_{\phi,\le}$ and $\E_\phi$ hold.
By considering the cases of $d(x, y) \leq \phi^{-1}(t)$ and $d(x, y) \ge \phi^{-1}(t)$ separately and using similar argument as those for Lemma \ref{L:lemmastep1}
for the second case, we have
$$
\frac{1}{V(x,\phi^{-1}(t))}\exp\left(-\frac{a d(x,y)^{1/N}}{\phi^{-1}(t)^{1/N}}\right)
\leq   c_1\left(\frac{1}{V(x,\phi^{-1}(t))} \wedge \frac{t}{V(x,d(x,y))\phi(d(x,y))}\right).
$$
Thus
it follows from Lemma \ref{L:lemmastep1}
 and $\UHKD(\phi)$
 that
\begin{equation}\label{heat:remark1}
p(t,x,y)\le c_2\left(\frac{1}{V(x,\phi^{-1}(t))} \wedge \frac{t}{V(x,d(x,y))\phi(d(x,y))}\right).
\end{equation}
Clearly, compared with $\UHK(\phi_c, \phi_j)$, \eqref{heat:remark1} is far from the optimality.
However, inequality \eqref{heat:remark1} is  useful in the derivation of  characterizations of parabolic Harnack inequalities in the next section.
For our later use, we say  $\UHK_{weak}(\phi)$  holds if the heat kernel satisfies the upper bound estimates \eqref{heat:remark1}.
\end{remark}

\medskip

\noindent{\bf Proof of Theorem \ref{T:main-1}.} The ${\rm (ii)}\Longrightarrow{\rm (i)}$ part follows from
Theorem \ref{thm:ujeuhkds} and
Proposition \ref{P:3.2}(ii).
By Propositions \ref{P:3.1},
\ref{l:jk} and \ref{T:gdCS}, we get
${\rm (i)}\Longrightarrow {\rm (iii)}$.  Clearly $ {\rm (iii)}\Longrightarrow  {\rm (iv)}$ by Proposition \ref{P:csp}, while $ {\rm (iv)}\Longrightarrow{\rm (ii)}$ follows from Propositions \ref{P:exit} and \ref{ehi-ukdd}.
This proves the theorem.
\qed

\section{Characterizations of two-sided heat kernel estimates}\label{S:5}

In this section, we will establish
stable
characterizations of two-sided heat kernel estimates. Since we have obtained characterizations for $\UHK(\phi_c, \phi_j)$ in Theorem \ref{T:main-1},
we will mainly be concerned with
lower bound estimates for heat kernel in this section.
We first need some definitions.

\begin{definition}\label{PER}
\rm \begin{description}
\item{(i)} We say that the
\emph{parabolic H\"older regularity}
($\PHR(\phi)$)
holds for the Markov process $X$ if there exist  constants $c>0$,
$\theta\in (0, 1]$ and $\eps \in (0, 1)$ such that for every $x_0 \in M $, $t_0\ge0$, $r>0$ and for
 every bounded measurable function $u=u(t,x)$ that is  caloric in
$Q(t_0,x_0,\phi(r), r)$,    there is a properly exceptional set ${\cal N}_u\supset {\cal N}$ so that
\be\label{e:phr}
|u(s,x)-u(t, y)|\le c\left( \frac{\phi^{-1}(|s-t|)+d(x, y)}{r} \right)^\theta  \esssup_{ [t_0, t_0+\phi (r)] \times M}|u|
\ee
 for  every
 $s, t\in (t_0+\phi(r)-\phi(\eps r), t_0+\phi (r))$ and $x, y \in B(x_0, \eps r)\setminus {\cal N}_u$.

\item{(ii)} We say that the \emph{elliptic H\"older regularity}
($\EHR$)
holds for the process $X$, if there exist constants $c>0$,
$\theta\in (0, 1]$ and $\eps \in (0, 1)$ such that  for every $x_0 \in M $, $r>0$ and for
 every bounded measurable function $u$ on $M$ that is harmonic in $B(x_0, r)$, there is a properly exceptional set ${\cal N}_u\supset {\cal N}$
so that
\be\label{e:ehr}
|u(x)-u(y)| \leq c   \left( \frac{d(x, y)}{r}\right)^\theta\esssup_{M}  |u| \ee
for any $x,$ $y \in B(x_0, \eps r)\setminus {\cal N}_u.$

\end{description}
\end{definition}

Clearly $\PHR(\phi) \Longrightarrow \EHR$. Note that in the definition of $\PHR(\phi)$ (resp. $\EHR$) if the inequality \eqref{e:phr} (resp. \eqref{e:ehr}) holds for some $\varepsilon\in(0,1)$, then it holds for all $\eps \in(0,1)$ (with possibly different constant $c$). We take $\EHR$ for example. For every $x_0 \in M $ and $r>0$, let $u$ be a bounded function on $M$ such that it is harmonic in $B(x_0, r)$. Then for any $\eps' \in (0, 1)$
    and $x\in B(x_0,\eps'r)\setminus {\cal N}_u$, $u$ is harmonic on $B(x,(1-\eps')r)$. Applying \eqref{e:ehr} for $u$ on
    $B(x,(1-\eps')r))$, we find that for any $y\in B(x_0,\eps'r)\setminus {\cal N}_u$ with $d(x,y)\le (1-\eps')\eps r$, $$|u(x)-u(y)|\le c \left( \frac{d(x, y)}{r}\right)^\theta\esssup_{z\in M}  |u(z)|.$$ This implies that for any $x,y\in B(x_0,\eps'r)\setminus {\cal N}_u$, \eqref{e:ehr} holds with $c'=c\vee \frac{2}{[(1-\eps')\eps]^\theta}.$

\subsection{$\PI(\phi)+ \J_{\phi_j} +
\CS(\phi)\Longrightarrow \HK_- (\phi_c, \phi_j)$}\label{subsection:lower}

\begin{proposition}\label{L:log-l}
 Let $B_r=B(x_0,r)$ for some $x_0\in M$ and $r>0$.
Assume that $u\in \sF^{B_R}_{loc}$ is a bounded and superharmonic function in a ball
$B_R$ such that $u\ge 0$ on $B_R$. If $\VD$, \eqref{polycon}, $\CS(\phi)$ and $\J_{\phi,
\le}$ hold, then for any $l>0$ and $0<2r\le R$,
\begin{align*}&\int_{B_r}\,d \Gamma_c(\log (u+l),\log (u+l))+\int_{B_r\times B_r}\left[\log  \Big(\frac{u(x)+l}{u(y)+l} \Big)\right]^2\, J(dx,dy)\\
&\le \frac{c_1V(x_0,r)}{\phi(r)}\bigg(1+\frac{\phi(r)}{\phi(R)}\frac{\T_\phi\,(u_-; x_0,R)}{l}\bigg),\end{align*}
 where $\T_\phi\,(u_-;x_0,R)$ is the nonlocal tail of $u_-$ with respect to $\phi(r)$ in $B(x_0,R)$ defined as \eqref{def-T}, and  $c_1$ is a constant independent of $u$, $x_0$, $r$, $R$ and $l$.
 \end{proposition}

\begin{proof}
For pure jump  Dirichlet forms, a similar statement
is given in
\cite[Proposition 4.11]{CKW2}. In the present setting, we need to take into account on Dirichlet forms with both local and non-local terms.
 According to $\CS(\phi)$, $\J_{\phi, \le}$ and Proposition \ref{CSJ-equi},
we can choose $\varphi\in \sF^{B_{3r/2}}$ related
to $\mbox{Cap} (B_r,B_{3r/2})$ so that
\begin{equation}\label{cap-estimate-1}
\sE(\varphi,\varphi)\le 2  \mbox{Cap} (B_r,B_{3r/2}) \le
\frac{c_1V(x_0,r)}{\phi(r)}.
 \end{equation}

 Let $u$ be  a bounded superharmonic function in a ball $B_R$. As
$\frac{\varphi^2}{u+l}\in \sF^{B_{3r/2}}$ for any $l>0$,
\begin{equation}\label{cap-estimate-2}
\sE\left(u,\frac{\varphi^2}{u+l}\right) \geq 0.
\end{equation}
By the proof of \cite[Proposition 4.11]{CKW2},
\begin{align*}
\sE^{(j)}\left(u,\frac{\varphi^2}{u+l}\right)\le& \sE^{(j)}(\varphi,\varphi)+
\frac{c_2V(x_0,r)}{\phi(R)l}\T_\phi\, (u_-; x_0,R)-\int_{B_r\times B_r}\left[\log  \Big(\frac{u(x)+l}{u(y)+l} \Big)\right]^2\, J(dx,dy).\end{align*} That is,
\begin{align*}\int_{B_r\times B_r}\left[\log  \Big(\frac{u(x)+l}{u(y)+l} \Big)\right]^2\, J(dx,dy)\le &\sE^{(j)}(\varphi,\varphi)+
\frac{c_2V(x_0,r)}{\phi(R)l}\T\, (u_-; x_0,R)-\sE^{(j)}\left(u,\frac{\varphi^2}{u+l}\right).
\end{align*}
 On the other hand, by the Leibniz and chain rules and the Cauchy-Schwarz inequality \eqref{e:sc-ine},
 \begin{align*}
&\int \vp^2\,d \Gamma_c(\log (u+l),\log (u+l))\\
 &= -\int \vp^2\,d\Gamma_c\left(u+l, \frac{1}{u+l}\right)= -\int \vp^2\,d\Gamma_c\left(u, \frac{1}{u+l}\right)\\
 &=-\int d\Gamma_c\left(u,\frac{\varphi^2}{u+l}\right)+2\int \frac{\vp}{u+l}\,d\Gamma_c(u,\vp)\\
 &\le-\int d\Gamma_c\left(u,\frac{\varphi^2}{u+l}\right) +2\int \,d\Gamma_c(\vp,\vp)+\frac{1}{2}\int \frac{\vp^2}{(u+l)^2}\,d\Gamma_c(u,u)\\
 &=- \int d\Gamma_c\left(u,\frac{\varphi^2}{u+l}\right) +2\int \,d\Gamma_c(\vp,\vp)+\frac{1}{2}\int \vp^2\,d\Gamma_c(\log (u+1),\log (u+1)). \end{align*} Namely, \begin{align*}&\int \vp^2\,d \Gamma_c(\log (u+l),\log (u+l))\le -2 \int d\Gamma_c\left(u,\frac{\varphi^2}{u+l}\right) +4\int \,d\Gamma_c(\vp,\vp).
\end{align*}
Hence,
$$
\int_{B_r}\,d \Gamma_c(\log (u+l),\log (u+l))\le 4 \sE^{(c)}(\vp,\vp)-2\sE^{(c)}\left(u,\frac{\varphi^2}{u+l}\right).
$$
Putting both estimates together, we conclude that
 \begin{align*}&\int_{B_r}\,d \Gamma_c(\log (u+l),\log (u+l))+\int_{B_r\times B_r}\left[\log  \Big(\frac{u(x)+l}{u(y)+l} \Big)\right]^2\, J(dx,dy)\\
 &\le \int_{B_r}\,d \Gamma_c(\log (u+l),\log (u+l))+2\int_{B_r\times B_r}\left[\log  \Big(\frac{u(x)+l}{u(y)+l} \Big)\right]^2\, J(dx,dy)\\
 &\le 2\sE^{(j)}(\varphi,\varphi)+ 4 \sE^{(c)}(\vp,\vp)+
\frac{2c_2V(x_0,r)}{\phi(R)l}\T_\phi\, (u_-; x_0,R)-2\sE\left(u,\frac{\varphi^2}{u+l}\right)\\
  &\le 4\sE(\vp,\vp)+\frac{2c_2V(x_0,r)}{\phi(R)l}\T_\phi\, (u_-; x_0,R)-2\sE\left(u,\frac{\varphi^2}{u+l}\right)\\
  &\le \frac{c_5V(x_0,r)}{\phi(r)}\bigg(1+\frac{\phi(r)}{\phi(R)}\frac{\T_\phi\,(u_-; x_0,R)}{l}\bigg),\end{align*}
  where in the last inequality we used \eqref{cap-estimate-1} and \eqref{cap-estimate-2}.
  The proof is complete. \qed
\end{proof}

With  Proposition \ref{L:log-l}, we can follow the arguments for \cite[Corollary 4.12
and Propisition 4.13]{CKW2} to obtain the following.

\begin{proposition}\label{P:EHR}
Assume  $\VD$, $\RVD$ and \eqref{polycon}. Then
$$ \PI(\phi)+\J_{\phi,\le}+\CS(\phi)\Longrightarrow \EHR.$$
\end{proposition}

\medskip

\begin{proposition}\label{ejlhk} If $\VD$, $\RVD$, \eqref{polycon}, $\PI(\phi)$,
$\J_{\phi,\le}$ and
$\CS(\phi)$ hold, then we have $\NDL(\phi)$.
\end{proposition}

\begin{proof}
It follows from Proposition \ref{P:exit} that under $\VD$, $\RVD$ and \eqref{polycon},
$$
 \PI (\phi)+\J_{\phi,\le}+\CS(\phi)\Longrightarrow \E_\phi,
$$
where we also used the fact that $\PI(\phi)$ implies $\FK(\phi)$ by
Proposition \ref{P:3.1}.
With this and Proposition \ref{P:EHR}, the desired assertion essentially follows from the proof of \cite[Proposition 4.9]{CKW2}.
\qed \end{proof}

The following proposition establishes the
 $ {\rm (v)}\Longrightarrow{\rm (i)}$ part in Theorem \ref{T:main}.

 \begin{proposition}\label{ejlhk2} Suppose $\VD$, $\RVD$, \eqref{polycon}, $\PI(\phi)$, $\J_{\phi_j}$ and
 $\CS(\phi)$ hold. Then we have $\HK_- (\phi_c, \phi_j)$.
\end{proposition}

\begin{proof} According to Proposition
\ref{ejlhk}, we have $\NDL(\phi)$. In particular, for any $x,y\in M_0$ and $t>0$ with $d(x,y)\le c_0\phi^{-1}(t)$ for some constant $c_0>0$,
$$p(t,x,y)\ge  \frac{c_1}{V(x,\phi^{-1}(t))}.$$
On the other hand, under our assumptions, we can get from the arguments in step (ii) for the proof of \cite[Proposition 5.4]{CKW1} that for all $x,y\in M_0$ and $t>0$ with $d(x,y)\ge c_0\phi^{-1}(t)$,
$$
p(t,x,y)\ge \frac{c_2 t}{V(x,d(x,y))\phi_j(d(x,y))}.
$$
This establishes the heat kernel lower bound \eqref{e:1.31} for $\HK_- (\phi_c, \phi_j)$.
The upper bound of $\HK_- (\phi_c, \phi_j)$ follows from Proposition \ref{P:3.1}
and the equivalence between (i) and (iv) of Theorem \ref{T:main-1}.
 \qed\end{proof}

\subsection{From $\HK_- (\phi_c, \phi_j)$ to $\HK (\phi_c, \phi_j)$}
To consider $\HK (\phi_c, \phi_j)$,
we assume in addition
that $(M,d,\mu)$ is connected and satisfies the chain
condition (see the end of Remark \ref{R:1.22}\,(i)).
We  emphasize that
the results in the previous sections hold true without this additional assumption on the state space $(M,d,\mu)$.

\begin{proposition}\label{P:twosided}
Suppose that $(M,d,\mu)$ is connected and satisfies the chain condition.
Under  $\VD$ and \eqref{polycon},  $\HK_-(\phi_c, \phi_j)$
implies $\HK (\phi_c, \phi_j)$.
In particular, if $\VD$, $\RVD$, \eqref{polycon}, $\PI(\phi)$, $\J_{\phi_j}$ and
 $\CS(\phi)$ hold, then  so does $\HK (\phi_c, \phi_j)$.
\end{proposition}

\begin{proof}
By $\HK_- (\phi_c, \phi_j)$ and \eqref{e:1.25},
we only need to verify the case that $t\in (0,1]$ and $x,y\in M_0$ with $d(x,y)\ge c_0\phi_c^{-1}(t)$ for some constant $c_0>0$. The proof is based on the standard chaining argument, e.g.\ see the proof of \cite[Proposition 5.2(i)]{BGK}.

 First we assume $\HK_-(\phi_c, \phi_j)$ holds.
Then, by Remark \ref{R:1.22}(ii), $\NL (\phi)$ holds.
Furthermore, in view of Remark \ref{R:1.22}(iii), it suffices to consider the case that $t\in [0, 1]$.
Now, fix $t\in (0,1]$
and $x,y\in M_0$. Let $r=d(x,y)$. Since the space $(M,d,\mu)$ satisfies the chain condition, there exists a constant $C>0$ such that, for any $x,y\in
M$ and for any $n\in {\mathbb N}$, there exists a sequence
$\{x_{i}\}_{i=0}^{n}\subset M$ such that $x_{0}=x$, $x_{n}=y$ and
$
d(x_{i},x_{i+1})\leq C { d(x,y)}/{n}$ for all $i=0,1,\cdots,n-1$. In the following, we
set
$n=m:=m(t,r),$ where $m(t,r)$ is defined by
\eqref{e:1.22}. Define $r_n= C {r}/{(3n)}$. In particular, by \eqref{e:scdf},
\eqref{e:1.12} and \eqref{polycon}, $r_n\asymp \phi_c^{-1}(t/n)$.
By $\NL (\phi)$,
for all $z_i\in B(x_i, r_n)$ and $0\le i\le n+1$,
$$p(t/n, z_i,z_{i+1})\ge \frac{c_1}{V(z_i, \phi_c^{-1}(t/n))}.$$ Hence,
\begin{align*}&p(t,x,y)\\
&\ge \int_{B(z_1,r_n)}p(t/n,x,z_1)\,\mu(dz_1) \int_{B(z_2,r_n)}p(t/n,z_1,z_2)\,\mu(dz_2)\cdots \int_{B(z_{n-1},r_n)}p(t/n,z_{n-1},y)\,\mu(dz_{n-1})\\
&\ge \frac{c_2}{V(x,\phi_c^{-1}(t/n))} c_3^n\ge \frac{c_4}{V(x,\phi_c^{-1}(t))}n^{d_1/\beta_{2,\phi_c}} c_3^n\ge  \frac{c_5}{V(x,\phi_c^{-1}(t))} c_6^n,  \end{align*} where the constants $c_3,c_6\in (0,1)$, and in the third inequality we used $\VD$ and \eqref{polycon}. That is, we arrive at a lower bound of $p(t,x,y)$ with the form given in \eqref{eq:fibie3},
thanks to $\VD$ and \eqref{polycon} again.
This gives $\HK (\phi_c, \phi_j)$ and hence the first assertion of this theorem.
The last assertion follows from the first one,
Proposition \ref{ejlhk} and the fact that $\NDL(\phi)$ implies $\NL(\phi)$.
\qed\end{proof}

We need the following simple lemma
for
the proof of  Theorem \ref{T:main},
which holds without the connectedness and the chain condition on the state space $(M,d,\mu)$.

\begin{lemma}\label{near-} Under $\VD$ and \eqref{polycon}, $\UHK(\phi_c, \phi_j)$ and $\NL(\phi)$ together imply $\NDL(\phi)$. \end{lemma}
\begin{proof} We will use the ideas of the argument in \cite[Subsection 4.1]{BGK} and the proof of \cite[Lemma 3.2]{BBK2}. By carefully checking these proofs, to obtain the required assertion we only need to verify that for any $x\in M_0$, $t>0$ and $r>0$ with $r\asymp \phi^{-1}(t)$,
$$\sup_{0\le s\le t} \sup_{y\in B(x,r),z\in B(x,2r)^c} p(s,y,z)\le \frac{c_1}{V(x,\phi^{-1}(t))}.$$

According
to $\VD$, \eqref{polycon} and $\UHK(\phi_c, \phi_j)$,  for any $x\in M_0$, $t\ge1$ and $r>0$ with $r\asymp \phi^{-1}(t)=\phi_j^{-1}(t)$,
$$\sup_{0\le s\le t} \sup_{y\in B(x,r),z\in B(x,2r)^c} p(s,y,z)\le \sup_{y\in B(x,r)}\frac{c_2t}{V(y,r)\phi_j(r)}\le \frac{c_3}{V(x,r)}\le \frac{c_4}{V(x,\phi^{-1}(t))}.$$
On the other hand, also by $\VD$, \eqref{polycon} and $\UHK(\phi_c, \phi_j)$,  for any $x\in M_0$, $t\in (0,1]$ and $r>0$ with $r\asymp \phi^{-1}(t)
=\phi_c^{-1}(t)$,
\begin{align*}&\sup_{0\le s\le t}\sup_{y\in B(x,r),z\in B(x,2r)^c} p(s,y,z) \\
&\le  \sup_{y\in B(x,r)}\frac{c_5t}{V(y,r)\phi_j(r)}+\sup_{0\le s\le t} \sup_{y\in B(x,r)}\frac{c_5}{V(y,\phi^{-1}_c(s))}\exp\left(-c_6\frac{r}{\bar\phi_c(s/r)}\right)\\
&\le \frac{c_7t}{V(x,r)\phi_c(r)} +\frac{c_7}{V(x,\phi^{-1}_c(t))}\sup_{0\le s\le c_*\phi_c(r)} \left(\frac{r}{\phi^{-1}_c(s)}\right)^{d_2}\exp\left(-c_6\frac{r}{\bar\phi_c^{-1}(s/r)}\right)\\
&\le \frac{c_7t}{V(x,r)\phi_c(r)} +\frac{c_8}{V(x,\phi^{-1}_c(t))}\sup_{0\le s\le c_*\phi_c(r)} \left(\frac{\phi_c(r)}{s}\right)^{d_2/\beta_{1,\phi_c}}\exp\left(-c_9\left(\frac{\phi_c(r)}{s}\right)^{1/(\beta_{2,\phi_c}-1)}\right)\\
&\le \frac{c_7t}{V(x,r)\phi_c(r)} +\frac{c_{10}}{V(x,\phi^{-1}_c(t))}\sup_{0\le s\le c_*\phi_c(r)} \exp\left(-c_{11}\left(\frac{\phi_c(r)}{s}\right)^{1/(\beta_{2,\phi_c}-1)}\right)\\
&\le \frac{c_{12}}{V(x,\phi_c^{-1}(t))},
 \end{align*} where in the second inequality we
 used the fact that $\phi_c(r)\le c_0 \phi_j(r)$ for all $r\in (0,r_0]$ and some $r_0>0$, in the third inequality we applied \eqref{e:1.12} and \eqref{e:effdiff}, and the fourth inequality follows from the following elementary inequality:
 $$r^{d_2/\beta_{1,\phi_c}}\le c_{13} \exp\left(\frac{c_{9}}{2}r^{1/(\beta_{2,\phi_c}-1)}\right),\quad r\ge c_{14}>0.
$$
 Combining both conclusions above, we get the desired conclusion. \qed
\end{proof}

We are
now in a position to give the

\medskip

\noindent{\bf Proof of Theorem \ref{T:main}.}
It is obvious that ${\rm (i)}\Longrightarrow {\rm (ii)}$, thanks to Proposition \ref{l:jk}. ${\rm (ii)}\Longrightarrow {\rm (iii)}$ follows from Lemma \ref{near-}.
By Proposition \ref{P:4.1},
under $\VD$ and \eqref{polycon}, $\NDL(\phi)$ implies $\PI(\phi)$, and $\NDL(\phi)$ also implies $\E_\phi$ under the additional assumption $\RVD$.
With these at hand, we have $ {\rm (iii)}\Longrightarrow {\rm (iv)}$
by the ${\rm (ii)}\Longrightarrow {\rm (iii)}$ part of Theorem \ref{T:main-1} and Proposition \ref{T:gdCS}.
$ {\rm (iv)}\Longrightarrow {\rm (v)}$ has been proved in Proposition \ref{P:csp}.
Furthermore, by Proposition
\ref{ejlhk2},
$ {\rm (v)}\Longrightarrow{\rm (i)}$.
When the space $(M,d,\mu)$ is connected and satisfies the chain condition,
$ {\rm (v)}\Longrightarrow {\rm (vi)}$ has been proven in Proposition \ref{P:twosided}. Clearly,
$ {\rm (vi)} \Longrightarrow{\rm (i)}$. This completes the proof of the theorem. \qed

\section{Characterizations of parabolic Harnack inequalities}\label{S:6}

The goal of this section is to present three different characterizations of parabolic Harnack inequalities, see Theorems \ref{Thm: prop1}, \ref{Thm: prop2} and \ref{Thm:anal}.

By the arguments in \cite[Section 3.1]{CKW2}, we have the
following consequences of $\PHI(\phi)$. Note that, though
\cite{CKW2} is concerned with
pure jump non-local Dirichlet forms,  the
arguments in \cite[Section 3.1]{CKW2} works for
general symmetric Dirichlet forms $(\sE, \sF)$ under  the present setting.

 \begin{proposition}\label{uodest} Assume that $\VD$, \eqref{polycon} and $\PHI(\phi)$ hold. Then
  $\UHKD(\phi)$ and $\NDL(\phi)$, as well as $\UJS$, hold true.
 Consequently, $\PI(\phi)$ and $\E_{\phi,\ge}$ hold, and $X:=\{X_t\}_{t\ge0}$ is conservative. If furthermore $\RVD$ is satisfied, then we also have $\FK(\phi)$ and $\E_{\phi,\le}$ $($and so $\E_{\phi}$$)$.
 \end{proposition}

  \begin{proof}
	See \cite[Propositions 3.1, 3.2, 3.5 and 3.6]{CKW2}
	for the proof.
	\qed  \end{proof}

 We point out that $\PHI(\phi)$ alone can not guarantee $\J_{\phi_j,\le}$.
 Indeed, following the argument of \cite[Corollary 3.4]{CKW2}, under $\VD$, \eqref{polycon}, $\UJS$ and $\NDL(\phi)$, since $\phi(r)\le \phi_j(r)$ for all $r>0$,
 we can only obtain $\J_{\phi,\le}$, which is weaker than $\J_{\phi_j,\le}$.
 See Example \ref{PHIexm} for a concrete counterexample.
 In spite of this, we still have the following statement.

\begin{proposition}\label{ndl-holder}
Assume that $\VD$, \eqref{polycon}, $\NDL(\phi)$, $\E_{\phi,\le}$ and $\J_{\phi,\le}$ hold. For every $\delta \in (0, 1)$,
there exist positive constants
$C>0$ and $\gamma\in(0,1]$, where $\gamma$ is independent of $\delta$,
 so that  for any bounded caloric function $u$ in
$Q(t_0,x_0,\phi(r),r)$,
there is a properly exceptional set ${\cal N}_u\supset {\cal N}$ such that
$$|u(s,x)-u(t, y)|\le C\left( \frac{\phi^{-1}(|s-t|)+d(x, y)}{r} \right)^\gamma  \esssup_{ [t_0, t_0+\phi (r)] \times M}\,|u|
$$
for  every $s, t\in (t_0+\phi(r)-\phi(\delta r), t_0+\phi (r))$ and $x, y \in B(x_0, \delta r)\setminus {\cal N}_u$.
 In other words, under $\VD$ and \eqref{polycon}, $\NDL(\phi)+\E_{\phi,\le}+\J_{\phi,\le}$ imply {\rm PHR}$(\phi)$ and {\rm EHR}. In particular, $\PHI(\phi)$ implies {\rm PHR}$(\phi)$ and {\rm EHR}.
 \end{proposition}

\begin{proof} The proof is the same as that of \cite[Proposition 3.8]{CKW2}, so it is omitted.\qed
\end{proof}

We next  present  characterizations of $\PHI(\phi)$.
Recall that in Remark \ref{Heat:remark} upper bound estimate \eqref{heat:remark1} of the heat kernel $p(t,x,y)$ is named by   $\UHK_{weak}(\phi)$.

\begin{theorem}\label{Thm: prop1}
Assume that $\mu$ and $\phi$ satisfy $\VD$, $\RVD$ and \eqref{polycon} respectively. Then the following  holds
\begin{align*}
\PHI(\phi) \Longleftrightarrow \UHK_{weak}(\phi) + \NDL(\phi) +\UJS.
\end{align*}
\end{theorem}
\begin{proof} According to Proposition \ref{uodest}, we have the assertion that $$\PHI(\phi)\Longrightarrow \UHKD(\phi)+\NDL(\phi)+\UJS +\E_\phi+  \,\J_{\phi,\le},$$ where we note that $\RVD$ is only used to prove $\PHI(\phi)\Longrightarrow \E_{\phi,\le}$. Then by Remark \ref{Heat:remark}, we get
$$\PHI(\phi)\Longrightarrow\UHK_{weak}(\phi) + \NDL(\phi) +\UJS.$$

For $\UHK_{weak}(\phi) + \NDL(\phi) +\UJS\Longrightarrow\PHI(\phi)$,
we can follow most of the arguments
in \cite[Subsection 4.1]{CKW2}.
One different point is that
in the proof of \cite[Lemma 4.1]{CKW2} (see the arXiv version of the paper
\cite{CKW2}),
we need to verify that under $\VD$ and \eqref{polycon}, $\UHK_{weak}(\phi)$ implies
the following:
for any $x,y\in M_0$ and $t>0$ with $d(x,y)\asymp \phi^{-1}(t)$,
$$p(t,x,y)\le \frac{c_1}{V(x,\phi^{-1}(t))}.$$
Yet,
this inequality is a direct consequence of \eqref{heat:remark1}. The proof is complete.  \qed \end{proof}

The next  characterization of $\PHI(\phi)$ involves the property of exit times $\E_\phi$.

\begin{theorem}\label{Thm: prop2} Assume that $\mu$ and $\phi$ satisfy $\VD$, $\RVD$ and \eqref{polycon} respectively. Then the following  hold
\begin{align*}
\PHI(\phi)&\Longleftrightarrow \PHR(\phi)+\E_\phi+ \UJS +\J_{\phi,\le}\\
&\Longleftrightarrow  \EHR +\E_\phi +\UJS+\J_{\phi,\le }.
\end{align*}
\end{theorem}

\begin{proof} According to Propositions \ref{ndl-holder} and \ref{uodest}, we have
$$
\PHI(\phi)\Longrightarrow \PHR(\phi)+\E_\phi+ \UJS +\J_{\phi,\le}
 \Longrightarrow \EHR +\E_\phi +\UJS+\J_{\phi,\le }.
$$
For $\EHR +\E_\phi +\UJS+\J_{\phi,\le }\Longrightarrow \PHI(\phi)$, we mainly follow the arguments in \cite[Subsection 4.2]{CKW2}. In particular, the proof of \cite[Proposition 4.9]{CKW2} yields that $\EHR$ together with $\E_\phi$ imply $\NDL(\phi)$. Then according to the argument of \cite[Proposition 3.5]{CKW2}, $\FK(\phi)$ holds,
thanks to $\RVD$.
By Proposition \ref{ehi-ukdd},
 under $\FK(\phi)$, $\E_\phi$ and $\J_{\phi,\le}$, $\UHKD(\phi)$ holds true. This along with Remark \ref{Heat:remark} yields that $\UHK_{weak}(\phi)$ holds.
Combining all these with Theorem \ref{Thm: prop1}, we have
$$\EHR +\E_\phi +\UJS+\J_{\phi,\le }\Longrightarrow \PHI(\phi).$$ The proof is complete. \qed  \end{proof}

Finally, we turn to
the stable
analytic characterization of $\PHI(\phi)$.

\begin{theorem}\label{Thm:anal} Assume that $\mu$ and $\phi$ satisfy $\VD$, $\RVD$ and \eqref{polycon} respectively. Then the following  hold
\begin{align*}
\PHI(\phi)&\Longleftrightarrow \PI(\phi)+\J_{\phi,\le }+\Gcap(\phi)+\UJS\\
 &\Longleftrightarrow \PI(\phi)+\J_{\phi,\le }+\CS(\phi)+\UJS.
\end{align*}
\end{theorem}

\begin{proof}
By Proposition \ref{P:exit}, Theorem \ref{Thm: prop2} and Proposition
\ref{P:EHR},
$$
\PI(\phi)+\J_{\phi,\le }+\CS(\phi)+\UJS\Longrightarrow \PHI(\phi),
$$
where we used Proposition \ref{P:3.1} that $\PI(\phi)$ implies $\FK(\phi)$ under the additional assumption $\RVD$.

It follows from Proposition \ref{P:csp} that
$$
\PI(\phi)+\J_{\phi,\le }+\Gcap(\phi)+\UJS\Longrightarrow \PI(\phi)+\J_{\phi,\le }+\CS(\phi)+\UJS.
$$
Finally, according to Proposition \ref{uodest} and Theorem \ref{Thm: prop1}, we have $$  \PHI(\phi)\Longrightarrow \NDL(\phi)+\UJS +\E_\phi+  \,\J_{\phi,\le}+\UHK_{weak}(\phi).$$
This in particular implies that the Dirichlet form $(\sE, \sF)$ is conservative by Proposition \ref{P:3.2}(ii).
Furthermore, by the argument of \cite[Proposition 3.5]{CKW2}, $\NDL(\phi)\Longrightarrow \PI(\phi)$.
We note that, from the proof of Lemma \ref{Conserv}, we can see that $\UHK_{weak}(\phi)$ along with the fact $(\sE,\sF)$ is conservative implies $\EP_{\phi,\le}$,
which in turn gives us $\Gcap(\phi)$
by Proposition \ref{P:2.4}.
 Thus, we prove that
$$
\PHI(\phi)\Longrightarrow \PI(\phi)+\J_{\phi,\le }+\Gcap(\phi)+\UJS.
$$
The proof is complete.
\qed
\end{proof}

\noindent{\bf Proof of Theorem  \ref{C:1.25}.}
As noted earlier, the equivalence between {\rm (i)} and {\rm (ii)} of Theorem \ref{C:1.25} has been  proved in Theorem \ref{Thm: prop1},   while
the equivalence between
{\rm (i)},  {\rm (iii)} and  {\rm (iv)}  has been established in Theorem \ref{Thm: prop2}.
The equivalence between {\rm (i)},  {\rm (v)} and  {\rm (vi)}  follows from Theorem \ref{Thm: prop1},  Theorem \ref{Thm: prop2}
and  Theorem  \ref{Thm:anal}.  The last assertion of Theorem \ref{C:1.25} follows from the equivalence between {\rm (i)} and  {\rm (v)} of Theorem \ref{C:1.25}
and from   Theorem \ref{T:main}.
\qed

\section{Examples/Applications}\label{Sect7}

In this section, we give some examples/applications of our results.

\begin{example}\label{PHIexm} {\bf ($\PHI(\phi)$ alone does not imply $\J_{\phi_j,\le}$)}\,\,
\rm Let $M=\bR^d$, and
$$
J(x,y)\asymp \begin{cases} \frac{1}{|x-y|^{d+\alpha}}&\quad |x-y|\le 1; \\
 \frac{1}{|x-y|^{d+\beta}} &\quad  |x-y|\ge1,\end{cases}
 $$
 where $\alpha,\beta\in (0,2)$. We consider the following regular Dirichlet form
 $$
 \sE(f,g)=\int_{\bR^d}\nabla f(x)\cdot A(x) \nabla g(x)\,dx+\iint_{\bR^d\times \bR^d} (f(x)-f(y))(g(x)-g(y))J(x,y)\,dx\,dy
 $$
 and
 $\sF=\overline{C_1(\bR^d)}^{\sE_1}$, where $A(x)$ is a measurable $d\times d$ matrix-valued function
on $\bR^d$ that is
uniformly elliptic and bounded. It has been proven
in \cite[Theorem 1.4]{CK3} that $\HK (\phi_c, \phi_j)$ holds with
$\phi_c(r)=r^2$ and $\phi_j(r)=r^\alpha {\bf 1}_{\{r\le 1\}}+
r^\beta {\bf 1}_{\{r\ge 1\}}$. Hence by \eqref{eq:5-6Fuz},
$\PHI(\phi)$ holds with
$$
\phi(r)=\phi_c(r)\wedge \phi_i(r)=r^\beta\wedge r^2.
$$
Since $\PHI(\phi)$ holds regardless of the choice of $\alpha\in (0,2)$,
$\PHI(\phi)$ alone  can not imply the upper bound of the jumping kernel.
\end{example}

\noindent {\bf Example \ref{Lip-dom}} {\bf(continued)}~
Here we provide proof of Example \ref{Lip-dom}.

First, consider a reflected Brownian motion $\{B_t\}_{t\ge0}$ on $U$. It is known that its heat kernel enjoys Aronson-type Gaussian bounds, see for instance
\cite[Theorem 2.31]{GSC}. (In fact, as discussed in \cite[Theorem 2.31]{GSC}, similar results hold
for reflected Brownian motions on inner uniform domains in Harnack-type Dirichlet spaces, so the results in this example can be extended to that framework.)
In particular, according to \cite[Theorem 2.31]{GSC}, we know that the following Poincar\'e inequality for the strongly local part of $(\sE,W^{1,2}(U))$ given by \eqref{eq:DFw}, i.e.,
there exist constants $C_1>0 $ and $\kappa_1\ge1$ such that
for any  ball
$B_r:=B(x,r)$
with $x\in U$ and $r>0$
and for any $f \in W^{1,2}(U)\cap B_b(U)$,
$$
\int_{B_r} (f-\ol f_{B_r})^2\, d\mu \le C_1 r^2\left(\int_{B_{\kappa_1 r}} \nabla f(z)\cdot A(z) \nabla
f(z)\,dz\right).
$$
On the other hand, for $(\sE,W^{1,2}(U))$ given by \eqref{eq:DFw}, it is obvious that condition $\J_{\phi_j}$ holds with $\phi_j(r)=r^\alpha$, which in turn yields that there exist constants $C_2>0 $ and $\kappa_2\ge1$ such that
for any  ball $B_r=B(x,r)$ with $x\in U$ and $r>0$  and for any $f \in W^{1,2}(U)\cap B_b(U)$,
$$
\int_{B_r} (f-\ol f_{B_r})^2\, d\mu \le C_2 r^\alpha\left(\int_{B_{\kappa_2 r}}\int_{B_{\kappa_2 r}}
\frac{(f(y)-f(z))^2}{d(y,z)^{d+\alpha}}c(y,z)\,dy\,dz\right).
$$
Hence,
we have, for any  ball $B_r=B(x,r)$
with $x\in U$ and $r>0$  and for any $f \in W^{1,2}(U)\cap B_b(U)$,
\begin{align*}
&\int_{B_r} (f-\ol f_{B_r})^2\, d\mu \\
&\le C_0 (r^2\wedge r^\alpha)\left( \int_{B_{\kappa_0 r}} \nabla f(z)\cdot A(z) \nabla
f(z)\,dz+ \int_{B_{\kappa_0 r}}\int_{B_{\kappa_0 r}}
\frac{(f(y)-f(z))^2}{d(y,z)^{d+\alpha}}c(y,z)\,dy\,dz\right)
\end{align*} with $C_0=C_1\vee C_2$ and $\kappa_0=\kappa_1\vee \kappa_2$. That is, $\PI(\phi)$ holds with $\phi(r)=r^2\wedge r^\alpha$.
Note further that, in the present setting
$\CS(\phi)$
holds trivially,
as mentioned in Remark \ref{rek:scj}\,(iii).
Therefore, it immediately
follows from our stability theorem (Theorem \ref{T:main}) that the
heat kernel for the Dirichlet form \eqref{eq:DFw} enjoys the
estimates $\HK (\phi_c, \phi_j)$, hence $\PHI (\phi)$ as well.

An alternative way to prove Example \ref{Lip-dom} is
to use subordination of Brownian motion and use
the transferring method
to be discussed below.

\medskip

The stability results in Theorems \ref{T:main}, \ref{T:main-1} and \ref{C:1.25}  allow us to obtain heat kernel estimates and parabolic Harnack inequalities for a large class of symmetric diffusions with
jumps using ``transferring method"; that is, by first establishing heat kernel estimates and parabolic Harnack inequalities for
a particular diffusion with jumps with
nice
jumping kernel $J(x, y)$, we   then use
Theorems \ref{T:main}, \ref{T:main-1} and \ref{C:1.25} to obtain heat kernel estimates and parabolic Harnack inequalities for other symmetric jump processes
whose strongly local parts along the diagonal are comparable to that of the original process and
whose jumping kernels
are comparable to $J(x, y)$.
Examples for the pure jump case have been
given in \cite[Section 6.1]{CKW1} and \cite[Section 5]{CKW2}.

In the following,
we illustrate this method
for symmetric Dirichlet forms \eqref{e:1.1}
on a $d$-sets $M$ on which there exists a diffusion  whose heat kernel enjoys (sub-)Gaussian estimates
as in \eqref{e:1.7}.

\begin{example}\label{Exmne1}{\bf (Diffusion with jumps on $d$-set.)} \quad \rm
Let  $(M, d , \mu)$ be an Alfhors $d$-regular set.
Suppose that there is a $\mu$-symmetric diffusion $\{Z_t; t\ge0; \bP_x, x\in M\}$ on $M$
such that
it has a transition density function $q(t, x, y)$  with respect to the measure $\mu$ that has the following two-sided estimates:
\begin{equation}\label{e:1.7}
q(t, x, y) \asymp t^{-d/\beta} \exp \left( - \left(\frac{  d(x, y)^\beta }{ t} \right)^{1/(\beta -1)}  \right),
\quad t>0, x, y\in M
\end{equation}
 for some $\beta \geq 2$. Denote by $(\bar\sE,\bar\sF)$ the corresponding Dirichlet form.
   A prototype
  is a Brownian motion on the
   $D$-dimensional    unbounded Sierpi\'{n}ski gasket; see for instance \cite{BP}. In this case,
$d=\log (D+1)/\log 2$ is the Hausdorff dimension of
the gasket, and $\beta=\log (D+3)/\log 2$ is called the walk dimension in \eqref{e:1.7}.

Take any $\alpha \in (0, \beta)$, and a symmetric strongly local regular Dirichlet form $(\sE^{(c)}, \bar \sF)$
in $L^2(M; \mu)$ with the property that $\sE^{(c)} (f,f)\asymp \bar\sE (f,f)$ for all
$f\in \bar\sF$. Consider the following regular Dirichlet form  $(\sE,\bar\sF)$ in $L^2(M; \mu)$ defined by
 \begin{equation}\label{eq:DFdsetw}
\sE (u,v)=
\sE^{(c)} (u,v)
+ \int_{M}\int_{M} (u(x)-u(y))(v(x)-v(y))
\frac{c(x,y)}{d(x,y)^{d+\alpha}}\,\mu(dx)\,\mu(dy),
\end{equation}
where
$c(\cdot,\cdot)$ is a symmetric measurable
function on $M\times M$
that is bounded between two positive constants.
Define $\phi_c (r)= r^\beta$, $\phi_j (r)=r^\alpha$ and $\phi (r)=\phi_c (r) \wedge \phi_j (r)= r^\beta\wedge r^\alpha$.
We claim that $(\sE,\bar\sF)$ enjoys
 $\HK (\phi_c, \phi_j)$
and $\PHI (\phi)$.

Below we prove this using subordination and the transferring method.
First, let $\{\xi_t\}_{t\ge0}$ be a $\gamma$-stable subordinator with non-zero drift,
independent of $\{Z_t\}_{t\ge0}$, such that its Laplace exponent $\bar \phi$ is given by \begin{equation}\label{eq:suborg}
\bE[\exp (-\lambda \xi_t)]=\exp
(-t\bar\phi(\lambda)),\quad \lambda, t>0
\end{equation}
with
$\bar\phi(\lambda)=\lambda+\lambda^{\gamma}$ for $\gamma:=\alpha/\beta\in(0,1)$.
The process $\{X_t\}_{t\ge0}$ defined by $X_t=Z_{\xi_t}$ for any $t\ge0$ is called a $\gamma$-stable subordinated process with drift.
Let $\{\eta_t(u):t>0,u\ge0\}$ be the transition density of $\{\xi_t\}_{t\ge0}$, and $\nu(z)$ be the L\'evy density of $\{\xi_t\}_{t\ge0}$. Let $p(t,x,y)$ be the heat kernel of $\{X_t\}_{t\ge0}$, and $J(x,y)$ be the jumping density of $\{X_t\}_{t\ge0}$. Then, it is known that
\begin{align*}
p(t,x,y)&=\int_0^\infty q(u,x,y)\eta_t(u)\,du,\quad t>0,x,y\in M,\\J(x,y)&=\int_0^\infty q(u,x,y)\nu(u)\,du,\quad x,y\in M.
\end{align*}

By \cite[Section 6.1]{CKW1}, we know that
$$J(x,y)\simeq \frac{1}{d(x,y)^{d+\beta\gamma}}
=\frac{1}{d(x,y)^{d+\alpha}},\quad x,y\in M.$$ Namely, $\J_{\phi_j}$ holds with $\phi_j(r)=r^{\alpha}$.

On the other hand, by the definition of $\{\xi_t\}_{t\ge0}$, $\eta_t(u)=\eta_t^{(\gamma)}(u-t)$, where $\eta_t^{(\gamma)}(u)$ corresponds to the transition density of the standard $\gamma$-stable subordinator (without drift). Then,
\begin{align*}p(t,x,y)&=\int_0^\infty q(u,x,y)\eta^{(\gamma)}_t(u-t)\,du=\int_t^\infty q(u,x,y)\eta_t^{(\gamma)}(u-t)\,du\\
&=\int_0^\infty q(u+t,x,y)\eta_t^{(\gamma)}(u)\,du=\int_M q(t,x,z) \int_0^\infty q(u,z,y)\eta^{(\gamma)}_t(u)\,du\,\mu(dz)\\
&=:\int_M q(t,x,z) q^{(\gamma)}(t,z,y)\,\mu(dz),\end{align*} where $\{q^{(\gamma)}(t,x,y): t>0,x,y\in M\}$ is the heat kernel corresponding to the standard $\gamma$-stable subordination of the process $Z$. In particular, according to \cite[Section 6.1]{CKW1},
$$q^{(\gamma)}(t,x,y)\simeq t^{-d/\alpha}\wedge\frac{t}{d(x,y)^{d+\alpha}},\quad t>0,x,y\in M.$$
Furthermore, by standard calculations (see the proof of \cite[Theorem 2.13]{SV} for the case that $\beta=2$, $\gamma=\alpha/2$ and $M=\bR^d$), one can check that $p(t,x,y)$ enjoys the form of \eqref{HKjum} with $V(x,r)\simeq r^d$, $\phi_c(r)=r^\beta$ and $\phi_j(r)=r^{\alpha}$. This is,
the heat kernel for $\{X_t\}_{t\ge0}$ enjoys $\HK (\phi_c, \phi_j)$ (hence $\PHI (\phi)$ as well) with $\phi(r)=r^\beta\wedge r^\alpha$.

 Let $Y=\{Y_t, t\geq 0; \bP^Y_x, x\in M\}$ be the Hunt process associated with the regular Dirichlet form
 $(\sE, \bar \sF)$ in $L^2(M; \mu)$ given by \eqref{eq:DFdsetw}.
 Clearly, $\J_{\phi_j}$ holds and
 so does $\PI(\phi)$ for $(\sE, \bar \sF)$
 by the same argument as in Example \ref{Lip-dom}.
Note that
if two local Dirichlet forms are comparable, then their
energy measures are also comparable
(see for instance \cite[Section 3.2]{FOT}).
Hence we see that $\CS(\phi)$
holds for the process $Y$ because it holds for $\{X_t\}_{t\ge0}$.
Now thanks to Theorem \ref{T:main}, we obtain
$\HK (\phi_c, \phi_j)$,  and consequently $\PHI (\phi)$
 for the Hunt process $Y$, or equivalently,  for the Dirichlet form $(\sE, \bar \sF)$ in $L^2(M; \mu)$.
 \end{example}

It is desirable to prove Example \ref{Exmne1} directly
(i.e. without using the subordination) from
the stability results of the diffusion and the jump process as we did in Example \ref{Lip-dom}. However, it is highly non-trivial to verify
$\CS(\phi)$
when $\beta>2$. Indeed, in that case the
cut-off
functions for the diffusion and the jump process may be different and we cannot simply sum up two forms.

 \medskip

We end this section a remark on two possible extensions of Example \ref{Exmne1}.

\begin{remark}\rm  (i) We can start from more general symmetric diffusions on general measure metric spaces. For example, let $(M,d,\mu)$ be a metric measure space as in the setting of this paper that is connected and also satisfies $\VD$ and the chain condition. Assume that there is a $\mu$-symmetric conservative diffusion process $\{Z_t\}$
whose heat kernel enjoys \eqref{eq:fibie2}.
This includes
symmetric diffusions on  certain fractal-like manifolds; see \cite[Section 6.1]{CKW1}.

(ii) We can also consider more general subordinator. For instance, let $\{\xi_t\}_{t\ge0}$ be a  subordinator with non-zero drift such that its Laplace exponent $\bar \phi$ defined by \eqref{eq:suborg} has the following form
$$
\bar \phi(\lambda)=b\lambda+\phi_0(\lambda),
$$
where
$b>0$, and $\phi_0(r)$ satisfies \eqref{polycon000} with
$\beta_{1,\phi_0},\beta_{2,\phi_0}\in (0,1)$, and the associated L\'evy measure $\nu(dz)$ of $\phi_0(r)$ has a density function $\nu(z)$ with respect to the Lebesgue measure such that
the function $ t \mapsto t\nu(t)$ is  non-increasing
  on $(0, \infty).$
Under these assumptions, two-sided estimates for the transition density of the subordinator corresponding to $\phi_0(r)$ recently have been obtained in \cite[Theorem 4.4]{CKKW}. Thus, with aid of  \cite[Theorem 4.4]{CKKW}, the argument of Example \ref{Exmne1} could be workable for this larger class of subordinators.
\end{remark}

\bigskip

\noindent \textbf{Acknowledgements.}
The research of Zhen-Qing Chen is partially supported by Simons Foundation Grant 520542, a
 Victor Klee Faculty Fellowship at UW, and NNSFC grant   11731009.
 \ The research of Takashi Kumagai is supported
by JSPS KAKENHI Grant Number JP17H01093 and by the Alexander von Humboldt Foundation.\
 The research of Jian Wang is supported by the National
Natural Science Foundation of China (No.\ 11831014), the Program for Probability and Statistics: Theory and Application (No.\ IRTL1704) and the Program for Innovative Research Team in Science and Technology in Fujian Province University (IRTSTFJ).

 \vskip 0.3truein

\noindent {\bf Zhen-Qing Chen}

\smallskip \noindent
Department of Mathematics, University of Washington, Seattle,
WA 98195, USA

\noindent and School of Mathematics and Statistics, Beijing Institute of Technology, China

\noindent
E-mail: \texttt{zqchen@uw.edu}

\medskip

\noindent {\bf Takashi Kumagai}:

\smallskip \noindent
 Research Institute for Mathematical Sciences,
Kyoto University, Kyoto 606-8502, Japan

\noindent Email: \texttt{kumagai@kurims.kyoto-u.ac.jp}

\medskip

\noindent {\bf Jian Wang}:

\smallskip \noindent
College of Mathematics and Informatics \& Fujian Key Laboratory of Mathematical Analysis and Applications (FJKLMAA), Fujian Normal
University, 350007, Fuzhou, P.R. China.

\noindent Email: \texttt{jianwang@fjnu.edu.cn}


{\small
\begin{thebibliography}{}
\bibitem[AB]{AB} S. Andres and M.T. Barlow.
Energy inequalities for cutoff-functions and some applications.
{\em J. Reine Angew. Math.} {\bf 699} (2015), 183--215.

\bibitem[B]{B}
M.T. Barlow. Diffusions on fractals.
In: {\sl Lectures on Probability Theory and Statistics,
Ecole d'\'Ete de Probabilit\'es de Saint-Flour XXV - 1995}, 1--121.
Lect. Notes Math. {\bf 1690}, Springer 1998.

\bibitem[BB1]{BB1}
M.T. Barlow and R.F. Bass.
Brownian motion and harmonic analysis on Sierpi\'{n}ski carpets.
{\it Canad. J. Math.} {\bf 51} (1999), 673--744.

\bibitem[BB2]{BB2}
M.T. Barlow and R.F. Bass. Stability of parabolic Harnack inequalities.
{\it Trans. Amer. Math. Soc.} {\bf 356} (2003), 1501--1533.

\bibitem[BBCK]{BBCK}  M.T. Barlow, R.F. Bass, Z.-Q. Chen and M. Kassmann.
Non-local Dirichlet forms and symmetric jump processes.
{\it Trans. Amer. Math. Soc.} {\bf 361} (2009), 1963--1999.

\bibitem[BBK1]{BBK1}
 M.T. Barlow, R.F. Bass and T. Kumagai.  Stability of parabolic
Harnack inequalities on metric measure spaces.
{\it J. Math. Soc. Japan} {\bf 58} (2006), 485--519.


\bibitem[BBK2]{BBK2}  M.T. Barlow, R. F. Bass and  T. Kumagai.
Parabolic Harnack inequality and heat kernel estimates for random walks
with long range jumps. {\it Math. Z.} {\bf 261} (2009), 297--320.


\bibitem[BGK]{BGK} M.T. Barlow, A. Grigor'yan and  T. Kumagai.
On the equivalence of parabolic Harnack inequalities and heat kernel estimates.
{\it J. Math. Soc. Japan \bf 64} (2012), 1091--1146.


\bibitem[BP]{BP}
M.T. Barlow and E.A. Perkins.
Brownian motion on the Sierpi\'{n}ski gasket.
{\it Probab. Theory Relat. Fields \bf 79} (1988), 543--623.

\bibitem[BL]{BL}
R.F. Bass and D. A. Levin.
Transition probabilities for symmetric jump processes. {\it Trans.
Amer. Math. Soc.} {\bf 354} (2002), 2933--2953.

\bibitem[BKKL]{BKKL}
J. Bae, J. Kang, P. Kim and J. Lee.
Heat kernel estimates for symmetric jump processes with mixed polynomial growths. To appear in {\it Ann. Probab.}, available at {\tt arXiv:1804.06918}.

\bibitem[BGT]{BGT}
N.H. Bingham, C.M. Goldie and J.L. Teugels. {\sl Regular Variation}. Cambridge University Press,
Cambridge, 1987.

\bibitem[BM]{BM} M. Biroli and U. A. Mosco. Saint-Venant type principle for Dirichlet forms on discontinuous
media. {\it  Ann. Mat. Pura Appl. \bf 169} (1995), 125--181.



\bibitem[CKS]{CKS} E.A. Carlen, S. Kusuoka and D.W. Stroock.
Upper bounds for symmetric Markov transition functions.
{\em Ann. Inst. Henri Poincar\'{e}-Probab. Stat.}
\textbf{23} (1987), 245--287.


\bibitem[C]{Chen}Z.-Q. Chen.
On notions of harmonicity.
{\it Proc. Amer. Math. Soc. \bf 137} (2009), 3497--3510.


\bibitem[CF]{CF} Z.-Q. Chen and M. Fukushima.
{\sl Symmetric Markov Processes, Time Change, and Boundary Theory}.
Princeton Univ. Press, Princeton and Oxford, 2012.


\bibitem[CKK]{CKK}  Z.-Q. Chen, P. Kim and  T. Kumagai.
Weighted Poincar\'e inequality and heat kernel estimates for finite range jump processes.
{\it Math. Ann. \bf 342} (2008), 833--883.

\bibitem[CKK2]{CKK2}  Z.-Q. Chen, P. Kim and  T. Kumagai.
On heat kernel estimates and parabolic Harnack inequality for jump processes on metric measure spaces.
{\it Acta Math. Sin. (Engl. Ser.)} {\bf 25}
(2009), 1067--1086.


\bibitem[CKKW1]{CKKW}  Z.-Q. Chen, P. Kim,  T. Kumagai and J. Wang.
Time fractional poisson equations: representations and estimates.
Available at {\tt arXiv:1812.04902}.


\bibitem[CKKW2]{CKKW2}  Z.-Q. Chen, P. Kim,  T. Kumagai and J. Wang.
Heat kernel estimates for
reflected diffusions with jumps on metric measure spaces.
In preparation.


\bibitem[CK1]{CK1}  Z.-Q. Chen and  T. Kumagai.
Heat kernel estimates for stable-like processes on $d$-sets. {\it
Stochastic Process Appl.} {\bf 108} (2003), 27--62.

\bibitem[CK2]{CK2}  Z.-Q. Chen and  T. Kumagai.
Heat kernel estimates for jump processes of mixed types on metric
measure spaces. {\it Probab. Theory Relat. Fields} {\bf 140} (2008), 277--317.

\bibitem[CK3]{CK3}  Z.-Q. Chen and T. Kumagai.
A priori H\"older estimate, parabolic Harnack principle and heat kernel
estimates for diffusions with jumps.
  {\it Revista Mat. Iberoamericana \bf 26} (2010), 551--589.

\bibitem[CKW1]{CKW1}  Z.-Q. Chen, T. Kumagai and J. Wang.
Stability of heat kernel estimates for symmetric non-local Dirichlet forms. To appear in {\it Memoirs Amer. Math. Soc.},
available at {\tt arXiv:1604.04035}.

\bibitem[CKW2]{CKW2}  Z.-Q. Chen, T. Kumagai and J. Wang.
Stability of parabolic Harnack inequalities for symmetric non-local Dirichlet forms. To appear in {\it J. European Math. Soc.}, available at {\tt arXiv:1609.07594}.


\bibitem[CKW3]{CKW3}  Z.-Q. Chen, T. Kumagai and J. Wang.
Elliptic Harnack inequalities for symmetric non-local Dirichlet forms.
{\it J. Math. Pures Appl. \bf 125} (2019), 1--42.


\bibitem[ChK]{ChK} Z.-Q. Chen and K. Kuwae.
On subhamonicity for symmetric Markov processes.
{\it J. Math. Soc. Japan \bf 64} (2012), 1181--1209.

\bibitem[De]{De}  T. Delmotte. Parabolic Harnack inequality and estimates of Markov chains on graphs.
{\it   Revista Mat. Iberoamericana \bf 15} (1999), 181--232.

\bibitem[FOT]{FOT}
M. Fukushima, Y. Oshima and M. Takeda.
{\sl Dirichlet Forms and Symmetric Markov Processes}.
de Gruyter, Berlin, 2nd rev. and ext. ed., 2011.


\bibitem[Gr1]{Gr1} A. Grigor'yan. The heat equation on noncompact Riemannian
manifolds.
(in Russian) {\it Matem. Sbornik. \bf 182} (1991), 55--87.
(English transl.) {\it Math. USSR Sbornik \bf 72} (1992), 47--77.


\bibitem[Gr2]{gribk}
A. Grigor'yan.
{\sl Heat Kernel and Analysis on Manifolds.}
Amer. Math. Soc., Providence, RI, International Press, Boston, MA, 2009.

\bibitem[GH]{GH} A. Grigor'yan and  J. Hu.
Upper bounds of heat kernels on doubling spaces.
{\it Mosco Math. J. \bf 14} (2014), 505--563.

\bibitem[GHH]{GHH}  A. Grigor'yan, E. Hu  and  J. Hu.
Two-sided estimates of heat kernels of jump type Dirichlet forms. {\it Adv. Math.} {\bf 330}  (2018), 433--515.



\bibitem[GHL]{GHL3}
A. Grigor'yan, J. Hu and K.-S. Lau.
Generalized capacity, Harnack inequality and heat kernels on metric spaces.
{\it J. Math. Soc. Japan \bf 67} (2015), 1485--1549.

\bibitem[GT]{GT} A. Grigor'yan and A. Telcs. Two-sided estimates of heat kernels
on metric measure spaces. {\it Ann. Probab. \bf 40} (2012), 1212--1284.

\bibitem[GSC]{GSC}
P. Gyrya and L. Saloff-Coste.
Neumann and Dirichlet heat kernels in inner uniform domains.
{\it Ast\'erisque \bf 336}, 2011.

\bibitem[HK]{HK}
B.M. Hambly and T. Kumagai.
Transition density estimates for diffusion processes on
 post critically finite
self-similar fractals.
{\it Proc. London Math. Soc. \bf 78} (1999), 431--458.





\bibitem[MS]{MS}
M. Murugan and L. Saloff-Coste.
Heat kernel estimates for anomalous heavy-tailed random walks.
To appear in {\it Ann. Inst. Henri Poincar\'{e}-Probab. Stat.}, available at {\tt arXiv:1512.02361}.

\bibitem[Sa1]{Sa1}
L. Saloff-Coste. A note on Poincar\'{e}, Sobolev, and Harnack inequalities.
{\it Inter. Math. Res. Notices \bf 2} (1992), 27--38.


\bibitem[Sa2]{Sa2}
L. Saloff-Coste.
{\sl Aspects of Sobolev-type Inequalities}.
Cambridge Univ. Press, Cambridge, 2002.




\bibitem[SV]{SV}
R. Song and Z. Vondracek. Parabolic Harnack inequality for the
mixture of Brownian motion and stable process. {\it Tohoku Math. J. \bf 59}
(2007), 1--19.

\bibitem[St1]{St1} K.-T. Sturm. Analysis on local Dirichlet spaces II. Gaussian upper bounds for the fundamental
solutions of parabolic Harnack equations. {\it Osaka J. Math. \bf 32} (1995), 275--312.

\bibitem[St2]{St2}  K.-T. Sturm. Analysis on local Dirichlet spaces III. The parabolic Harnack inequality.
{\it J. Math. Pures Appl. \bf 75} (1996), 273--297.

\end{thebibliography}

}
\end{document}